\documentclass{amsart}

\usepackage{amsfonts,amssymb,verbatim,amsmath,amsthm,latexsym,textcomp,amscd}
\usepackage{latexsym,amsfonts,amssymb,epsfig,verbatim}
\usepackage{amsmath,amsthm,amssymb,latexsym,graphics,textcomp}
\usepackage{graphicx}

\usepackage{url}
\usepackage{psfrag}
\usepackage{amsmath,amsthm,amsfonts,amssymb,mathrsfs,bm,graphicx,stmaryrd}
\usepackage{mathtools}
\usepackage[usenames]{color}
\usepackage{hyperref}
\usepackage [latin1]{inputenc}

\usepackage{tikz}
\usetikzlibrary{shapes,arrows}

\begin{document}

\tikzstyle{decision} = [diamond, draw, fill=gray!20, 
    text width=4.5em, text badly centered, node distance=3cm, inner sep=0pt]
\tikzstyle{block} = [rectangle, draw, fill=gray!20, 
    text width=10em, text centered, rounded corners, minimum height=3em]
\tikzstyle{line} = [draw, -latex']
\tikzstyle{cloud} = [draw, ellipse,fill=gray!20, node distance=3cm,
    minimum height=2em]

\newtheorem{theorem}{Theorem}[section]
\newtheorem{prop}[theorem]{Proposition}
\newtheorem{assume}[theorem]{Assumption}
\newtheorem{lemma}[theorem]{Lemma}
\newtheorem{cor}[theorem]{Corollary}
\newtheorem{definition}[theorem]{Definition}
\newtheorem{conj}[theorem]{Conjecture}
\newtheorem{claim}[theorem]{Claim}
\newtheorem{qn}[theorem]{Question}
\newtheorem{defn}[theorem]{Definition}
\newtheorem{defth}[theorem]{Definition-Theorem}
\newtheorem{obs}[theorem]{Observation}
\newtheorem{rmk}[theorem]{Remark}
\newtheorem{ans}[theorem]{Answers}
\newtheorem{slogan}[theorem]{Slogan}
\newtheorem{corollary}[theorem]{Corollary}
\newtheorem{proposition}[theorem]{Proposition}
\newtheorem{observation}[theorem]{Observation}
\newtheorem{property}{Property}[subsection]
\newtheorem{remark}[theorem]{Remark}
\newtheorem{example}{Example}[section]

\newtheorem{question}{Question}
\newtheorem{maintheorem}{Theorem}
\newtheorem{maincoro}[maintheorem]{Corollary}
\newtheorem{mainprop}[maintheorem]{Proposition}

\newcommand{\bluecomment}[1]{\textcolor{blue}{#1}}
\newcommand{\boundary}{\partial}
\newcommand{\hhat}{\widehat}
\newcommand{\C}{{\mathbb C}}
\newcommand{\B}{{\mathbb B}}
\newcommand{\Ga}{{\Gamma}}
\newcommand{\G}{{\Gamma}}
\newcommand{\s}{{\Sigma}}
\newcommand{\PSL}{{PSL_2 (\mathbb{C})}}
\newcommand{\pslc}{{PSL_2 (\mathbb{C})}}
\newcommand{\pslr}{{PSL_2 (\mathbb{R})}}
\newcommand{\Gr}{{\mathcal G}}
\newcommand{\integers}{{\mathbb Z}}
\newcommand{\natls}{{\mathbb N}}
\newcommand{\ratls}{{\mathbb Q}}
\newcommand{\reals}{{\mathbb R}}
\newcommand{\proj}{{\mathbb P}}
\newcommand{\lhp}{{\mathbb L}}
\newcommand{\tube}{{\mathbb T}}
\newcommand{\cusp}{{\mathbb P}}
\newcommand\AAA{{\mathcal A}}
\newcommand\HHH{{\mathbb H}}
\newcommand\BB{{\mathcal B}}
\newcommand\CC{{\mathcal C}}
\newcommand\DD{{\mathcal D}}
\newcommand\EE{{\mathcal E}}
\newcommand\FF{{\mathcal F}}
\newcommand\GG{{\mathcal G}}
\newcommand\HH{{\mathcal H}}
\newcommand\II{{\mathcal I}}
\newcommand\JJ{{\mathcal J}}
\newcommand\KK{{\mathcal K}}
\newcommand\LL{{\mathcal L}}
\newcommand\MM{{\mathcal M}}
\newcommand\NN{{\mathcal N}}
\newcommand\OO{{\mathcal O}}
\newcommand\PP{{\mathcal P}}
\newcommand\QQ{{\mathcal Q}}
\newcommand\RR{{\mathcal R}}
\newcommand\SSS{{\mathcal S}}
\newcommand\TT{{\mathcal T}}
\newcommand\UU{{\mathcal U}}
\newcommand\VV{{\mathcal V}}
\newcommand\WW{{\mathcal W}}
\newcommand\XX{{\mathcal X}}
\newcommand\YY{{\mathcal Y}}
\newcommand\ZZ{{\mathcal Z}}
\newcommand{\iid}{{i.i.d.\ }}
	\renewcommand{\ae}{{a.e.\ }}
\newcommand\CH{{\CC\Hyp}}
\newcommand{\Chat}{{\hat {\mathbb C}}}
\newcommand\MF{{\MM\FF}}
\newcommand\PMF{{\PP\kern-2pt\MM\FF}}
\newcommand\ML{{\MM\LL}}
\newcommand\PML{{\PP\kern-2pt\MM\LL}}
\newcommand\GL{{\GG\LL}}
\newcommand\Pol{{\mathcal P}}
\newcommand\half{{\textstyle{\frac12}}}
\newcommand\Half{{\frac12}}
\newcommand\Mod{\operatorname{Mod}}
\newcommand\Area{\operatorname{Area}}
\newcommand\ep{\epsilon}
\newcommand\Hypat{\widehat}
\newcommand\Proj{{\mathbf P}}
\newcommand\U{{\mathbf U}}
 \newcommand\Hyp{{\mathbf H}}
\newcommand\D{{\mathbf D}}
\newcommand\Z{{\mathbb Z}}
\newcommand\R{{\mathbb R}}
\newcommand\Q{{\mathbb Q}}
\newcommand\E{{\mathbb E}}
\newcommand\EXH{{ \EE (X, \HH_X )}}
\newcommand\EYH{{ \EE (Y, \HH_Y )}}
\newcommand\GXH{{ \GG (X, \HH_X )}}
\newcommand\GYH{{ \GG (Y, \HH_Y )}}
\newcommand\ATF{{ \AAA \TT \FF }}
\newcommand\PEX{{\PP\EE  (X, \HH , \GG , \LL )}}
\newcommand{\lct}{\Lambda_{CT}}
\newcommand{\lel}{\Lambda_{EL}}
\newcommand{\lgel}{\Lambda_{GEL}}
\newcommand{\lre}{\Lambda_{\mathbb{R}}}

\newcommand\til{\widetilde}
\newcommand\length{\operatorname{length}}
\newcommand\tr{\operatorname{tr}}
\newcommand\cone{\operatorname{cone}}
\newcommand\gesim{\succ}
\newcommand\lesim{\prec}
\newcommand\simle{\lesim}
\newcommand\simge{\gesim}
\newcommand{\simmult}{\asymp}
\newcommand{\simadd}{\mathrel{\overset{\text{\tiny $+$}}{\sim}}}
\newcommand{\ssm}{\setminus}
\newcommand{\diam}{\operatorname{diam}}
\newcommand{\pair}[1]{\langle #1\rangle}
\newcommand{\T}{{\mathbf T}}
\newcommand{\I}{{\mathbf I}}
\newcommand{\pG}{{\partial G}}
 \newcommand{\oxi}{{[1,\xi)}}
\newcommand{\cg}{\mathcal{G}}

\newcommand{\tw}{\operatorname{tw}}
\newcommand{\base}{\operatorname{base}}
\newcommand{\trans}{\operatorname{trans}}
\newcommand{\rest}{|_}
\newcommand{\bbar}{\overline}
\newcommand{\lbar}{\underline}
\newcommand{\UML}{\operatorname{\UU\MM\LL}}
\newcommand{\EL}{\mathcal{EL}}
\newcommand{\ncox}{{N_C([o,\xi))}}
\newcommand{\qle}{\lesssim}

\def\ind{{\mathbf 1}}
\def\N{\mathbb{N}}
\def\L{\mathbb{L}}
\def\P{\mathbb{P}}
\def\Z{\mathbb{Z}}
\def\R{\mathbb{R}}
\def\S{\mathcal{S}}
\def\X{\mathcal{X}}
\def\E{\mathbb{E}}
\def\l{\ell}
\def\Ups{\Upsilon}
\def\om{\omega}
\def\Om{\Omega}

\newcommand\Gomega{\Omega_\Gamma}
\newcommand\nomega{\omega_\nu}
\newcommand\omegap{{(\Omega,\P)}}
\newcommand\omegapp{{(\Omega',\P)}}

\DeclarePairedDelimiter\floor{\lfloor}{\rfloor}
\newcommand{\cF}{\mathcal{F}}
\newcommand{\cD}{\mathcal{D}}
\newcommand{\cZ}{\mathcal{Z}}
\newcommand{\cf}{\mathcal{F}}
\newcommand{\cA}{\mathcal{A}}
\newcommand{\cB}{\mathcal{B}}
\newcommand{\cC}{\mathcal{C}}
\newcommand{\cT}{\mathcal{T}}
\newcommand{\rb}{\mathfrak{A}}
\newcommand{\Var}{\hbox{Var}}
\newcommand{\ee}{\hbox{e}_1}
\newcommand{\cd}{\mathcal{D}}
\newcommand{\ce}{\mathcal{E}}
\setcounter{tocdepth}{1}

\title{First Passage Percolation on Hyperbolic groups}

\author{Riddhipratim Basu}
\address{Riddhipratim Basu, International Centre for Theoretical Sciences, Tata Institute of Fundamental Research, Bengaluru, India}

\email{rbasu@icts.res.in}

\author{Mahan Mj}
\address{Mahan Mj, School
of Mathematics, Tata Institute of Fundamental Research. 1, Homi Bhabha Road, Mumbai-400005, India}

\email{mahan@math.tifr.res.in}



\date{\today}

 \begin{abstract}
We study first passage percolation (FPP) on a Gromov-hyperbolic  group $G$ with boundary $\pG$ equipped with the Patterson-Sullivan measure $\nu$. We associate an i.i.d.\ collection of random passage times to each edge of a  Cayley graph of $G$, and investigate classical questions about asymptotics of first passage time as well as the geometry of geodesics in the FPP metric. Under suitable conditions on the passage time distribution, we show that the `velocity' exists in $\nu$-almost every direction $\xi\in \pG$, and is almost surely constant by ergodicity of the $G-$action on $\pG$.
For every $\xi\in \pG$, we also show almost sure coalescence of any two geodesic rays directed towards $\xi$. Finally, we show that the variance of the first passage time grows linearly with word distance along word geodesic rays in every fixed boundary direction. This provides an  affirmative answer to a conjecture  in \cite{bz-tightness, benjamini-tessera}.
\end{abstract}

\maketitle

\tableofcontents

\section{Introduction}\label{sec-intro} 
First passage percolation (FPP) is a  well-known probabilistic model for fluid flow through random media. It assigns \iid weights to edges of a graph and analyses the first passage time (i.e., the weight of the minimum weight path) as well as the geodesic (the optimal path) between any two points. For the Cayley graph of $\Z^2$ with respect to standard generators, this was introduced by Hammersley and Welsh \cite{HW} more than fifty years back. While $\Z^2$ and more generally $\Z^d$ have been investigated thoroughly, the literature on other background geometries is sparse. In the special case of (Gromov) hyperbolic geometry \cite{gromov-hypgps},  some results have been established by Benjamini, Tessera and Zeitouni \cite{bz-tightness,benjamini-tessera} and a number of test questions have been raised there (see also the  recent work \cite{CS18} on a related but different theme). In particular, \cite{bz-tightness} established the tightness of fluctuations of the passage time from the center to the boundary of a large ball, and \cite{benjamini-tessera} established the almost sure existence of bi-geodesics in hyperbolic spaces. The aim of this paper is to undertake a more detailed study into FPP on Cayley graphs of hyperbolic groups and address the fundamental questions that have been thoroughly investigated and, in many cases, answered for FPP on the Euclidean lattice. It turns out that some of the basic questions, (e.g., the existence of the limiting constant for the scaled expected passage time along a direction) becomes harder in our setting, whereas some of the other questions (e.g., fluctuations of passage times and coalescence of geodesics) which are unresolved (or only resolved under strong unproven assumptions) in the Euclidean settings can be addressed in the hyperbolic setting primarily due to the nice behavior forced upon  geodesics by the underlying geometry. 

Let us first briefly describe our main results informally, while postponing the formal statements and precise assumptions to later in the paper.  For the rest of this paper, $G$ will be a finitely generated (hence countable) hyperbolic group (precise definitions are given in Section \ref{sec-prel}), $\Ga = \Ga(G,S)$ a Cayley graph of $G$ with respect to a finite (symmetric) generating set $S$. The Cayley graph $\Ga$ comes naturally equipped with the word metric.
Let $\nu$ denote the Patterson-Sullivan measure, or 
more precisely, an element of the Patterson-Sullivan measure class, on the boundary   $\partial G$ of $G$. This is a quasi-conformal and
Hausdorff measure on $\partial G$ (see  \cite{haissinsky,bhm} for a  probabilistically oriented account). Note that
 $\partial G$ equipped with the Patterson-Sullivan measure
is typically a fractal and therefore a little difficult to get one's hands on (even in the simplest non-trivial case of a cocompact lattice $G$ in the hyperbolic plane, where $\partial G$ is topologically a circle). We therefore adopt the viewpoint that it is rather more
helpful  to  think of the uniform measure on
the (discrete) boundary of the $n-$ball
for large $n$  in 
$\Ga $ as a good
discrete approximant
of $(\partial G, \nu)$. 

 We shall consider first passage percolation with i.i.d.\ positive weights on the edges of $\Ga$ coming from a distribution with sub-Gaussian tails (see Section \ref{sec-prel} for details). The main results of this paper are the following:

\begin{enumerate}
\item {\bf Velocity exists and is constant:} One of the first results in Euclidean FPP is that, under minimal conditions, the  expected first passage time (from origin) in any direction grows linearly in the Euclidean distance and there exists a limiting constant (often referred to as time constant or velocity, we shall use the latter). This is a straightforward consequence of sub-additivity for rational directions. One can also show that the velocity varies continuously with the direction and upgrade this to a shape theorem \cite{CD81}. In the hyperbolic setting, the situation is  quite different in flavor. We parametrize directions in $G$ by points in the boundary $\pG$. We show that for $\nu$-almost every  $\xi \in \pG$, the velocity  (i.e., limit of the linearly scaled expected passage time) $v(\xi)$ exists along a word geodesic ray from the identity element in the direction $\xi$ (Theorem \ref{thm-vexists}). Further, $v(\xi)$ is constant almost everywhere. We provide examples (Section \ref{sec-eg}) to show that we cannot replace "almost everywhere" by "everywhere". \footnote{It is worth noting that velocity, in a given direction, usually refers to the almost sure limit of scaled passage times in the literature. This is equal to the scaled limit of expected passage times in the Euclidean case. Even though we consider the scaled limit of expected passage times here, our results also show convergence in probability. We expect almost sure convergence to hold, but postpone this to future work.}

\item {\bf Linear Variance:} Once the first order behavior has been established, the next natural question is to understand the fluctuation of passage times along the word geodesic in a fixed direction. This question remains unresolved in Euclidean FPP, though it is widely believed that the fluctuations are subdiffusive, exhibiting a power law behavior with exponent $\chi=\chi(d)<\frac{1}{2}$ for all dimensions $d$. In particular, it is predicted that $\chi(2)=\frac{1}{3}$. However the best rigorous upper bound on the exponent still remains $\frac{1}{2}$ for FPP on $\Z^d$ (\cite{Kes93,bks-sublinear}). In contrast, in the hyperbolic setting one expects the variance to grow linearly in the word distance, and this was conjectured in  \cite[Question 5]{bz-tightness}, \cite[Section 4]{benjamini-tessera}. We confirm this conjecture (Theorem \ref{t:linvar}). 
\end{enumerate}

It is well understood from the study of FPP in the Euclidean case that questions of fluctuations of the first passage times are intimately connected with the geometry of finite and semi-infinite geodesics. Indeed our proof of Theorem \ref{t:linvar} requires understanding  the geometry of geodesics as well as semi-infinite geodesic rays, and yield some results that are interesting in their own right.

\begin{enumerate}
\item[(3)] {\bf Direction of Geodesic Rays:}  Almost surely FPP geodesic rays have a well-defined direction (Theorems \ref{t:direxists} and  \ref{dirnexists}). As in the case of word-geodesics, these are parametrized by points $\xi \in \pG$. A similar result is known in the Euclidean setting only under the unproven assumption of uniform curvature of the limit shape \cite{New95} or strong convexity of the same \cite{DH14}. Direction of Busemann functions of geodesics in the Euclidean setting has been established in \cite{AH16}.

\item[(4)] {\bf Coalescence of Geodesic Rays:} We show that for each direction $\xi \in \pG$, $o_1, o_2 \in G$, the pair of FPP geodesic rays from $o_1$ and $o_2$ in the direction $\xi$ almost surely coalesce, i.e., the set of edges in the symmetric difference of the two geodesic rays outside a sufficiently large ball centered at the identity element is empty (Theorem \ref{coalesce}). This is of course not true in general for word geodesic rays in hyperbolic groups (see Section \ref{sec-fppz}):  geodesic rays in $\Ga$ converging to the same $\xi \in \pG$ eventually lie in a uniformly bounded neighborhood of each other. Coalescence in hyperbolic groups forces FPP geodesics to actually coincide beyond a point--a strictly stronger phenomenon. It follows from coalescence that for each $\xi\in \pG$, almost surely there exists a unique geodesic ray from the identity in direction $\xi$. In the planar Euclidean setting, coalescence and uniqueness of geodesic rays is known either in almost every direction or under some additional unproven assumptions such as the differentiability of the boundary of the limit shape  \cite{LN96, DH17}.
\end{enumerate}

\subsection{Outline of the paper} The rest of this paper is organized as follows. In Section \ref{sec-prel}, we make formal definitions and set up basic notations for FPP on $\Ga$ (Section \ref{sec-prelperc}), recall preliminaries on hyperbolic groups (Section \ref{sec-prelhg}) and collect some standard probabilistic tools (Section \ref{sec-concn}) that we shall need in the rest of the paper. 

\begin{center}
\begin{figure}[htbp!]
\centering
\begin{tikzpicture}[node distance = 3.3cm, auto]
    \node [block] (prelim) {{\scriptsize Preliminaries}\\ {\scriptsize (Sec. 2)}};
    \node [block, right of= prelim, node distance= 4.5cm] (canon) {{\scriptsize Consequences of automatic}\\ {\scriptsize structure (Sec. 3)}};
    \node [block, below of= canon, node distance= 2cm] (velocity) {{\scriptsize Existence of velocity}\\ {\scriptsize (Sec. 5)}};
      \node [block, below of=prelim, node distance=2 cm] (approx) {{\scriptsize Approximating geodesics} \\ {\scriptsize (Sec. 4)}};
     \node [block, left of=approx, node distance=4.5 cm] (direction) {{\scriptsize Direction of geodesic rays} \\ {\scriptsize (Sec. 6)}};
      \node [block, below of=direction, node distance=2 cm] (coal) {{\scriptsize Coalescence of geodesic rays}\\ {\scriptsize (Sec. 7)}};
      \node [block, right of=coal, node distance=4.5 cm] (variance) {{\scriptsize Linear variance growth}\\ {\scriptsize (Sec. 8)}};
     \path [line] (prelim) -- (direction);
      \path [line] (canon) -- (velocity);
    \path [line] (prelim) -- (approx);
    \path [line] (approx) -- (velocity);
    \path [line] (direction) -- (coal);
     \path [line] (coal) -- (variance);
     \path [line] (approx) -- (variance);
     \path [line] (approx) -- (coal);
\end{tikzpicture}
\caption{Logical dependence among different sections of the paper}
\label{fig:outline}
\end{figure}
\end{center}

The next three sections aim at establishing the existence of velocity (Theorem \ref{thm-vexists}). Section \ref{sec-hprels} first recalls Cannon's theorem on the existence of  automatic structures,  and its connections with the Patterson-Sullivan measure \cite{calegarifujiwara,calegarimaher}. 
A convention we  shall follow in this paper is the following:  we shall
 refer to 
the Patterson-Sullivan measure class as
the Patterson-Sullivan measure.  Any two members of the class
 are absolutely continuous
with respect to each other with bounded Radon-Nikodym derivatives.
Since the statements we make are true  up to  bounded multiplicative
constants, this will not be an issue.
Next, we use these results to establish the main technical lemma of this section (Lemma \ref{ofreq-exist}) that proves the existence of frequency of occurrence of geodesic words along a geodesic ray $[1,\xi)$ from the identity in direction $\xi\in \pG$. 
The main aim of Section \ref{sec-wqg} is to establish an approximation result, Theorem \ref{t:approxqg}, for FPP geodesics. We look at large cylindrical neighborhoods $N_B([1,\xi))$  and look at passage times $T_B(x,y)$ between $x, y \in [o,\xi)$ when restricted to $N_B([o,\xi))$. We describe precisely in Theorem \ref{t:approxqg} how $T_B(x,y)$ approximates the passage time $T(x,y)$ in $\Ga$ between $x, y$. Finally, in Section \ref{sec-vel}, we prove that velocity exists (Theorem \ref{thm-vexists}). Counterexamples are also provided to show that this is a statement about a full measure subset of $\pG$ with respect to the Patterson-Sullivan measure; and cannot be upgraded to a statement about all of $\pG$.

Section \ref{sec-dir} establishes the fact that FPP geodesic rays almost surely have a well-defined direction in $\pG$ (Theorems \ref{t:direxists} and  \ref{dirnexists}). The main technical tool here is Proposition \ref{bt}, which is an adaptation to our context of the main theorem of \cite{benjamini-tessera}. 

Section \ref{sec-coalesce} proves coalescence of geodesics (Theorem \ref{coalesce}). The main geometric tool for this is the construction of hyperplanes (Section \ref{sec-hypgeoprel}). We combine this geometric tool with probabilistic estimates and concentration inequalities in Section \ref{sec-coalesce-omega} to establish coalescence.

Section \ref{sec-linvar} uses the technology developed in Section \ref{sec-coalesce} to prove a conjecture of Benjamini, Tessera and Zeitouni \cite{benjamini-tessera,bz-tightness} which asserts that variance in passage time increases linearly with distance along word geodesic rays.

 A comment on the expository style adopted. Our aim here is to make the paper  accessible, as far as possible, to people working on either Gromov-hyperbolic groups or on FPP. We have therefore striven to provide some background from both topics, and have provided detailed proofs of several results which experts on one or the other topic, might well be familiar with. This is to make the paper as self-contained as possible.

To conclude, we mention that several of our main results go through in a more general setup than what is considered here. See Figure \ref{fig:outline} for an outline of the logical dependence structure of the paper. The velocity result Theorem \ref{thm-vexists} uses crucially  the group structure and its consequences from Section \ref{sec-hprels}. However, Sections \ref{sec-dir}, \ref{sec-coalesce} and \ref{sec-linvar} are independent of Sections \ref{sec-hprels} and \ref{sec-vel}, and use only a couple of results from Section \ref{sec-wqg}. Indeed, we only use Lemma \ref{l:geodlen} in the proof of Lemma \ref{l:sidetoside}, and quote Corollary \ref{l:expgeodlen} to prove the easier upper bound of variance in Sec. \ref{sec-linvar}; arguments similar to the proofs of Lemma \ref{l:geodlen} and Lemma \ref{l:positive} are also used in the proof of Proposition \ref{bt} and Lemma \ref{l:middle1}. As was pointed out in  comments on an earlier draft, the results of Sections \ref{sec-dir}, \ref{sec-coalesce} and \ref{sec-linvar} go through almost verbatim in the  purely geometric set-up of any bounded degree hyperbolic graph: see Remark \ref{rmk-ref}. The only issue to bear in mind is that in the setup of a Cayley graph, the identity element is chosen as a preferred base-point. However, for an arbitrary Gromov-hyperbolic graph with uniformly bounded degree, no such preferred base-point exists: any point can be chosen as a base-point. The boundary of any hyperbolic space being independent of the base-point, this does not create any complications for the group-independent arguments of Sections \ref{sec-dir}, \ref{sec-coalesce} and \ref{sec-linvar}.  The results go through for geodesics from the chosen base point in the direction of each point on the boundary. However, Theorem \ref{thm-vexists} fails in the absence of a group structure. Indeed,  we shall show in Section \ref{sec-free3}, Theorem \ref{thm-vexists} can even fail when we replace the Cayley graph of a hyperbolic group by a graph quasi-isometric to it.

\section{Preliminaries}\label{sec-prel} In this section, we formally define the first passage percolation model and collect together some basic notions from the classical theory of FPP. We also very briefly recall the preliminaries of the theory of hyperbolic groups and collect together some useful probabilistic estimates that we shall use throughout the paper.

\subsection{Preliminaries on first passage percolation}\label{sec-prelperc}
Let $\Ga$ be a graph and let $V, E$ denote the vertex and edge set of $\Ga$. Consider $\Ga$ as a metric space with the graph distance metric $d$ (where each edge in $\Ga$ is assigned  unit length): $d(x, y)$ is the minimum length of an edge-path joining $x,y$.  A minimum length path connecting $x,y\in \Ga$  will simply be referred to as  a geodesic and denoted $[x,y]$.

Fix a probability measure $\rho$ on $[0, \infty)$ and equip the Borel $\sigma$-algebra on the product space $\Omega= [0, \infty)^E$ with the product measure $\P=\rho^{\otimes E}$. A typical element of $(\Omega,\P)$ will be denoted by $\omega=\{\omega(e)\}_{e\in E}$ and the random variables $X_{e}: \Omega \to [0,\infty)$ given by $X_{e}(\omega)=\omega(e)$ will be independent and identically distributed with law $\rho$. Setting the edge length of the edge $e$ equal to $X_{e}$ defines a random metric on $\Ga$ (a priori, this is only a pseudo-metric; however, assuming henceforth that $\rho$ does not put any mass on $0$,  it is indeed a metric), the \textbf{first passage percolation} (FPP) metric. More precisely we have the following definitions.

\begin{defn}\label{def-fpp}
	Let $\gamma = \{e_1, \cdots e_k\}$ be an edge path. For $\omega \in \omegap$, the $\omega-$length of $\gamma$ is defined to be $$\ell_{\omega}(\gamma):= \sum_{e\in \gamma } \omega(e).$$ Define $$d_{\omega} (x,y) := \inf_\gamma \ell_{\omega}(\gamma),$$ where $\gamma$ ranges over edge paths connecting $x, y \in V$. The random variable $T(x,y)$ defined on $\omegap$ by {\footnote{We shall use the notation $d_{\omega}(x,y)$ while working with the metric in a fixed realization $\omega$ of the passage time configuration, and $T(x,y)$ while considering properties of the random variable.}} $$T(x,y)(\omega)=d_\omega (x,y)$$ will be called the {\bf first passage time} between $x$ and $y$.
\end{defn}

We shall assume throughout that $\rho$ is continuous, i.e., it does not have any atoms; in particular it does not put any mass on $0$. Under such hypotheses it is easy to see that paths attaining the first passage time exist and are unique almost surely. 

\begin{defn}
A path that realizes $d_{\omega}(x,y)$ will be called an {\bf $\omega-$ geodesic}, denoted $[x,y]_{\omega}$. Observe that under our hypothesis on $\rho$, for $\P$-a.e. $\omega\in \Omega$, there is a unique $\omega$-geodesic between each pair of points in $\Ga$. For fixed vertices $x,y\in \Ga$, this ($\P$-a.e. well-defined) random path $\Upsilon(x,y)$ (i.e., $\Upsilon(x,y)(\omega)$ denotes the $\omega$-geodesic between $x$ and $y$) will be called the {\bf FPP-geodesic} between $x$ and $y$. 	
\end{defn}

The study of first passage percolation on a graph usually focuses on understanding asymptotic properties of $T(x,y)$ and $\Upsilon(x,y)$ for two points far away in the underlying metric of the graph. 

\medskip

\noindent
\textbf{Assumptions on $\rho$:}
Throughout we shall assume that the passage time distribution $\rho$ satisfies the following conditions:
\begin{enumerate}
\item[i.] The support of $\rho$ is contained in $[0,\infty)$.
\item[ii.] There are no atoms in $\rho$. 
\item[iii.] $\rho$ has sub-Gaussian tails, i.e., 
\begin{equation}
	\label{e:subgau}
	\exists a>0 ~\text{such that}~\int e^{ax^2}~d\rho(x)<\infty.
	\end{equation}
\end{enumerate}
Observe that our conditions are somewhat stricter than the ones usually assumed in the study of Euclidean FPP. Indeed, for the study of shape theorems or fluctuations, it is customary to only assume that the mass of the atom at $0$ is smaller than the critical probability of Bernoulli percolation and some appropriate moment conditions. Existence of geodesics can also sometimes be ascertained under weaker hypotheses than above. The above conditions are not even optimal for our proofs, but in the interests of transparency we have chosen to go with the simplest set of assumptions which still covers a wide class of distributions. Some of the proofs become easier if one assumes a stronger condition that the support of the passage time distribution is bounded away from $0$ and infinity, but our hypotheses already deals with the essential difficulties of working with passage times which are unbounded and can take values arbitrarily close to $0$. One obvious area of improvement is \eqref{e:subgau}. In fact, this hypothesis is only invoked in the proof of Theorem \ref{t:approxqg}, at every other place, the proof only requires an exponential tail decay of $\rho$. Even Theorem \ref{t:approxqg} (and hence all results in this paper) can be proved if $\rho$ is an exponential distribution (i.e., $\rho$ has density $e^{-\lambda x}$ on $\R_{+}$ for some $\lambda>0$), however in the interest of brevity and clarity of exposition, we shall refrain from trying to get optimal hypotheses in our results. 

\subsection{Preliminaries on hyperbolic groups}\label{sec-prelhg} We collect here some of the basic notions and tools from hyperbolic metric spaces that we shall need in this paper and refer the reader to \cite{gromov-hypgps,GhH,CDP,bh-book} for more details on Gromov-hyperbolicity.
For us, $\Ga$ will denote the Cayley graph of a  group $G$, typically hyperbolic, with respect to a finite symmetric generating set $S$. 

{A few words about the conventions we follow for Cayley graphs are in order.  We shall assume throughout this paper that   a generating set $S$ of a group $G$ is \emph{symmetric}, i.e.\ $s \in S$ if and only if $s^{-1} \in S$.}
	
	\begin{defn}\label{def-cg} { Given a group $G$ and a symmetric generating $S$,
		a \emph{directed Cayley graph} $\Gamma_d=\Gamma_d(G,S)$ is a directed graph defined as follows:\\
		The vertex set $V=V(\Gamma_d)$ consists of $\{g \vert g \in G\}$. The edge set 
	$E=E(\Gamma_d)$ consists of ordered pairs $\{(g,h) \vert \,  g, h \in G; \, g^{-1}h \in S \}$.
	Since $S$ is assumed to be symmetric, it follows that $(g,h) \in E(\Gamma_d)$ if 
	and only if $(h,g) \in E(\Gamma_d)$.}
	
	{ Given a group $G$ and a symmetric generating $S$,
		an \emph{undirected Cayley graph} or simply a \emph{Cayley graph} 
		$\Gamma=\Gamma(G,S)$ is an undirected graph defined as follows:\\
		The vertex set $V=V(\Gamma)$ consists of $\{g \vert g \in G\}$. The edge set 
		$E=E(\Gamma)$ consists of unordered pairs $\{(g,h) \vert \,  g, h \in G; \, g^{-1}h \in S \}$.}
	\end{defn}
{Note that in going from a directed Cayley graph to an undirected Cayley graph, two directed edges corresponding to ordered pairs $(g,h)$ and $(h,g)$ are identified with a single unordered pair $(g,h)$. If some $s \in S$ is of order 2, the directed Cayley graph $\Gamma_d$ has a pair of directed edges $(g,h)$ and $(h,g)$ for $h=gs$, whereas the undirected Cayley graph has only a single undirected edge $(g,h)$. The directed Cayley graph therefore detects order 2 elements geometrically, whereas the undirected Cayley graph does not.
	Since we shall only be interested in the large scale properties of Cayley graphs, this nicety of order 2 elements will not cause any problems. \emph{We emphasize that we shall be working with the undirected Cayley graph throughout this paper.}
	For concreteness, we note that for $G=\Z$ and $S=\{\pm 1\}$, the Cayley graph is simply
the undirected graph underlying the real line $\R$ with vertices at the integer points and
undirected edges consisting of the intervals $[n, n+1]$ for $n \in \Z$.}

 Thus $G$ acts on the left by isometries (graph-isomorphisms) on $\Gamma$. A metric space $(X,d_X)$ is called a geodesic metric space if for all $x, y \in (X,d_X)$, there exists a geodesic connecting $x,y$,
i.e.\ there exists an isometric embedding
$\iota: [0, d_X(x,y)] \to (X,d_X)$ such that $\iota (0) = x; \iota (d_X(x,y)) = y$.
A geodesic in a geodesic metric space $(X,d_X)$  joining $x, y$ will be denoted as $[x,y]$. The $c-$neighborhood of a set $A$ in a metric space $(X,d)$ will be denoted
as $N_c(A)$.

\begin{defn}\label{def-hyp}\cite{gromov-hypgps}
	A geodesic metric space $(X,d)$ is said to be $\delta-$hyperbolic if for all $x, y, z \in X$, $[x,y] \subset N_\delta ([x,z] \cup [y,z])$.
	A geodesic metric space $(X,d)$ is said to be hyperbolic if it is $\delta-$hyperbolic for some $\delta \geq 0$.
	
	A finitely generated group $G$ is said to be hyperbolic with respect to some finite symmetric generating set $S$ if the Cayley graph $\Ga = \Ga (G,S)$ (equipped with graph distance) is hyperbolic.
\end{defn}

\begin{defn}\label{def-qi}\cite{gromov-hypgps}
	A map $f: (X,d_X) \to (Y,d_Y)$ between metric spaces is said to be a $(K,\ep)-${\bf quasi-isometric embedding} if for all $x_1, x_2 \in X$,
	$$\frac{1}{K} d_X(x_1,x_2) - \ep \leq d_Y(f(x_1),f(x_2))
	\leq K d_X(x_1,x_2) + \ep.$$
	
	A $(K,\ep)-$quasi-isometric embedding $f: (X,d_X) \to (Y,d_Y)$ is said to be a $(K,\ep)-${\bf quasi-isometry}, if further,  $Y \subset N_K(f(X))$.
	
	A $(K,\ep)-$quasi-isometric embedding $f: I \to (Y,d_Y)$
	is said to be a $(K,\ep)-${\bf quasi-geodesic}, if $I$ is an interval (finite, semi-infinite or bi-infinite) in $\R$ (equipped with Euclidean metric). 
	
	A subset $A $ of a geodesic metric space $(X,d_X)$ is said to be $\kappa-${\bf quasiconvex}, if for all $x_1, x_2 \in A$,
	and any geodesic $[x_1, x_2] \subset X$, $[x_1, x_2] \subset N_\kappa (A)$.	 
\end{defn}

It was shown by Gromov \cite{gromov-hypgps} that if $G$ is  hyperbolic with respect to some finite symmetric generating set $S$, it is hyperbolic with respect to any other finite symmetric generating set $S'$. Thus, hyperbolicity is a property of finitely generated groups, not their generating sets. This follows from the following theorem that says that hyperbolicity is invariant under quasi-isometry:

\begin{theorem}[Gromov]\cite{gromov-hypgps,GhH}\cite[p. 401]{bh-book} \label{thm-hypqiinv} Given $\delta,  \ep\geq 0$ and $K\geq 1$, there exists $\delta'\geq 0$ such that the following holds.\\
	Let $(X,d)$ be $\delta-$hyperbolic and $f: (X,d)\to (Y,d')$ be a $(K, \ep)-$quasi-isometry. Then $(Y,d')$ is $\delta'-$hyperbolic.
\end{theorem}

Theorem \ref{thm-hypqiinv} allows us the freedom to choose any Cayley graph of a hyperbolic group $G$. The qualitative results we prove in this paper will thus be independent of the generating set $S$.

\begin{defn}\label{def-gi}\cite{gromov-hypgps}\cite[p. 410]{bh-book} 
	For any $x, y, o$ in a metric space $(X,d)$, the Gromov inner product of $x, y$ is given by $$\langle x, y \rangle_o = \frac{1}{2} (d(x,o) + d(y,o) - d(x,y)).$$
\end{defn}

The Gromov inner product above can be used to define the {\bf Gromov boundary} $\partial X$ of a hyperbolic $(X,d)$ as follows \cite{gromov-hypgps,GhH,bh-book}. Fix a base-point $o \in X$. We consider sequences
$\{x_n \}$ in $X$ satisfying the condition that $\langle x_n, x_m \rangle_o \to \infty$ as $m,n \to \infty$. Two such sequences $\{x_n \}$ and $\{x_n' \}$ are defined to be equivalent if  $\langle x_n, x_n' \rangle_o \to \infty$ as $n \to \infty$ (see \cite[p. 431]{bh-book}
for  a proof that this is an equivalence relation.)

\begin{defn}\label{def-bdy}\cite[p. 431]{bh-book} 
	The boundary $\partial X$ of $X$ is defined (as a set) to be the set of equivalence classes of sequences  $\{x_n \}$ as above. We write $x_n \to \xi$, if $\xi \in \partial X$ is the equivalence class of $\{x_n \}$.
\end{defn}

The Gromov inner product extends to $\partial X$ as follows \cite[p. 401]{bh-book}: Let $\xi, \xi' \in \partial X$. Then
$$\langle \xi, \xi' \rangle_o :=\sup \liminf_{m,n \to \infty}
\langle x_n, x_m' \rangle_o,$$ where $\{x_n \}$ (resp. $\{x_m' \}$) range over sequences in the equivalence class defining $\xi$ (resp. $\xi'$). The Gromov inner product $\langle \xi, \xi' \rangle_o$ can be used to define a metric on $\partial X$.

\begin{defn}\label{def-visualmetric}
	A metric $d_v$ on $\partial X$ is said to be {\bf a visual metric with parameter $a>1$} with respect to the base-point $o$ if there exist $k_1, k_2 > 0$ such that $$k_1 a^{-\langle \xi, \xi' \rangle_o} \leq d_v (\xi_1, \xi_2) \leq k_2 a^{-\langle \xi, \xi' \rangle_o}.$$
\end{defn}

\begin{prop}\label{prop-vmexist}\cite[p. 435]{bh-book} Given $\delta \geq 0$, there exists $a > 1$ such that if
	$(X,d)$ is $\delta-$hyperbolic, then for any base-point $o$,
	a visual metric $d_v $
	with parameter $a>1$  exists on $\partial X$ with respect to the base-point $o$.
	Further, for any $o' \in X$ a visual metric
	with parameter $a>1$ and with respect to the base-point $o'$
	is equivalent (as a metric) to $d_v$.
	
	There exists a natural topology on $\hat{X} = X \cup \partial X$ such that
	\begin{enumerate}
		\item $X$ is open and dense in $\hat X$,
		\item $\partial X$ and $\hat{X}$ are compact if $X$ is proper,
		\item the subspace topology on $\partial X$ agrees with that given by $d_v$.
	\end{enumerate}
\end{prop}

We call $\hat{X}$ the {\bf  Gromov compactification} of $X$.
For $\xi \in \partial X$ and $o \in X$, a geodesic ray from $o$ and converging to $\xi \in \partial X$ will be denoted by $[o, \xi)$. For $\xi_1 \neq \xi_2 \in \partial X$, a bi-infinite geodesic $f: \R \to X$ converging to $\xi_1, \xi_2$ as $s \in \R$ tends to $\pm \infty$ will be denoted by $(\xi_1, \xi_2)$.

\begin{lemma}[Morse Lemma]\label{lem-morse}\cite[p. 401]{bh-book} Given $\delta, \ep \geq 0$ and $K \geq 1$, there exists $\kappa \geq 0$ such that the following holds:\\
	Let $(X,d)$  be a $\delta-$hyperbolic space. Let $f: I \to X$ be a $(K,\ep)-$quasi-geodesic with $I=[a,b]$ finite. Then  $f(I) \in  N_\kappa ([f(a),f(b)])$ for any geodesic $[f(a),f(b)]$ in $X$ joining $f(a),f(b)$. In particular, any two
	$(K,\ep)-$quasi-geodesics joining $x, y \in X$ lie in a $ \kappa-$neighborhood of each other.
	
	When $I=[0,\infty)$ is semi-infinite, $[f(a),f(b)]$ is replaced by $[f(a),\xi)$ where $\xi$ is a unique point in $\partial X$. Finally, when $I = \R$, $[f(a),f(b)]$ is replaced by $(\xi_1,\xi_2)$ where $\xi_1, \xi_2$ are unique points in $\partial X$. Further,  any two
	$(K,\ep)-$quasi-geodesic rays joining $x \in X$ (or $\xi'\in \partial X$) to $\xi \in \partial X$ lie in a $ \kappa-$neighborhood of each other.
\end{lemma}

The Gromov boundary (Definition \ref{def-bdy}) can be defined also in terms of asymptote-classes of geodesic rays: Define (semi-infinite) geodesic rays $\gamma, \gamma': [0,\infty) \to X$ to be asymptotic if there exists $C_0  \geq 0$ such that for all $t \geq 0$, $(\gamma(t), \gamma'(t)) \leq C_0$. The next Lemma says that geodesic rays $\gamma, \gamma$ are asymptotic if and only if they converge to the same $\xi \in \partial X$:

\begin{lemma}\label{asymptosamept}\cite[p. 427]{bh-book}
	Let $X$ be $\delta-$hyperbolic.
	Let $\gamma, \gamma': [0,\infty) \to X$  be asymptotic geodesic rays. Then 
	\begin{enumerate}
		\item There exist $m, m' \in [0,\infty) $ such that for all
		$t \in [0,\infty) $, $d(\gamma(t+m), \gamma'(t+m')) \leq 4 \delta$.
		\item There exists $\xi \in \partial X$ such that for any pair of sequences $\{t_n\}, \{s_n\}$ in $[0,\infty)$ diverging to infinity, the sequences $\{\gamma(t_n)\}, \{\gamma'(s_n)\}$ lie in the  equivalence class of $\xi$.
	\end{enumerate}
\end{lemma}

The Patterson-Sullivan measure (Definition \ref{def-ps}) will be crucially used in this paper. For now, it suffices to say that it is a Borel measure $\nu$ supported
on $\pG$ (with respect to the topology defined by the visual metric), and is quasi-invariant under the natural action of $G$ on $\pG$.

\subsection{Probabilistic Tools}\label{sec-concn}
Here we record the basic probabilistic tools of concentration bounds and the FKG inequality that we shall use throughout. Note that these are mostly standard, but we shall provide appropriate references (or proofs) to make the exposition self-contained.\\

\noindent
\textbf{Concentration Inequalities:} We shall have occasion to use a number of concentration inequalities for sums of i.i.d.\ variables. The first one we need is the Chernoff Inequality (see e.g.\ \cite[Theorem 2.3.1]{V18}).

\begin{theorem}[Chernoff Inequality]
	\label{t:chernoff}
	Let $X_{i}$ be independent Bernoulli variables with $\theta:=\sum \E X_{i}$. Then for $\alpha>0$
	$$\P\left(\sum X_{i}> (1+\alpha)\theta\right)\leq e^{\theta (\alpha- (1+\alpha)\log (1+\alpha))}.$$
\end{theorem}

We next need concentration results for sums of i.i.d. random variables with sub-exponential tails (see e.g.\ \cite[Theorem 2.8.1]{V18}).

\begin{theorem}[Concentration for sums of i.i.d.\ sub-exponential random variables]
	\label{t:subexp}
	Let $X_i$ be i.i.d.\ non-negative random variables with distribution $\nu$ such that for some $a>0$ we have $\int_{0}^{\infty} e^{ax}~d\nu(x)<\infty$ and $\E[X_{i}]=\mu$. Then for each $\delta>0$, we have
	$$\P\left(\sum_{i=1}^n X_{i} \geq (1+\delta)n\mu \right) \leq e^{-cn}$$
	for some $c=c(\delta,\nu)>0$. Further, $c\geq c_1\delta$ for some $c_1 (=c_1(\nu))>0$ if $\delta$ is sufficiently large. 
\end{theorem}

 As any sub-Gaussian random variable is also sub-exponential, clearly the above result will also hold for $X_{i}\sim \rho$ where $\rho$ satisfies our hypotheses on the passage time distribution. 
 
The next result we shall need shows that for a sum of i.i.d.\ sub-Gaussian random variables the total contribution coming from terms that are sufficiently large is  only a small fraction with high probability. This is a less standard result, even though it follows from essentially the same arguments as classical concentration inequalities. We provide a short proof in the appendix for completeness. 

\begin{theorem}
	\label{t:iidtail}
	Let $X_{i}$ be i.i.d.\ sub-Gaussian non-negative random variables (i.e., $\P(X_i\geq t)\leq C_1e^{-c_1t^2}$ for some $C_1,c_1>0$ and $t>0$). Let $\epsilon>0$ be fixed. Then for $M>0$ sufficiently large, there exists $c(M)>0$ such that we have for all $n\geq n_0(c_1,C_1)$.
	$$\P\left(\sum_{i=1}^n (X_{i}-M)_{+} \geq \epsilon n \right) \leq e^{-cn}$$
	{where $x_{+}:=\max\{x,0\}$}.
	Further as $M\to \infty$ the constant $c=c(M)$  above also goes to $\infty$.
\end{theorem}
We postpone the proof to  Appendix \ref{sec-app}.\\

\noindent
\textbf{FKG Inequality:} Finally we recall the standard FKG correlation inequality, which is widely used in the study of FPP and related percolation models. We call a Borel subset of $\Omega$ increasing (resp.\ decreasing) if $\omega\in A$ implies $\omega'\in A$ if $\omega'(e)\geq \omega(e)$ for all $e\in E$ (resp.\ $\omega'(e)\leq \omega(e)$ for all $e\in E$). The following is a variant of the standard FKG inequality on product spaces (see e.g.\ \cite[Lemma 2.1]{Kesten2003}) 

\begin{theorem}
\label{t:FKG}
For any two increasing (or decreasing) Borel subsets $A$ and $B$ of $\Omega$, we have 
$$\P(A \cap B)\geq \P(A)\cdot\P(B).$$
\end{theorem}
In particular, Theorem \ref{t:FKG} shows that conditioning on an increasing (resp.\ decreasing) event makes another increasing (resp.\ decreasing) event more likely.  \\

\noindent {\bf Basic Set-up:} 
As mentioned above, $G$ will denote a finitely generated Gromov-hyperbolic group.  A Cayley graph of $G$ with respect to a fixed finite symmetric generating set $S$ will be denoted by $\Ga = \Ga(G,S)$ and its boundary (independent of the generating set) will be denoted by $\partial G$. The word-metric on $\Ga$ with respect to $S$ will be denoted by $d$.  Thus $d(x,y)$ equals the minimum number of edges in an edge-path joining $x, y$. We shall often use $|y|$ as a shorthand for $d(1,y)$.
We shall consider  FPP on $\Ga$ with edge weights distributed according to a measure $\rho$ satisfying the hypothesis in Section \ref{sec-prelperc}.\\

We shall now prepare the ground for one of our main results. Let $1$ denote the vertex of $\Ga$ corresponding to the identity element of $G$. Our objective is to study the asymptotics of the first passage time $T(x,y)$ for $x,y\in \Ga$ as $d(x,y)\to \infty$. By group invariance it suffices to set $x=1$, and study $T(1,y)$ for large $|y|$.  In analogy with the Euclidean case, it is natural to study $T(1,x_{n})$ as $x_{n}$ moves along some fixed direction parametrized by an element in $\pG$. Let $\xi\in \pG$ be fixed and let $[1,\xi)$  denote a fixed word geodesic ray from $1$ in the direction $\xi$, i.e., $\{1=x_0,x_1,\ldots, x_{n},\ldots\}$ such that $x_{n}\to \xi$, and each finite subpath of $[1,\xi)$ is a word geodesic between the corresponding endpoints. Clearly, this would imply $d(1,x_{n})=n$. It is not too difficult to show that $\E T(1,x_n)$ grows linearly in $n$, and we shall show that a limiting velocity $v(\xi):=\lim_{n\to \infty} \frac{\E T(1,x_n)}{n}$ exists (Theorem \ref{thm-vexists}) for almost every direction $\xi\in \pG$. The next three sections are devoted to the proof of this theorem.

\section{Automatic structure, Patterson-Sullivan measures and frequency}\label{sec-hprels} 
This section is devoted to recalling and developing some of the  technical tools from hyperbolic geometry that will go into the proof of Theorem \ref{thm-vexists}. We shall introduce an appropriate measure $\nu$ on $\pG$ and show that for $\nu$-almost every $\xi\in \pG$, there exists a limiting frequency of occurrence of fixed length geodesic words along $[1,\xi)$ (Lemma \ref{ofreq-exist}), provided the length lies in $d\, \natls$ for a suitable $d$.

In Sections \ref{sec-auto} and \ref{sec-cf}, we recall some facts about symbolic dynamics and hyperbolic groups \cite{gromov-hypgps,cp-book}. We refer to the excellent set of notes \cite{calegari-notes} where the necessary basics are summarized. Most of the relatively recent material here is due to Calegari, Fujiwara and Maher \cite{calegarifujiwara,calegarimaher}. In Section \ref{sec-freq}, we introduce the notion of ordered frequency--a modification of the counting function due to Rhemtulla \cite{rhemt}, rediscovered by Brooks \cite{brooks}. We use a vector-valued Markov chain argument (Proposition \ref{p:frequency}) along with the tools recalled from \cite{calegarifujiwara,calegarimaher} to prove that ordered frequencies exist (Lemma \ref{ofreq-exist}).

\subsection{Automatic structures on hyperbolic groups}\label{sec-auto}
Our starting point is a theorem of Cannon \cite{cannon-conetype,cannon-book}, saying that a hyperbolic group admits an automatic structure. We say briefly what this means, referring the reader to \cite{cannon-conetype,cannon-book,calegari-notes} for details.
Since the generating set $S$ of $G$ is chosen to be symmetric, $S$ generates $G$ as a semigroup. Then \cite{cannon-conetype} there exists a finite state automaton $\GG$ (equivalently, a finite digraph with directed edges labeled by $S$ and a distinguished initial state $1$) such that 
\begin{enumerate}
	\item any word obtained by starting at $1$ and reading letters successively on $\GG$  gives a geodesic in the Cayley graph $\Ga = \Ga(G,S)$ (the empty word being sent to the identity element in $G$). The set of all such words is denoted by $\L$ and is called the formal language accepted by $\GG$. 
	\item Let $e: \L \to G$ denote the evaluation map, sending $w\in \L$ to the element $g \in G$ it represents. For all $g \in G$, there exists a unique $w \in \L$ such that $e(w) = g$.
\end{enumerate}

A language $\L$ generated as above by reading words on a  finite state automaton $\GG$ (without reference to a group)  is called a {\bf prefix-closed regular language}. We also refer to $\L$ as the language accepted by $\GG$. If, as above, $\L$ encodes geodesics in $\Ga$, it is called a {\bf geodesic language}. The collection of geodesics in 
$\Gamma$ obtained in the process is called a {\bf geodesic combing}, or simply, a combing of $\Ga$. The finite directed and labeled graph $\GG$ is said to {\bf parametrize} the combing. Cannon's theorem  thus proves the following:

\begin{theorem}\label{cannon-auto}\cite{cannon-conetype}
	For $G$ hyperbolic and any symmetric generating set $S$, there exists  a
	finite state automaton $\GG$ that parametrizes a  combing of $\Ga = \Ga(G,S)$ corresponding to the prefix-closed regular geodesic language $\L$ accepted by $\GG$.
\end{theorem} 

\begin{rmk}\label{rmk-cannon}
Note that by the second condition above in the definition of a geodesic combing, for every $g \in G$, there is a unique geodesic word $w$ evaluating to $g$ under the 
evaluation map $e$.
\end{rmk}

Directed paths in $\GG$ starting at the initial vertex $1$  are in one-to-one correspondence with words in $\L$. We use the suggestive notation $1$ for the initial state as the evaluation map $e$ is assumed to send $1 \in \GG$ 
to the identity element $1 \in G$.
Let $\PP^0$ denote this collection of directed paths and let $\PP^0_n$ denote the subset of $\PP^0$ consisting of directed paths of length $n$  (starting at 1 by definition). The evaluation map $e:\L \to G$ then naturally gives an evaluation map $e:\PP^0 \to G $ by sending a word in $\L$ or equivalently a labeled path in $\PP^0$ to the element in $G$ it represents (we use the same letter $e$ for both). The set of {\it all} directed paths (without restriction on the base-point) in $\GG$ will be denoted as $\PP$ and the set of all directed paths of length $n$ in $\PP$ will be denoted by $\PP_n$. 
An important ingredient in the proof of Theorem \ref{cannon-auto} is the following:
\begin{defn}\label{def-cone}
	For $g\in G$, let $y \in \PP^0$ be the unique geodesic word such that $e(y)=g$. The cone $\cone(g)$ consists of the image under $e$ of all paths extending $y$. 
\end{defn} 
The uniqueness of $y$ in Definition \ref{def-cone} is guaranteed by Remark \ref{rmk-cannon}.
Note that $\cone(g)$ is the image under $e$ of the cylinder set in $\PP^0$ determined by $y$.
The underlying directed graph of $\GG$ is also called a {\bf topological Markov chain} and its vertices are called {\bf states}. Let $\VV(\GG)$ denote the set of states. We define an equivalence relation on $\VV(\GG)$ by declaring $v_1, v_2$ to be equivalent if there are directed paths from $v_1$ to $v_2$ and vice versa. Each equivalence class is called a {\bf component} and the resulting quotient directed graph is denoted $C(\GG)$. Then $C(\GG)$ has no directed loops.

Let $V$ denote the complex vector space of complex functions on $\VV(\GG)$.
Let $M$ denote the transition matrix of $\GG$: thus $M_{kl}$ equals one if there is a directed edge from the vertex labeled  $k$ to  the vertex labeled $l$ and is zero otherwise.  Let $\lambda$ denote its maximal eigenvalue. Similarly, for  each component $C$, let $M_C$ denote the transition matrix of $C$ and let $\lambda_C$ denote its maximal eigenvalue. Note that $\lambda_C$ is real as the transition matrix is non-negative (by Perron-Frobenius).
For all $C$, $\lambda_C \leq \lambda$. A component $C$ is said to be 
{\bf maximal} if   $\lambda_C = \lambda$. Recall that $B_n(1)$ denotes the $n-$ball about $1 \in G$ (recall that $1$ denotes the identity element of $G$). Let $$G_n := \{g \in G \mid d(1,g) = n\}$$ denote its `boundary', the $n-$sphere.

\begin{theorem}\cite{coornert-pjm} \cite[Lemma 4.15]{calegarifujiwara}
	\cite[Theorem 3.7]{calegarimaher}\label{cf-maxcomp}
	Let $G, \GG, C(\GG), M, \lambda$ be as above.
	Then each directed path in $C(\GG)$ is contained in at most one maximal component.
	There exist $K$ such that
	$$\frac{1}{K}\lambda^n \leq |G_n| \leq K\lambda^n$$
\end{theorem}

Theorem \ref{cf-maxcomp} shows in particular, that the algebraic and geometric multiplicities of the maximal eigenvalue $\lambda$ are equal. Calegari and Fujiwara \cite[Lemmas 4.5, \, 4.6]{calegarifujiwara} show that the following limits exist for all $v \in V$:
\begin{equation}\label{leftright}
\begin{array}{ll}
r(v)  :=  \lim_{n \to \infty} n^{-1} \sum_{i = 0}^{n} \lambda^{-i} M^i v.\\
l(v)   : =  \lim_{n \to \infty} n^{-1} \sum_{i = 0}^{n} \lambda^{-i} (M^T)^i v.
\end{array}
\end{equation}
Further $r(v)$ (resp. $l(v)$) equals the projection of $v$ onto the right (resp. left) $\lambda-$eigenspace of $M$.

\begin{defn}\label{def-nmu}
	Let $v_i$ denote the {\bf initial vector}, taking the value $1$ on the initial state $1$ and zero elsewhere. Let $v_u$ denote the {\bf uniform vector}, taking the constant value $1$ on all $x \in \VV(\GG)$.
	Let  $N$ be the matrix given by
	\begin{enumerate}
		\item \LARGE $N_{pq} =  
				\frac{M_{pq}r(v_u)_q}{\lambda \, r(v_u)_p}	 $
		\normalsize	if $r(v_u)_p \neq 0$,
		
		\normalsize	\item $N_{pp}=1$ and $N_{pq} = 0$ for $p \neq q$, if $r(v_u)_p = 0$.
	\end{enumerate}
	\normalsize
	Define a probability measure $\mu$ on  $\VV(\GG)$ by
\LARGE	$$\mu_p = \frac{ r(v_u)_p\, l(v_i)_p}{\sum_p r(v_u)_p\, l(v_i)_p}.$$
\end{defn}

Lemma 4.9 of \cite{calegarifujiwara} shows that $N$ is a stochastic matrix preserving  $\mu$. 
The measure $\mu$ and the matrix $ \, N$ define measures on $\PP_n$ by the usual strategy of defining measures on cylinders in path spaces. Let $\sigma = v_0v_1 \cdots v_n \in \PP_n$. Then define $$\mu (\sigma) := \mu(v_0) N_{v_0v_1}N_{v_1v_2}\cdots N_{v_{n-1}v_n}.$$

\subsection{Patterson-Sullivan measures}\label{sec-cf} We shall opt for two different normalizations for the Patterson-Sullivan measures (Definition \ref{def-ps} below) following \cite{coornert-pjm} and, more particularly, \cite[p. 1361]{calegarifujiwara}.

\begin{theorem}\cite{coornert-pjm,calegari-notes,calegarifujiwara}\label{thm-nmu}
	Define two  sequences of finite measures on $\Gamma \cup \partial G$ by
	$${\nu_n} = \LARGE \frac{ \sum_{|g| \le n} \lambda^{-|g|}\delta_g}{ \sum_{|g| \le n} \lambda^{-|g|}},$$
	and
	$$\widehat{\nu_n} = \LARGE \frac 1 n \sum_{|g| \le n} \lambda^{-|g|}\delta_g.$$
	Let $\nu$ and $\widehat{\nu}$ be respectively the weak limits of  $\nu_n$ and $\widehat{\nu_n}$ (up to subsequential limits). Then,  $\nu$ and $\widehat{\nu}$ are supported on $\partial G$. Further, any two subsequential limits are absolutely continuous with respect to each other with uniformly bounded Radon-Nikodym derivatives.
\end{theorem}

Though, a priori, we only have a measure class, it is proved in \cite{calegarifujiwara} that the limits $\nu$ and $\widehat{\nu}$  of $\nu_n$ and $\widehat{\nu_n}$ in Theorem \ref{thm-nmu} actually exist (see the comment after \cite[Definition 4.14]{calegarifujiwara} and Lemma \ref{psnmu} below culled from \cite{calegarifujiwara}).

\begin{defn}\label{def-ps} The measures
$\nu$ and $\widehat{\nu}$ (and sometimes their normalized versions)
will be called
the  {\bf Patterson-Sullivan measures} on $\pG$.
\end{defn}

It turns out \cite{calegarifujiwara} that $\widehat{\nu_n}, \widehat{\nu}$ are finite (but not necessarily probability) measures in the same measure classes as  $\nu_n$ and ${\nu}$ \cite{coornert-pjm} respectively. Further, the Radon-Nikodym derivatives of $\widehat{\nu}$ with respect to $\nu$ on $\pG$ are uniformly bounded away from zero and infinity. 

\begin{rmk}\label{rmk-ps}{\rm
We shall refer to 
  $\nu_n$ and $\widehat{\nu_n}$ as {\bf approximants} of $\nu$ and $\widehat{\nu}$ respectively (their existence and basic properties are proven in \cite{coornert-pjm}). Both $\nu$ and $\widehat{\nu}$  will be used in what follows as some properties are easier to state for one than the other. However, owing to the fact that
they are absolutely continuous with respect to each other with uniformly bounded  Radon-Nikodym derivatives, statements about one hold for the other up to uniformly bounded constants.}
\end{rmk}

Patterson-Sullivan measures of cones $\cone(g)$ are given by $$\nu(\cone(g))= \lim_{n \to \infty} \nu_n(\cone(g)),$$ and $$\widehat{\nu}(\cone(g))= \lim_{n \to \infty} \widehat{\nu_n}(\cone(g)).$$

The relation between $\mu, \, N$ (Definition \ref{def-nmu}) and $\nu$ is given as follows:

\begin{lemma}\label{psnmu}\cite[Sections 3.3, 3.4]{calegarimaher} Given $g \in G$, let $n \in \natls$
	and $\sigma = 1v_1v_2 \cdots v_n \in \PP^0_n$ be the unique directed path such that $e(\sigma) = g$. Then $$\nu(\cone(g))=N_{1v_1}N_{v_1v_2}\cdots N_{v_{n-1}v_n}.$$
\end{lemma}

\begin{defn}\label{def-psnmu}
	For $\sigma= 1v_1v_2 \cdots v_n\in \PP^0_n$ (as in Lemma \ref{psnmu} above), we define $$\nu(\sigma) = \nu(\cone(g))=N_{1v_1}N_{v_1v_2}\cdots N_{v_{n-1}v_n}.$$
	The Patterson-Sullivan measure $\nu$ on $\cone(g)$ also gives a measure on $G_n$ for all $n$, simply by defining $\nu (g) := \nu (\cone(g))$.
\end{defn}

Definition \ref{def-psnmu} thus gives us a well-defined way of lifting the Patterson-Sullivan measure $\nu$ to the path space $\PP^0$, by identifying the cylinder set corresponding to $g$ with the Borel subset of $\pG$ given by 
the boundary of $\cone(g)$ (see \cite{coornert-pjm} for details). Let $\SSS$ denote the left shift taking a sequence of vertices to the sequence omitting the first vertex. In particular, $\SSS(\PP^0) = \PP^0$. Then \cite[Lemma 4.19]{calegarifujiwara}, there exists a constant $c>0$ such that
\begin{equation}\label{eq-shift}
c\mu = \lim_{n \to \infty} \frac 1 n \sum_{i=0}^{n-1} \SSS^i_*\widehat{\nu},
\end{equation}
where $\SSS_{*}$ denotes the push-forward.  We caution the reader that in \cite{calegarifujiwara}, the constant
$c$ is not explicit.

The following Proposition shows that the Patterson-Sullivan measure $\nu$ and the uniform measure on $G_n$ are equivalent to each other with uniform constants.

\begin{prop}\cite[Section 4]{calegarifujiwara}\cite[Proposition 3.11]{calegarimaher} \label{uniform=ps}
	There exist $K, C \geq 0$ such that the following holds.
	For any $n \in \natls$ and $g\in G_n$,  let $B(g,C) =(N_C(g) \cap G_n)$ and let $B_0(g,C) =e^{-1}(B(g,C)) \subset \PP^0_n$ denote the pre-image of $B(g,C)$ under the evaluation map. Then 
	$$\frac{1}{K}\nu(B_0(g,C))/\nu(\PP^n_0) \leq |B(g,C)|/|G_n| \leq K\nu(B_0(g,C))/\nu(\PP^n_0).$$
\end{prop}

The following Proposition gives a quantitative estimate on the behavior of typical geodesics (recall that Definition \ref{def-psnmu} gives a well-defined way to lift $\nu$ to $\PP^0_\infty$).
\begin{prop}\cite[Proposition 4.10]{calegarifujiwara}\cite[Proposition 3.12]{calegarimaher}\label{cfm-quantmix} There exist $c_1,  c>0$ such that the following holds.
	Let $\sigma$ be a path in   $(\PP^0_n,\nu)$. Then, 
	apart from a prefix of size at most $c_1\log ( n)$,  $\sigma$ is entirely contained in a single maximal component of $\GG$  with probability $1 - O(e^{-c\log ( n)} )$.

	Also,  if $C$ is a  maximal component of $\GG$, then, as $n \to \infty$, a path $\sigma$ enters and
	stays in $C$ with probability ${\mu(C)}$, where $\mu(C)$ is computed from \eqref{eq-shift}.
\end{prop}

\begin{rmk}
There is a small typographical error in the statement of  \cite[Proposition 3.12]{calegarimaher}, where $1 - O(e^{-c\log ( n)} )$ is written as 
$1 - O(e^{-c n} )$. The statement follows from the fact that 
the number of steps that a Markov chain spends in
a communicating class satisfies a law with an exponentially decaying tail.
At any rate, the only output of Proposition \ref{cfm-quantmix} that we shall
use in this paper is
Corollary \ref{eventuallyunique}, which says that almost every path in
$\PP^0_\infty$ eventually lands in a maximal component.
\end{rmk}
Let $\PP^0_\infty$ (resp. $\PP_\infty$) denote the collection of infinite paths in $\PP^0$ (resp. $\PP$).

\begin{lemma}\label{ct-fsa}\cite[Lemma 3.5.1]{calegari-notes}
	The evaluation map $e: \PP^0_n \to G$ extends continuously to $\bbar e: \PP^0_\infty \to \partial G$ such that the cardinality $|\bbar{ e}^{-1} (\xi)|$ is uniformly bounded independent of
	$\xi \in \partial G$.
\end{lemma}

Let $\PP(C) \subset \PP$ denote the collection of paths contained in a maximal component $C$ and let $\PP_\infty(C)$ denote the collection of infinite paths contained in  $C$.
From \eqref{eq-shift}, Remark \ref{rmk-ps} and Proposition \ref{cfm-quantmix}, we have the following immediate corollary:
\begin{cor}\label{eventuallyunique}
	For $\nu-$a.e.\ $\sigma \in \PP^0_\infty$, there is a unique maximal component $C$ so that $\SSS^i (\sigma) \in  \PP_\infty(C)$ for all $i$
	sufficiently large.
\end{cor}

We shall define a $\nu$-full subset $\PP^\prime_\infty\subset \PP^0_\infty$ of paths  starting at the identity element  such that certain mixing conditions are satisfied along all trajectories $\sigma\in \PP^\prime_\infty$ (recall 
that the evaluation map identifies $\sigma \in \PP^0_\infty$ with semi-infinite geodesic words in $\Ga$).
 This leads us to the notion of ordered frequencies.

\subsection{Frequency}\label{sec-freq} In this subsection we shall introduce the notion of ordered frequencies. This  is a refinement of the counting function of Rhemtulla \cite{rhemt}, rediscovered by Brooks \cite{brooks}.
We shall use the technology recalled in Sections \ref{sec-auto} and \ref{sec-cf} to prove below the main technical lemma of this section: ordered frequencies exist along almost every word geodesic ray (Lemma \ref{ofreq-exist}). We shall need some basic facts from the theory of ergodic Markov chains.

\subsubsection{Markov Chain Trajectories}\label{sec-markovpower} 
We refer to \cite{levin-book} for details on mixing in Markov chains. Let $P$ denote the transition matrix of an irreducible  (but not necessarily aperiodic) Markov chain on a finite state space $\s$. Note that we are not assuming reversibility of the Markov chain. Let $d$ denote the period of $P$. 
Let $k,n\in \N$ with $k$ a multiple of $d$, 
$n\gg k$ be fixed and let us also fix $\mathbf{x}=(x_1,x_2,\ldots,x_{k})\in \s^{k}$. Let $\{X^x_n\}_{n\geq 1}$ denote a realization of the chain starting with $x\in \s$. Let $N_{n}(\mathbf{x},x)$ denote the number of positive integers $i \leq \frac{n}{k}-1$ such that $(X^x_{ik+1},\ldots, X^x_{(i+1)k})=\mathbf{x}$. We have the following result. 

\begin{prop}
	\label{p:frequency}
For each $x\in \s$, the following holds almost surely. 	
	For all $k\in d\N$, $\mathbf{x}\in \s^{k}$, and  $\epsilon>0$,  there exist $f(\mathbf{x},x)\geq 0$ (non-random) and $N_0$ (random but finite, depending on $k, \mathbf{x}, \epsilon$) such that 
	we have for all $n\geq N_0$
	$$\frac{N_n(\mathbf{x},x)}{n} \in  (f(\mathbf{x},x)-\epsilon, f(\mathbf{x},x)+\epsilon),$$
	i.e.\
 $$\frac{N_n(\mathbf{x},x)}{n} \to  f(\mathbf{x},x)$$
	almost surely.
\end{prop}

The proof is standard and uses the fact that associated vector-valued Markov chain, restricted to each of its recurrent component (which is determined by the initial state) is aperiodic and hence mixing. We provide the argument in  Appendix \ref{sec-app} for completeness.

\subsubsection{Ordered Frequency}\label{sec-ofreq}
We now turn to a refinement of the counting function of \cite{rhemt,brooks}.
Let $d$ denote the L.C.M.\ of the periods of
 (the topological Markov chain underlying) the finite state automaton $\GG$ parametrizing a geodesic combing $\L$
of $\Gamma$. (In the proofs below, it will suffice to take $d$ to be 
the L.C.M.\ of the periods of the maximal components.)
 By Corollary \ref{eventuallyunique}, for $\nu-$a.e.\ $\sigma$, there exists  a unique maximal component $C$ such that  $\sigma$  eventually lies in $ \PP_\infty(C)$. Let $\PP_\infty^{\max} \subset \PP^0_\infty$ denote the collection of such semi-infinite geodesic words.

 Let $n$ be a multiple of $d$.
For a geodesic word $w=g_1g_2\cdots g_n$ of length $n$, 
and $\sigma \in \PP_\infty^{\max}$, we shall define a notion of frequency 
of occurrence of $w$ in $\sigma$. Let $\{y_i\}, \, i=0,1,, \cdots$ be the sequence of (ordered equispaced) points on $\sigma$, such that
$e(y_0)=1$ and $d(e(y_i),1)=in$; so that $d(e(y_i),e(y_{i+1}))=n$.

An {\bf ordered occurrence} of $w$ in $\sigma$ is a pair $y_i,y_{i+1}$ such that 
$y_i^{-1}y_{i+1}$ equals $w$ as an ordered word. Let  $n_w ([y_0,y_i])$ be the  number of distinct ordered occurrences of $w$ in $[y_0,y_i] \subset \sigma$.


\begin{defn}\label{def-ofreq} Let $d, \sigma, y_j$ be as above.
The {\bf ordered frequency} of occurrence of $w$ along $\sigma$ is defined to be $$f_w (\sigma):= \lim_{i\to \infty} \frac{n_w([y_0,y_i])}{i},$$ provided the limit exists.
\end{defn}

\begin{rmk}\label{rmk-initialpoint}
Let $C$ be the unique maximal component   $\sigma$  eventually lies in.
Let $j \in \natls$ be the least integer such that $y_j$ onward, 
$\sigma$ lies in $C$. Then, observe that the existence of $f_w (\sigma)$ is equivalent to the existence of
 $\lim_{i\to \infty} \frac{n_w([y_j,y_i])}{i-j}.$ 
This observation will be useful in the proof of Lemma \ref{ofreq-exist} below.
\end{rmk}

\begin{rmk}\label{rmk-ofreq} {\rm
We have made Definition \ref{def-ofreq}  only  for $w$ with length
a multiple of $d$, though the definition per se works for arbitrary $w$. This is to address the fact that maximal components $C$, while irreducible, need not be mixing. The existence of ordered frequencies
(Lemma \ref{ofreq-exist} below) will be important for a coarse-graining argument
in Section \ref{sec-coarseg} to establish the existence of velocity.
}
\end{rmk}

Irreducibility of topological Markov chains corresponding  to  maximal components $C$ give us the following.
\begin{lemma}[Ordered frequencies exist]\label{ofreq-exist}
For $d$ and $\PP_\infty^{\max}$ as above and  for almost all $\sigma=\{x_0,x_1,\ldots,\} \in \PP_\infty^{\max}$ the following holds:
For each geodesic word $w$ of length a multiple of $d$ and each $\epsilon>0$ there exist $f_{w}(\sigma)$ and $N_0=N_0(\sigma, \ep, w)$ such that for all $i\geq N_0$ we have
$$\frac{n_w([y_0,y_i])}{i}\in (f_{w}(\sigma)-\epsilon, f_{w}(\sigma)+\epsilon);$$
that is, the ordered frequency $f_{w}(\sigma)$ exists for all such $w$.	
\end{lemma}

\begin{proof}  By Theorem \ref{cf-maxcomp}, the Markov chain $N$ (Definition \ref{def-nmu}) restricted to $C$ is irreducible as $C$ is maximal. Note also that maximal components corresponding to $N$ are maximal for the topological Markov chain $M$. Further, Lemma \ref{psnmu} shows that the law of $\sigma \in (\PP^0_\infty, \nu)$ is the same as the law of trajectories given by the Markov chain $N$.
	
	 Let $|w| = k$, where $k$ is a multiple of $d$.  Hence, the associated vector-valued Markov chain of $k-$tuples is mixing.  Let  $w=g_1\cdots g_k \in \PP_k$, where $g_i$'s are generators of $G$.  The ordered frequency $f_w (\sigma)$  equals the frequency of occurrence of the $k-$tuple $(g_1,\cdots, g_k)$ from the state space $\s^k$ by Proposition \ref{p:frequency}. Since, the vector-valued Markov chain of $k-$tuples is mixing, it follows from Proposition \ref{p:frequency} and Remark \ref{rmk-initialpoint} 
	 that there exists a 
full measure subset of $ \PP_\infty(C)$ for which $f_w (\sigma)$ exists.	
\end{proof}

Henceforth we fix $\PP^\prime_\infty\subset \PP^0_\infty$ to satisfy the conclusion of Lemma \ref{ofreq-exist}. Also, let $\pG' = \bbar{ e}(\PP^\prime_\infty)$ be the $\nu-$full subset of $\pG$ obtained as the image of $\PP^\prime_\infty$.  Since $e(\PP_\infty^\prime) \subset \pG$ has full measure with respect to the Patterson-Sullivan measure $\nu$ (Corollary \ref{eventuallyunique}), we have the following:

\begin{cor}\label{cor-fullmrefreq}
For $\nu-$a.e.\ $\xi \in \pG$, there exists a geodesic ray $\sigma= [1, \xi)\in \PP^\prime_\infty$.
\end{cor}

\section{Approximating FPP geodesics}\label{sec-wqg} 
The aim in this section is to develop another technical tool needed for the proof of Theorem \ref{thm-vexists}. Recall that $x_{n}$ is a point at distance $n$ from the identity element $1$ along some fixed word geodesic along some boundary direction, and our objective is to understand $\E T(1,x_n)$. The basic idea is to prove that $T(1,x_{n})$ can be approximated by a sum of many independent random variables. To this end we shall consider FPP geodesics between $x, y \in [1,\xi)$ constrained to lie in $N_B([x,y])$. We shall show 
 (see Theorem \ref{t:approxqg} below for a precise statement) that FPP geodesic lengths between  $x, y \in \Gamma$ can be well approximated by the weight of the smallest weight path  joining $x, y$ in $N_B([x,y])$ for large $B$. By the Morse Lemma \ref{lem-morse} it is irrelevant which geodesic one chooses. Notice that if the support of $\rho$ is bounded away from $0$ and $\infty$, this is almost trivial (see Lemma \ref{lem-bddaway}), but for more general $\rho$ one needs to work a bit more.

Fix a point $x_{n}$ with $d(1,x_n)=n$, and a geodesic $[1,x_n]$. For $B\in \N$, let $T_{B}(1,x_{n})$ denote the weight of the smallest weight path in $\Ga$ joining $1$ and $x_{n}$ that does not exit $N_{B}([1,x_n])$. The main result in this section is the following. 

\begin{theorem}
	\label{t:approxqg}
	Given $\epsilon>0$, there exist $B=B(\epsilon)\in \N$ and $c=c({\epsilon}), C=C(\epsilon)>0$ such that for all $n\in \N$, all $x_{n}$ with $d(1,x_{n})=n$, and all choices of $[1,x_{n}]$ we have 
	$$\P\bigg(T_{B}(1,x_{n})\geq T(1,x_{n})+\epsilon n\bigg)\leq Ce^{-cn}.$$ 
\end{theorem}

Observe that it suffices to prove Theorem \ref{t:approxqg} only for $n$ sufficiently large, and we shall take $n$ to be sufficiently large in the proof without explicitly mentioning it every time. Before proceeding with the proof of Theorem \ref{t:approxqg},
we start with the following simple, deterministic, test case:
\begin{lemma}\label{lem-bddaway}
	Given $\delta, K$, there exists $B$ such that the following holds.
	Suppose that $(X,d)$ is a $\delta-$hyperbolic graph  and
	that $\rho$ is supported in $[\frac{1}{K}, K]$. Then for all $\omega \in \omegap$ and $x,y \in X$ $$[x,y]_\omega \subset N_B([x,y]).$$
\end{lemma}

\begin{proof}
	Let $d_\om$ denote the metric on $X$ induced by $\om \in \omegap$. Then the identity map from $X$ to itself gives a $K-$bi-Lipschitz map from $(X,d)$ to $(X,d_\om)$. The Lemma is now an immediate consequence of the Morse Lemma \ref{lem-morse}.
\end{proof}

The remainder of the proof of Theorem \ref{t:approxqg} is a truncation argument which has little to do with the hyperbolicity assumption and works for FPP on any bounded degree connected graph. The first lemma we need shows that the (word) length of the FPP-geodesic between two points at distance $n$ is $O(n)$ with exponentially small failure probability. This is a rather standard result; analogous results have been proved in the Euclidean case in \cite{Kes86} and in the hyperbolic graph context in \cite[Section 2]{benjamini-tessera}.

\begin{lemma}
	\label{l:geodlen}
	There exist $n_0\in \N$, $R<\infty$ and $c>0$ (depending both on $\rho$ and $|S|$) such that  for all pairs  of vertices $(u,v)$ in $\Gamma$ with $d(u,v)\geq n_0$ we have 
	$$ \P\bigg(\ell(\Upsilon(u,v))\geq R d(u,v)\bigg)\leq e^{-cd(u,v)}.$$
	Further, $c\geq c_1\sqrt{R}$ for some $c_1>0$ for $R$ sufficiently large.
\end{lemma}

\begin{proof}
	Fix $\delta>0$ (to be chosen arbitrarily small later) and choose $\eta\in (0,1/2)$ such that $\rho([0,\eta))\leq \delta$ (this can be done as there is no atom at $0$). Call an edge $e$ \textbf{good} if the weight of $e$ is at least $\eta$ and \textbf{bad} otherwise. Observe that the weight of a path is at least $\eta $ times the number of good edges in the path. Observe that by our assumption on $\rho$ and Theorem \ref{t:subexp}, it follows, by choosing $R\gg \int_{0}^{\infty} x~d\rho$ that 
	$$\P\bigg(T(u,v) \geq \sqrt{R}d (u,v)\bigg) \leq e^{-cd(u,v)}$$ 
	for some $c>0$. Also, $c\geq c_1\sqrt{R}$ for some $c_1>0$ for $R$ sufficiently large. 
	
	Hence, to prove the lemma, it suffices to prove that all paths between $u$ and $v$ with word length larger than $R d(u,v)$ has $\om-$length larger than $\sqrt{R}d(u,v)$ with exponentially small (in $d(u,v)$) failure probability. Noticing that the number of (self-avoiding) paths of length $j$ starting from $u$ is bounded above  by $|S|^{j}$, the above probability is bounded above  by $$\sum_{j=R d(u,v)+1}^{\infty}  |S|^{j}\P(A_j)$$ where $A_j$ denotes the event that a self avoiding path $\gamma$ of length $j$ has weight smaller than $\sqrt{R}d(u,v)$. Now observe that $\P(A_j)$ is further bounded above  by the probability that the number of bad edges on $\gamma$ is at least $j-\frac{\sqrt{R}d(u,v)}{\eta}$. Now choose $R$ sufficiently large so that for all $j>R d(u,v)$, we have $j-\frac{\sqrt{R}d(u,v)}{\eta}> \frac{j}{2}$. Observe that number of bad edges on $\gamma$ is a sum of i.i.d.\ Bernoulli variables with expectation $\delta j$. By choosing $\delta$ sufficiently small and using Chernoff inequality \ref{t:chernoff}, it follows that $$\P(A_{j})\leq e^{-cj \log (1/2\delta)}$$ where the constant $c$ is absolute (i.e., does not depend on $j$ or $\delta$). Now by choosing $\delta$ sufficiently small this term decays sufficiently fast to kill the entropy term $|S|^j$ and hence we get  $$\sum_{j=R d(u,v)+1}^{\infty}  |S|^{j}\P(A_j) \leq e^{-cR d(u,v)}.$$
	The exponent here also can be made arbitrarily large by making $R$ large,  thus completing the proof. 
\end{proof}

Lemma \ref{l:geodlen} has the following immediate corollary that we shall use in Section \ref{sec-linvar}.

\begin{cor}
\label{l:expgeodlen}
There exists $C'>0$ such that for each $n$ and for each $x_{n}$ with $d(1,x_{n})=n$ we have $\E [\ell(\Upsilon(1,x_n))]\leq C'n$.
\end{cor}

We now move towards the truncation argument and the proof of Theorem \ref{t:approxqg}. Let $\epsilon>0$ be fixed and let $R$ be as in Lemma \ref{l:geodlen}. Let us fix $0<\epsilon'<\frac{\epsilon}{2R}$ and $M=M(R,\epsilon)>0$ to be chosen sufficiently large later. For $\omega\in (\Omega, \P)$ let us define $\omega'\in \Omega$ by setting, for all edges $e\in \Ga$, 
\begin{enumerate}
	\item $\omega'(e)=\omega(e)$ if $\omega(e)\in [\epsilon',M]$;
	\item $\omega'(e)=\epsilon'$ if $\omega(e)\leq \epsilon'$;
	\item $\omega'(e)=M$  if $\omega(e)\geq M$.
\end{enumerate} 
The main idea is to use the fact that the geodesic $[1,x_{n}]_{\omega'}$, in the environment $\omega'$,  must lie within a bounded neighborhood of $[1,x_n]$. We then show that for the right choice of $M$ and $\epsilon'$, the geodesics $[1,x_n]_{\omega}$ and $[1,x_n]_{\omega'}$ are close in length  except for a very small measure set of $\omega$'s. We now make this heuristic precise.   

\begin{proof}[Proof of Theorem \ref{t:approxqg}]
	Let $\epsilon'$ be fixed as above and let $M$ be fixed sufficiently large to be specified later. Let $B=B(\epsilon',M)$, given by Lemma \ref{lem-bddaway} be such that for all $\omega\in \Omega$, we have that $[1,x_n]_{\omega'}$ is contained in $N_{B}([1,x_{n}])$. Note that $[1,x_n]_{\omega'}$ is not necessarily unique but the above conclusion holds for all choices of $[1,x_n]_{\omega'}$ by Lemma \ref{lem-bddaway}. 
	
	Let $X'(e)$ denote the random variable define by $X'(e)(\omega)=\omega'(e)$. Observe that, trivially, $\E X'(e)\leq \E X(e)+\epsilon'$. Set $R'= \frac{2(\E X(e)+\epsilon')}{\epsilon'}$, and let $\gamma^B(\omega)$ (resp.\ $\gamma^B(\omega')$) denote the path that attains weight $T_{B}(1,x_n)(\omega)$ (resp.\ $T_{B}(1,x_n)(\omega')$). Notice that by our choice of $B$, we have $\gamma^B(\omega')=[1,x_n]_{\omega'}$.  Let us define the following events ({recall that $R$ is the constant from Lemma \ref{l:geodlen}}): 
	$$\cA=\{\omega\in \Omega: \ell([1,x_n]_{\omega})\leq Rn\};$$
	$$\cA'=\{\omega\in \Omega: \ell([1,x_n]_{\omega'})\leq R'n\}.$$
	$$\cB=\left\{\omega\in \Omega: \sum_{e\in [1,x_n]_{\omega'}} (\omega(e)-\omega'(e))_{+} \leq \frac{\epsilon n}{2}\right\}.$$
	
	We shall first show that, on $\cA\cap \cA' \cap \cB$, we have $T_{B}(1,x_n)\leq T(1,x_n)+\epsilon n$. Observe that for any $\omega\in \cA$, we have that 
	\begin{equation}
	\label{e:com1}
	d_{\omega}(1,x_n) \geq\ell_{\omega'}([1,x_n]_{\omega})-Rn\epsilon' \geq d_{\omega'}(1,x_n)-\frac{\epsilon n}{2}
	\end{equation}
	by our choice of $\epsilon'$. Further, we have, for all $\omega\in \cA' \cap \cB$ 
	\begin{equation}
	\label{e:compare2}
	\ell_{\omega}(\gamma^{B}(\omega))\leq \ell_{\omega}([1,x_n]_{\omega'}) \leq d_{\omega'}(1,x_n)+\frac{\epsilon n}{2}.
	\end{equation}
	Comparing \eqref{e:com1} and \eqref{e:compare2}, we get that for all $\omega\in \cA\cap \cA' \cap \cB$ 
	$$\ell_{\omega}(\gamma^{B}(\omega))\leq  d_{\omega}(1,x_n) +\epsilon n,$$
	as desired.
	
	It remains now to provide an appropriate lower bound for $\P(\cA\cap \cA' \cap \cB)$. Note first that Lemma \ref{l:geodlen} gives $\P(\cA^c)\leq e^{-cn}$ for some $c>0$. Next observe that in the event that $\ell_{\omega'}([1,x_n])\leq 2n \E[X'(e)]$,  one obtains that $\cA'$ holds using the definition of $R'$. Using Theorem \ref{t:subexp} we get $\P((\cA')^c)\leq e^{-cn}$ for some $c>0$. 
	
	To complete the proof of the theorem it now suffices to show that 
	\begin{equation}
	\label{e:claim}
	\P(\cA'\cap \cB^c)\leq e^{-cn}
	\end{equation}
	for some $c>0$.

	For this, simply note that 
	$\cA'\cap \cB^c \subseteq \cB''$
	where $\cB''$ is the event 
	$$\cB''=\left\{\omega\in \Omega: \sup_{\gamma: \ell(\gamma)\leq R'n}\sum_{e\in \gamma} (\omega(e)-\omega'(e))_{+} \geq \frac{\epsilon n}{2}\right\}$$
	where the supremum is taken over all self-avoiding paths $\gamma$ from $1$ to $x_n$ that have length $\leq R'n$. Notice now that $(\omega(e)-\omega'(e))_{+}=(\omega(e)-M)_{+}$. Further, for each fixed self-avoiding path $\gamma$ from $1$ to $x_n$ of length $\leq R'n$ we have, using Theorem \ref{t:iidtail}, for $M$ sufficiently large,
	$$\P\left(\sum_{e\in \gamma} (X(e)-M)_{+} \geq \frac{\epsilon n}{2}\right) \leq e^{-cn},$$
	where $c=c(M)$ can be made arbitrarily large by making $M$ arbitrarily large.
	There are at most $|S|^{R'n}$  self avoiding paths of length bounded by $R' n$. By choosing $M$ appropriately large and taking a union bound over such paths we  conclude that $\P(\cB'')\leq e^{-cn}$ for some $c>0$ which completes the proof of \eqref{e:claim} and hence the theorem. 
\end{proof}

\begin{rmk}{\rm
Notice that the proof of Theorem \ref{t:approxqg} is above is the only place in this paper where we have used the sub-Gaussian tail hypothesis \eqref{e:subgau}. Clearly the above proof is not optimal and the conditional can easily be relaxed. For example, it is not too hard to see that \eqref{e:claim} could be deduced easily if instead of \eqref{e:subgau}, we assumed that $\rho$ has sub-exponential tails together with the \textbf{bounded mean residual life property}, i.e., if $\rho$ has unbounded support then 
	\begin{equation}
	\label{e:mrl}
	\sup_{M>0} \big(\rho([M,\infty))\big)^{-1} \int_{M}^{\infty} x ~d\rho(x)<\infty.
	\end{equation}
	However, as already mentioned, we are not trying to get optimal hypotheses in this paper so we shall not discuss this in more detail.}
\end{rmk}

\section{Velocity}\label{sec-vel}
Recall that the Patterson-Sullivan measure on $\pG$ is denoted as $\nu$. The aim of this Section is to prove the following (see Definition \ref{def-velo} below for the precise definition of velocity):

\begin{theorem}\label{thm-vexists}
	For $\nu$-a.e.\ $\xi \in \pG$,  the velocity $v(\xi)$ in the direction of $\xi$ exists. Further,  $v(\xi)$ is constant $\nu$-almost everywhere in $\pG$.
\end{theorem}

Recall from Definition \ref{def-fpp} that the random variable 
$T(x,y)$ is  defined as $$T(x,y)(\om)=d_\om (x,y)$$
for $\om \in \omegap$, where $$d_\om (x,y) := \inf_\gamma l_\om(\gamma)$$
is the length of an $\om-$geodesic between $x, y$. We first show a law of large numbers result for approximate passage times defined below.

\subsection{Approximate velocity}\label{sec-appvel}
We  recall from Section \ref{sec-wqg}, the  $B-$neighborhood versions of the random variable $T(x,y)$.
For $x, y \in G$, define $d_{\om,B} (x,y) := \inf_\gamma l_\om(\gamma)$, where $\gamma$ ranges over all paths from $x$ to $y$ contained in 
$N_B([x, y])$. Recall that the $B-$passage time $T_B(x, y)$ from $x$ to $y$ is a random variable defined on $\omegap$ by the following:
$$T_B(x,y)(\om)=d_{\om,B} (x,y).$$
We shall use $T_B(x,y)$ as  a random variable on $\Omega$, and $d_{\om,B}$ as a function
on $G \times G$ (hence the two different pieces of notation).

Similarly, for subsets $U, V$ of $N_B([x, y])$ the $B-$passage time between $U, V$ will be denoted as $T_B(U, V)$ when $x, y$ are understood.
The expected  $B-$passage time from $x$ to $y$ is then given by
$$\E T_B(x,y):= \int_{\Omega} T_B(x,y)(\om)\, d\P.$$

For the following definition, let us fix $\xi\in \pG, B>4\delta$ and a word geodesic $[1,\xi)=\{1=x_0,x_1,\ldots, x_{n}, \ldots\}$ such that $d(1,x_{n})=n$ and $x_{n}\to \xi$.

\begin{defn}\label{def-velo}	
We define the upper (resp. lower) $B-$velocity   in the direction of $\xi$ to be 
$$\bbar{v_B}(\xi) := \limsup_{n \to \infty} \frac{\E T_B(1, x_n)}{n},  $$ (resp.)  
$$\lbar{v_B} (\xi) := \liminf_{n \to \infty} \frac{\E T_B(1, x_n)}{n}.$$ 
	
If $\bbar{v_B} (\xi) = \lbar{v_B} (\xi)$, we say that the $B-$velocity $v_B(\xi)$   in the direction of $\xi$  exists and equals $\bbar{v_B} (\xi) = \lbar{v_B} (\xi)$.
	
For $\xi \in \pG$, $B >0, \ep >0$, we say that the $(B, \ep)-$velocity   in the direction of $\xi$ exists if $$\bbar{v_B} (\xi) - \lbar{v_B} (\xi) \leq \ep.$$
	
For $\xi \in \pG$, we define the upper (resp.\ lower) velocity  in the direction of $\xi$ to be $$\bbar{v} (\xi) := \limsup_{n \to \infty} \frac{\E(T(1, x_n))}{n},$$ (resp.) 
$$\lbar{v} (\xi) := \liminf_{n \to \infty} \frac{\E(T(1, x_n))}{n}.$$ 
If $\bbar{v} (\xi) = \lbar{v} (\xi)$, we say that the velocity $v(\xi)$   in the direction of $\xi$  exists and equals $\bbar{v} (\xi) = \lbar{v} (\xi)$.
\end{defn}

\begin{rmk}
\label{r:well-defined}{\rm
Observe that a priori the quantities $v_{B}(\xi)$, $v(\xi)$ etc.\ defined above need not be well defined as they might depend on the choice of the word geodesic ray $[1,\xi)$. However, by considering two choices $\{1=x_0,x_1,\ldots, x_{n},\ldots\}$ and $\{1=y_0,y_1,\ldots, y_{n},\ldots\}$ of word geodesic ray from $1$ in the direction $\xi$ it follows that $|T(1,x_n)-T(1,y_n)|\leq T(x_n,y_n)$. Since $d(x_n,y_{n})$ is uniformly bounded by $2\delta$ (see the proof of \cite[Theorem 1.13]{bh-book} for instance), it follows that $n^{-1}\E T(x_n,y_n)\to 0$ and hence $v$ is indeed well-defined. The same argument works for $v_B$ since $B > 4 \delta$.}
\end{rmk}

\begin{rmk}
	{\rm Velocity usually refers to the inverse of the quantity in Definition \ref{def-velo}. We are following the convention from \cite{hm} where it is termed speed. In \cite{Kes93} the same quantity is called the ``time constant".}
\end{rmk}

We collect some basic properties of $B-$velocity.
\begin{prop}\label{approxbycnbhd} For a.e.\ $\xi \in \pG$, 
	\begin{enumerate}
		\item $\lim_{B\to \infty}\bbar{v_B} (\xi) = \bbar{v} (\xi).$
		\item  $\lim_{B\to \infty}\lbar{v_B} (\xi) = \lbar{v} (\xi).$
	\end{enumerate}
\end{prop}

\begin{proof}
	It follows from Theorem \ref{t:approxqg} (and the obvious fact $\sup_{n}\frac{1}{n}\E(T_B(1, x_n)) <\infty$) that  given $\ep > 0$, there exists $B_0$ such that for all $B \geq B_0$, $$\frac{1}{n}\big ( {\E(T_B(1, x_n))}-{\E(T(1, x_n))}\big ) \leq \ep.$$
	The given Proposition follows immediately.
\end{proof}

We shall need a basic theorem from Patterson-Sullivan theory.

\begin{theorem}\label{thm-psergodic} \cite{coornert-pjm}
	Let $G$ be a hyperbolic group and let $(\pG,\nu)$ denote its boundary equipped with a Patterson-Sullivan measure. Then the $G-$action on 
	$(\pG,\nu)$  is ergodic.
\end{theorem}

\begin{lemma}\label{lowerupperconstant}
	For a.e.\ $\xi \in \pG$, and every $B>0$, 
	\begin{enumerate}
		\item $\bbar{v_B} (\xi)$ is constant. 
		\item $\lbar{v_B} (\xi)$ is constant.
		\item $\bbar{v} (\xi)$ is constant.
		\item $\lbar{v} (\xi)$ is constant.
	\end{enumerate}	
\end{lemma}

\begin{proof}
	Each of the functions $\bbar{v_B} (\xi)$, $\lbar{v_B} (\xi)$, $\bbar{v} (\xi)$ and $\lbar{v} (\xi)$ are invariant under the $G-$action. This is because the formulae for  $[o,g] \cup g.[o,\xi)$ versus $[o, g.\xi)$ differ by only finitely many terms. Ergodicity of the $G-$action on $\pG$ (Theorem \ref{thm-psergodic}) furnishes the conclusion.
\end{proof}

The next lemma allows us to approximate $B-$passage times between $x, y$ by the $B-$passage times between the balls $N_B(x), N_B(y)$ provided
$d(x,y) \gg 1$.

\begin{lemma}\label{approxvel}
	For all $B>0$, and $\ep >0$, there exists $M \geq 0$ such that for all $n \geq M$, if $d(x,y) =n$,  then $$\frac{\E T_B(x,y) - \E T_B (N_B(x), N_B(y))}{n} \leq \ep.$$
\end{lemma}

\begin{proof}
	We start with the following deterministic inequality (for every $\om \in \omegap$):
	$$T_B (N_B(x), N_B(y))(\om)\leq T_B(x,y)(\om)\leq  T_B (N_B(x), N_B(y)) (\om)+ \max_{u\in N_B(x)} T(x,u)(\om)+\max_{v\in N_B(y)} T(y,v)(\om).$$
	The first inequality above is obvious. We turn to the second inequality. Let $u\in N_{B}(x)$ and $v\in N_{B}(y)$ be such that $T_{B}(\{u\}, \{v\})=T_B (N_B(x), N_B(y))$.   Consider the path from $x$ to $y$ obtained by concatenating the restricted geodesics between $x$ and $u$, $u$ and $v$, $v$ and $y$.  Then
	$$T_B(x,y)(\om)\leq  T_{B}(\{u\}, \{v\})(\om)+ T(x,u)(\om)+ T(y,v)(\om).$$
	(We caution the reader here that by $T_{B}(\{u\}, \{v\})$ we mean the passage time between $u, v$ when restricted to 	$N_B([x, y])$.
	
	Therefore, to prove the lemma, it would suffice to show the existence of a uniform (not depending on $d(x,y)$) upper bound on 
	$$\E \max_{u\in N_B(x)} T(x,u)+\E \max_{v\in N_B(y)} T(y,v).$$
	We shall only bound the first term above, the second term is bounded by an identical argument.
	
	Observe that $\max_{u\in N_B(x)} T(x,u) \leq \sum_{u\in N_{B}(x)} T(x,u)$ and further that for any $u\in N_B(x)$ we have $\E T(x,u)\leq d(x,u)\int x~d\rho(x) \leq B \int x~d\rho(x)$. Observe that $|N_B(x)|\leq |S|^{B}$ (recall that $S$ is the generating set of $\Ga$). Hence
	$$ \E \max_{u\in N_B(x)} T(x,u) \leq |S|^B B \int x~d\rho(x) $$
	which is independent of $d(x,y)$ completing the proof of the lemma.
\end{proof}

Note that the proof of Lemma \ref{approxvel} furnishes a  bound on 
the numerator occurring in LHS of its statement, in particular shows that the numerator is uniformly bounded in $n$ for a fixed $B$. However, for the applications (see Corollary \ref{cor-coarseexists} below), given $B$, we shall choose $n \gg B$ and use Lemma \ref{approxvel} in the form stated. In particular, for $B$ fixed, the expression in the LHS of Lemma \ref{approxvel} goes to zero at rate $\frac{1}{n}$ as $n\to \infty$.

\begin{rmk}\label{rmk-sep}{\rm 
	Recall that we have assumed that $G$ is $\delta-$hyperbolic. Also, assume
	that $B \geq 2\delta$, so that $N_B([1,\xi))$ is $\delta-$quasiconvex.
	There exist $\eta$ (depending only on $\delta$) and $i_0$ 
	(depending only on $B, \delta$)
	such that 
	for all $i\geq i_0$, the $(B+\eta)-$ neighborhood of $x_i$ disconnects $N_B([1,\xi))$, i.e.\
	paths  in $N_B([1,\xi))$ from $1$ to $\xi$ necessarily go through the $(B+\eta)-$balls about $x_i$ (see \cite[Section 4.2]{mitra-endlam} for instance). 
Now take an ordered sequence of points $y_0=1, y_1, y_2, \cdots$ on $[1,\xi)$ such that $d(y_i, y_{i+1}) > 4(B+\eta)$. Then the above disconnection property of  $(B+\eta)-$ neighborhoods shows that any path $\gamma \subset N_B([1,\xi))$ can be decomposed into connected pieces whose interiors lie either in
$ N_B([1,\xi))\setminus ( \cup_i N_{(B+\eta)} (y_i))$ or entirely inside some
$N_{(B+\eta)} (y_i)$.

Note that $N_{B+\eta}(z)$ has cardinality bounded by $a^{B+\eta}$ for some $a>0$ depending only on $\Gamma$, and its diameter is $2(B+\eta)$. By group-invariance, the expected passage time from any point in $N_{B+\eta}(z)$ to any other is bounded in terms of $B, \eta$  and the passage time distribution $\rho$. The number of connected pieces in the above decomposition that "backtrack", i.e.\ begin and end on the same 
$N_{(B+\eta)} (y_i)$ is thus bounded by $a^{B+\eta}$ and for any such piece, the expected passage time is bounded in terms of $B, \eta,\rho$.}
\end{rmk}

  Lemma \ref{approxvel} along with the above observation will be useful in understanding the behavior of  $\om-$geodesics constructed as a concatenation of segments that travel from $N_B(y_i)$ to  $N_B(y_{i+1})$,
where $\{y_i\}$ is a suitable coarse-graining of $[1,\xi)$. 

\subsection{Coarse-graining and existence of velocity}\label{sec-coarseg}

\begin{defn}\label{def-coarsegrain} 
	For $n \in \natls$ and $[1,\xi)$ a geodesic ray , let $y_0,y_1, y_2, \cdots$ be the sequence of points on 
	$[1,\xi)$ with $y_0=1$ and $d(y_i,1) = ni$. Let $$\E T([1,\xi),B,n, i):=\E(T_B(y_i, y_{i+1}))$$ denote the expected $B-$passage time from $y_i$ to $y_{i+1}$ along $[1,\xi)$.
	The {\bf coarse-grained  $(B, 0)$ velocity at scale $n$} along $[1,\xi)$ is said to exist if the limit  $$\lim_{m\to \infty} \frac{1}{m} \sum_{i=0}^{m-1} \E T([1,\xi),B,n, i)$$ of   Cesaro averages of the sequence $\{\E T([1,\xi),B,n, i)\}_i$ exists, and is then defined as $$v_B([1,\xi),n):=\frac{1}{n}\lim_{m\to \infty} \frac{1}{m} \sum_{i=0}^{m-1}  \E T([1,\xi),B,n, i).$$ 
	
	More generally, let $$\bbar{v_B}([1,\xi),n):= \frac{1}{n}\limsup_{m\to \infty} \frac{1}{m} \sum_{i=0}^{m-1}  \E T([1,\xi),B,n, i),$$ 
	and $$\underline{v_B}([1,\xi),n):=\frac{1}{n}\liminf_{m\to \infty} \frac{1}{m} \sum_{i=0}^{m-1}  \E T([1,\xi),B,n, i).$$ 
	The {\bf coarse-grained  $(B, \ep)$ velocity at scale $n$} along $[1,\xi)$ is said to exist if $\bbar{v_B}([1,\xi),n)-\underline{v_B}([1,\xi),n ) \leq \ep$.
\end{defn}

\begin{prop}\label{coarse-grainedCexists} Let $d$ and $\PP_\infty^\prime$ be as obtained from Lemma \ref{ofreq-exist}. Given $B\gg 1$ (large) and $\ep>0$ (small) there exists $M$ such that for all $n \in d\, \natls \cap [M,\infty)$ and  for all $\sigma = [1,\xi) \in \PP_\infty^\prime$, the
	coarse-grained  $(B, \ep)$ velocity at scale $n$ in direction $\xi$ exists.
\end{prop}

\begin{proof}
	Recall that $\pG' \subset \pG$ denotes the $\nu-$full subset of $\pG$ obtained as the image of $\PP^\prime_\infty$ under the evaluation map
	(Corollary \ref{cor-fullmrefreq}).

	Given $B, \ep$ as in the statement of the Proposition, 
	there exists $M \geq 0$ by Lemma \ref{approxvel}, such that for all $n \geq M$, if $d(x,y) =n$,  then 
	
	\begin{equation}\label{eq-approxvel}
	\frac{\E T_B(x,y) - \E T_B (N_B(x), N_B(y))}{n} \leq \frac{\ep}{2}.
	\end{equation}

	We next construct a coarse-graining of $[1,\xi)$ at scale $n$, i.e.\ let $y_0,y_1, y_2, \cdots$ be the sequence of points on 
	$[1,\xi)$ with $y_0=1$ and $d(y_i,1) = ni$.
	Then we have (see Figure \ref{f:tb})
	\begin{equation*}
	\E T_B(1,y_m) \leq   \sum_{i=0}^{m-1}\E T_B(y_i,y_{i+1})= \sum_{i=0}^{m-1} \E T([1,\xi),B,n, i)
	\end{equation*}
For a lower bound of $\E T_B(1,y_m)$, observe the following. From Remark \ref{rmk-sep}, the $(B+\eta)-$balls about $y_i$ disconnect
	$N_B([1,\xi))$ for all $i$, provided the coarse-graining scale $n$ is large enough. Hence there exists $B'$ (depending on $B, \eta$) such that the path attaining the weight $T_{B}(1,y_m)$ has disjoint (across $i$) sub-paths $\gamma_{i}$ contained in $N_{B}(y_{i},y_{i+1})$ with endpoints $u_i$, $v_{i}$ contained in $N_{B+\eta}(y_{i})$ and $N_{B+\eta}(y_{i+1})$ respectively, and such that there exist paths from $y_i$ to $u_{i}$ and $v_i$ to $y_{i+1}$ contained in $N_{B}(y_{i},y_{i+1})$ with word length bounded above by $B'$. (This statement follows from the last part of Remark \ref{rmk-sep}.)  Notice that, by arguing as in Lemma \ref{approxvel} it follows that the expected $\omega$-length of the maximum over  all paths of length $\leq B'$ from $y_{i}$ is at most $B''$ for some $B''$ independent of $i$. This implies that the expected length of the path between $u_i$ and $v_{i}$ described above is at least $\E T_B(N_B(y_i),N_B(y_{i+1}))-2B''$. It therefore follows that


		\begin{equation}\label{eq-minus}
	\E T_B(1,y_m) \geq   \sum_{i=0}^{m-1}\E T_B(N_B(y_i),N_B(y_{i+1})) -2mB''.
	\end{equation}

	\begin{figure}[h]
	\centering
	\includegraphics[width=\textwidth]{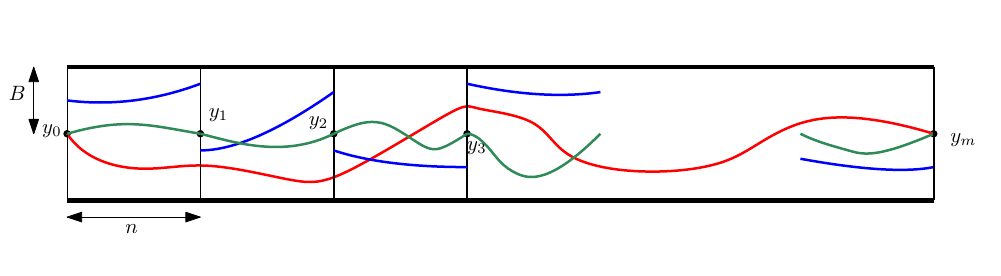}
	\caption{Approximating $T_{B}(1,y_{m})$ as in the proof of Proposition \ref{coarse-grainedCexists}. The red path denotes the path attaining $T_{B}(1,y_{m})$, i.e., the smallest FPP-length path connecting $y_0=1$ and $y_{m}$ that is contained in $N_{B}([y_0,y_m])$. The blue segments denote the paths attaining $T_{B}(N_{B}(y_i), N_{B}(y_{i+1}))$ and the green segments denote the paths attaining $T_{B}(y_{i},y_{i+1})$. Clearly, the blue segments are shorter (in FPP-length) than the green segments, and the FPP length of the red curve is sandwiched between the sum of lengths of the green segments and the blue segments.}
	\label{f:tb}
	\end{figure}

	Thus, $\frac{1}{mn}\E T_B(1,y_m)$ is sandwiched between the quantities
	$\frac{1}{mn}\sum_{i=0}^{m-1}\E T_B(y_i,y_{i+1})$ and $\big(\frac{1}{mn}\sum_{i=0}^{m-1}\E T_B(N_B(y_i),N_B(y_{i+1}))-\frac{B''}{n}\big)$.
	By  \eqref{eq-approxvel}, the difference between the latter two quantities is at most $(\frac{\ep}{2}+\frac{B''}{n})$.
	To prove the Proposition, it therefore suffices to show that the Cesaro averages
	$\frac{1}{m}\sum_{i=0}^{m-1}\E T_B(y_i,y_{i+1})$ and $\frac{1}{m}\sum_{i=0}^{m-1}\E T_B(N_B(y_i),N_B(y_{i+1}))$ converge as $m \to \infty$.
	
	To prove these convergence statements, we invoke
	Lemma \ref{ofreq-exist}: for every geodesic word $w$ of length $n$ and $\sigma = [1,\xi) \in \PP_\infty^\prime$, the ordered frequency $f_w (\sigma)$ exists.  
	Hence, using group invariance,
	
	\begin{equation}\label{eq-freqnb}
	\frac{1}{m}\sum_{i=0}^{m-1}\E T_B(N_B(y_i),N_B(y_{i+1})) \rightarrow
	\sum_{|w|=n} f_w (\sigma) \E T_B(N_B(1),N_B(e(w)))
	\end{equation}  as $m \to \infty$, and 
	
	\begin{equation}\label{eq-freq}
	\frac{1}{m}\sum_{i=0}^{m-1}\E T_B(y_i,y_{i+1}) \rightarrow
	\sum_{|w|=n} f_w (\sigma) \E T_B(1,e(w))
	\end{equation} 
	as $m \to \infty$.
	This completes the proof of the Proposition.
\end{proof}

Since $\ep \to 0$ as $n \to \infty$ in Lemma \ref{approxvel}, this gives
us the following immediate Corollary:

\begin{cor}\label{cor-coarseexists}
	For all $B > 0$ and all $\sigma = [1,\xi) \in \PP_\infty^\prime$, $$\bbar{v_B}([1,\xi),n)-\underline{v_B}([1,\xi),n ) \rightarrow 0$$ as $n \to \infty$.
\end{cor}

The main technical theorem of this section is the following:
\begin{theorem}\label{cvelexists}
	For every  $B>4 \delta$ and  $\xi \in \pG'$, the $B-$velocity ${v_B}(\xi))$  in the direction of $\xi$ exists.
\end{theorem}

\begin{proof}
It suffices, by Remark \ref{r:well-defined} to consider a geodesic ray $\sigma\in \PP^{\prime}_{\infty}$ in the direction $\xi \in \pG'$. Let us denote
$\sigma=\{1=x_0,x_1,\ldots, x_{N}, \ldots\}$. Recall that by definition of $\PP^{\prime}_{\infty}$, we have that for all $n\in d\N$ as above and $\epsilon>0$, there exists $N_0\in \N$ such that the following holds.
For all $N\geq N_0$ (such that $N$ a multiple of $n$), the fraction of ordered occurrences of each geodesic word $w$ with $|w|=n$ in the first $N$-length segment of $\sigma$ is in $[f_{w}(\sigma)-\epsilon, f_{w}(\sigma)+\epsilon]$.


 It therefore follows from the sandwiching argument of Proposition \ref{coarse-grainedCexists} that given $n\in \natls$ sufficiently large and $\ep>0$, there exists $N_0'$ such that for all $N\geq N_0'$ (again $N$ is a multiple of $n$) we have 
$$\frac{\E T_B(1,x_N)}{N} \in \big[\underline{v_B}([1,\xi),n)-\ep,\bbar{v_B}([1,\xi),n)+\ep\big].$$
Observe now that for $N_1,N_2\in \N$ with $|N_1-N_2|\leq n$ we also have that 
$$|\E T_{B}(1,x_{N_1})-\E T_{B}(1,x_{N_2})|\leq \E T_{B}(x_{N_1,X_{N_2}})\leq cn$$
for some absolute constant $c>0$. This implies that 
$$\frac{\E T_B(1,x_N)}{N} \in \big[\underline{v_B}([1,\xi),n)-\ep,\bbar{v_B}([1,\xi),n)+\ep\big]$$
holds for all $N$ larger than some $N''_{0}$, not merely the multiples of $n$.
Using Corollary \ref{cor-coarseexists}, this implies (by taking $n$ arbitrarily large) immediately that $\{\frac{\E T_B(1,x_N)}{N}\}_{N}$ is Cauchy, and hence 
$$v_{B}(\xi):=\lim_{N\to \infty} \frac{\E T_B(1,x_N)}{N}$$ 
exists, thus completing the proof.
\end{proof}

We are now ready to prove Theorem \ref{thm-vexists}.

\begin{proof}[Proof of Theorem \ref{thm-vexists}]
	Theorem \ref{cvelexists} shows that for all $B>0$, 
	${v_B}(\xi)$  in the direction of $\xi$ exists for a.e.\ $\xi \in (\pG,\nu)$. Hence by Lemma \ref{lowerupperconstant}, 
	${v_B}(\xi)$ is constant for a.e.\ $\xi \in (\pG,\nu)$.
	Letting $B \to \infty$, Theorem \ref{thm-vexists} is now an immediate 
	consequence of Proposition \ref{approxbycnbhd}.
\end{proof}

It is clear that $v(\xi)$ is upper bounded by $\int x d\rho(x)$. One can also show, along the lines of Lemma \ref{l:geodlen}, that $v(\xi)$ is bounded away from $0$ {uniformly in $\xi$}. This is done in the following lemma. 

\begin{lemma}
\label{l:positive}
There exist $\epsilon=\epsilon(\rho,\Ga)>0$ and $c=c(\epsilon)>0$ such that for any $n\in \N$ and any $x_n\in \Ga$ with $d(1,x_{n})=n$ we have
$$\P(T(1,x_n)\leq \epsilon n) \leq e^{-cn}.$$
\end{lemma}

\begin{proof}
By Lemma \ref{l:geodlen}, by choosing $R$ sufficiently large, it suffices to show that with probability at least $1-e^{-cn}$, each path $\gamma$ of length $\leq Rn$ connecting $1$ and $x_{n}$ satisfies $\ell_{\omega}(\gamma)\geq \epsilon n$. There are at most $D^{Rn}$ many such paths. Hence it suffices to show that for $\epsilon$ sufficiently small, the probability that any fixed path of length between $n$ and $Rn$ has $\omega$-length $\leq \epsilon n$ has probability upper bounded by $e^{-c(\epsilon)n}$ and $c(\epsilon)$ can be made arbitrarily large by making $\epsilon$ arbitrarily small. It suffices to prove the above statement for a fixed path of length $n$. Let $a(\epsilon)$ denote the probability the the weight of a specific edges is $\leq \sqrt{\epsilon}$. The probability described above can be upper bounded by $\P(\mbox{Bin}(n,a(\epsilon))\geq (1-\sqrt{\epsilon})n)$, and the desired conclusion follows, as in Lemma \ref{l:geodlen}, using Chernoff's inequality and noting that $a(\epsilon)\to 0$ as $\epsilon\to 0$.  
\end{proof}

It follows from Lemma \ref{l:positive} that $v(\xi)$ is uniformly bounded away from $0$.

\subsection{Special cases and examples}\label{sec-eg}

Let $g \in G$ be of infinite order. Then $g$ acts by North-South dynamics on $\pG$ with an attracting fixed point
(denoted $g^\infty$) and a repelling fixed point
(denoted $g^{-\infty}$). Note that the repelling fixed point of $g$ is the attracting fixed point of $g^{-1}$ and vice versa.

\begin{defn}\label{def-pole}
	An attracting point of an infinite order element in $\pG$ is called a pole.
\end{defn}

Note that the set of poles is countable; hence of measure zero with respect to the Patterson-Sullivan measure.
\begin{lemma}\label{velatpoles}
	For every pole $\xi$, the	velocity $v(\xi)$ exists.
\end{lemma}

\begin{proof}
	This is a reprise of an analogous argument in the Euclidean case (see e.g.\ \cite[Theorem A]{Kes93})  and we only sketch the argument. Let $\xi=g^\infty$.
	The sequence $\{g^n\}$ defines a quasigeodesic by the classification of isometries of a hyperbolic space \cite{gromov-hypgps}.  Clearly,
	$$\E T(1, g^{m+n}) \leq \E T(1,g^{m})+ \E T(g^{m}, g^{m+n})=\E T(1,g^{m})+ \E T(1, g^{n}),$$ where the last equality follows by group-invariance. The Lemma is now a consequence of Kingman's sub-additive ergodic theorem.
\end{proof}

An approach to proving Theorem \ref{thm-vexists} for more general directions using the subadditive ergodic theorem is provided in Appendix \ref{s:set}.

\subsubsection{FPP on $\Z$}\label{sec-fppz}
We  show first that  FPP on  Cayley graphs of $\Z$ with respect to different generating sets, but   with the same passage time distribution, may exhibit different speeds along the same direction.
This will be an ingredient for a counterexample in Section \ref{sec-free2} below. Consider the following two Cayley graphs of $\Z$: $\Ga_1=\Ga_1(\Z, \pm 1)$ and $\Ga_2=\Ga_2(\Z, \pm 1, \pm 2)$. Consider FPP on $\Ga_1$ and $\Ga_2$ with the same continuous passage time distribution $\rho$ with mean $\mu\in (0,\infty)$. let $a_{\rho}\in [0,\infty)$ and $b_{\rho}\in (0,\infty]$ denote the infimum and the supremum of support of $\rho$.

To distinguish between the two graphs, we shall denote the corresponding passage times by $T^{\Ga_1}(\cdot, \cdot)$ and $T^{\Ga_2}(\cdot, \cdot)$ respectively. Clearly $\frac{\E T^{\Ga_1}(0,n)}{n}\to \mu$ as $n\to \infty$. Notice that $d^{\Ga_2}(0,2n)=n$ and hence the following lemma gives an example illustrating the claim above. 

\begin{lemma}
	\label{l:pm2}
	If $2a_{\rho}<b_{\rho}$, then 
	$$\lim_{n\to \infty} \frac{\E T^{\Ga_2}(0,2n)}{n}<\mu.$$
\end{lemma}

\begin{proof}
	Observe that
	$$T^{\Ga_2}(0,2n)\leq \sum_{i=1}^{n} \min(X_i, Y_{2i-1}+Y_{2i})$$
	where $X_{i}$ and $Y_{i}$ are independent families of i.i.d. variables with distribution $\rho$. Therefore, it suffices to show that 
	$$\E[X_{i}-\min(X_i, Y_{2i-1}+Y_{2i})]>0.$$
	Clearly if $2a_{\rho}<b_{\rho}$, by assumption of continuity of $\rho$, there exists $\delta<\frac{b_{\rho}-2a_{\rho}}{4}$ such that $\P(Y_{i}\in [a_{\rho},a_{\rho}+\delta])\geq \delta$ and $\P(X_{i}\in [b_{\rho}-\delta,b_{\rho}])\geq \delta$. It therefore follows that 
	the non-negative variable $Z_{i}:=X_{i}-\min(X_i, Y_{2i-1}+Y_{2i})$ satisfies $\P(Z_i\geq \delta)\geq \delta^{3}$ and hence $\E Z_{i}\geq \delta^{4}$, completing the proof of the lemma. 
\end{proof}

\subsubsection{FPP on the free group}\label{sec-free2} We next give an example to show that there exists a hyperbolic group $G$ and a generating set $S$ such that 
\begin{enumerate}
	\item there exist $\xi_1, \xi_2 \in \pG$ such that $v(\xi_1), v(\xi_2)$ exist {\it but are unequal.}
	\item there exists $\xi \in \pG$ such that $v(\xi)$ does not exist.
\end{enumerate}

Let $G=F_2=\langle a, b\rangle$ be the free group on 2 generators. Fix $S=\{a^{\pm 1}, b^{\pm 1}, b^{\pm 2}\}$ to be the generating set and let $\Ga = \Ga(G,S)$. Let $\xi_1 = a^\infty, \, \xi_2 =  b^\infty$.
By Lemma \ref{velatpoles},  $v(\xi_1), v(\xi_2)$ exist.
By Lemma \ref{l:pm2}, $v(\xi_1) > v(\xi_2)$.\\

We now construct a direction $\xi \in \pG$ such that $v(\xi)$ does not exist. Let $w=a^{m_1}b^{n_1}a^{m_2}b^{n_2}\cdots$ be an infinite reduced word such that $m_i, n_i$ are defined recursively (as a tower function) by

\begin{enumerate}
	\item $m_1=1$,
	\item $n_i = 2^{2^{m_i}}$, for $i \geq 1$, 
	\item $m_{i+1} = 2^{2^{n_i}}$, for $i \geq 1$.
\end{enumerate}

Let $\xi \in \pG$ denote the boundary point represented by $w$.
Let $u_k, v_k$ denote the finite subwords of $w$ given by 
$$u_k=a^{m_1}b^{n_1}a^{m_2}b^{n_2}\cdots a^{m_k},$$ and
$$v_k=a^{m_1}b^{n_1}a^{m_2}b^{n_2}\cdots a^{m_k}b^{n_k}$$
Since the sequence $\{m_1,n_1,m_2,n_2, \cdots, m_k,n_k, \cdots \}$ grows like the tower function, the length of $u_k$ (resp. $v_k$) is dominated by $m_k$ (resp. $n_k$).  Further, since every vertex of 
$\Ga$ disconnects it, we have (by group-invariance),
$$\frac{T(1,u_k)}{T(1, a^{m_k})} \to 1, $$ and 
$$\frac{T(1,v_k)}{T(1, b^{n_k})} \to 1 $$ as $k \to \infty$.
Hence,  $$v(1,u_k) \to v(a^\infty) = v(\xi_1)$$ and 
$$v(1,v_k) \to v(b^\infty) = v(\xi_2)$$ as $k \to \infty$.
Since $v(\xi_1) \neq v(\xi_2)$, $v(\xi)$ does not exist.

\subsubsection{Graphs quasi-isometric to a tree}\label{sec-free3}  We modify the examples in Section \ref{sec-free2} above to show that the conclusions of Theorem \ref{thm-vexists} break down completely if we only require that the graph $\Ga$ is quasi-isometric to a Cayley graph $\Ga(G,S)$ of a hyperbolic group
(instead of being isometric to the latter). \\

\noindent {\bf Different velocities in different directions:} As in Section \ref{sec-free2} above, let $G = F_2=\langle a, b\rangle$ be the free group on 2 generators and let {$S=\{ a, b, a^{-1}, b^{-1}\}$.} Let $\Ga_a$ be the sub-tree
of $\Ga = \Ga(G,S)$ whose vertex set consists of group elements that can be represented by reduced words starting with $a$. We modify the Cayley graph 
$\Ga = \Ga(G,S)$ only on the sub-graph $\Ga_a$ of $\Ga$
 by introducing additional edges on $\Ga_a$ corresponding to generators 
 $\{a^{\pm 1}, a^{\pm 2}, b^{\pm 1}, b^{\pm 2} \}$. Let $\Ga'$ be the result of
modifying $\Ga$ thus. Note that $\Ga_a$ is a {quasi-isometrically} embedded subset
of $\Ga'$; hence the boundary $\partial \Ga_a$ of $\Ga_a$ (thought of as a 
hyperbolic metric space) embeds in the boundary $\partial \Ga'$ of $\Ga'$
(again regarded  as a 
hyperbolic metric space). {Note that there is a natural identification between $\partial \Gamma$ and $\partial
	\Gamma'$; so $\partial \Gamma_a$ may simply be regarded as a subset of
	both.}
 Then, the argument of Lemma \ref{l:pm2}
(and using the notation there) ensures that
$\forall \, \xi \in \partial \Ga_a \subset \partial \Ga'$, $v(\xi)$ exists,
and is less than $v(\xi')$, $\forall \, \xi' \in (\partial \Ga'
\setminus \partial \Ga_a)$. Note further that due to homogeneity of 
$\Ga_a$ and $(\Ga'
\setminus  \Ga_a)$ respectively,
 $v(\xi)$ is constant on $ \partial \Ga_a$; and 
$v(\xi')$  is constant on $ (\partial \Ga'
\setminus \partial \Ga_a)$. This provides an example where $v(\xi)$ is defined
for all $\xi \in \partial \Ga'$, but assumes different values on disjoint
positive measure subsets of $\partial \Ga'$.\\

\noindent {\bf Non-existence of velocity in any direction:} We now modify the example of Section \ref{sec-free2}(2) to construct a graph quasi-isometric to the regular 4-valent tree  $\Ga = \Ga(G,S)$ above, so that $v(\xi)$ does not exist in any direction. Construct $m_i, n_i \in \natls$ as in Section \ref{sec-free2} Modify $\Ga$ in each annulus
$$A_i = (N_{n_i} (1) \setminus N_{m_i} (1))$$ by introducing edges according to generators  $\{a^{\pm 1}, a^{\pm 2}, b^{\pm 1}, b^{\pm 2} \}$.
Let $\Ga''$ denote the modified graph. Then the argument in \ref{sec-free2}
shows that the velocity in any direction keeps oscillating between two distinct constants. Hence $v(\xi)$ does not exist in any direction $\xi\in 
\partial \Ga''$.

\section{Direction of $\om$-geodesic rays}\label{sec-dir}
Recall again our basic set-up: $G$ is a hyperbolic group and $\Ga = \Ga(G,S)$ is a Cayley graph with respect to a finite symmetric generating set. In Definition \ref{def-fpp} we have defined $\om-$geodesics between $x, y \in \Ga$. We would like to extend this notion to $\om-$geodesic rays. The Gromov boundary (resp.\ compactification) of $\Ga$ is denoted as $\pG$
(resp. $\hhat{\Gamma}$) (since $\pG$ is independent of the generating set $S$ we have not used $\partial \Gamma$ to denote the boundary of $\Ga$, using the suggestive notation $\pG$ instead).

\begin{defn}\label{def-fppray} 
	For  $\om \in \omegap$, a semi-infinite (resp.\ bi-infinite) path $\sigma_\omega$ is said to be an {\bf $\om-$geodesic ray (resp.\ a bi-infinite
	 $\om-$geodesic)} if every finite subpath of   $\sigma_\omega$ is an $\om-$geodesic.
	
	A path $\sigma$ is said to {\bf accumulate on} $\xi \in \pG$ if there exist $v_n \in  \sigma$ such that $v_n \to \xi$ as $n \to \infty$. 
\end{defn}

\noindent {\bf Defining directions of $\om-$geodesic rays:} We briefly motivate the notion of direction of an $\om-$geodesic ray. In Euclidean space $\R^n$, a direction $u$ is identified with an element of $T_I(0) \subset T_0(\R^n)$, the unit tangent sphere at $0 \in \R^n$ contained in the tangent space $T_0(\R^n)$ to $\R^n$ at $0$. Since tangent spaces are not available in our setup, we have to interpret $T_I(0)$ appropriately for $\Ga$. The exponential map 
from $T_0(\R^n)$ to $\R^n$ is a diffeomorphism sending  $tu$ (with $t \in \R_+$ and $u \in T_I(0)$) to the unique geodesic in $\R^n$ starting at $0$ and in the direction given by $u$.
This allows us to identify   $T_I(0)$ to the boundary $\partial \R^n$ given by asymptote classes of geodesics as in Lemma \ref{asymptosamept}. We shall thus parametrize directions of $\om-$geodesic rays in $\Ga$ by points $\xi \in \pG$. To this end we make the following definition. 

\begin{defn}\label{def-dir}
	An $\om-$geodesic ray $\sigma_\omega$ accumulating on $\xi\in \pG$ {\bf has direction $\xi$} if its only accumulation point  in $\pG$ is $\xi$.
\end{defn}

The main objective of this section is to establish that Definition \ref{def-dir} is indeed a natural definition, every $\omega$-geodesic ray has a direction (Theorem \ref{t:direxists}) and there exist $\om$-geodesic rays in each direction $\xi\in \pG$. (Theorem \ref{dirnexists}). The analogous statements for Euclidean lattices (where direction, as usual, is parametrized by points on the unit sphere) is due to Newman \cite{New95}, where it is proved under additional curvature assumptions on the limiting shape of random balls in dimension two. Recent partial progress without these assumptions has been made in \cite{H08, DH14, DH17} and a related result in terms of the Busemann functions have been established in \cite{AH16}. The hyperbolic geometry allows these results to go through without such assumptions in our case.

We shall say that  a sequence $x_n$ in $\Ga$ satisfies $x_n \to \infty$ as $n\to \infty$ if $d(x_n,1) \to \infty$ as $n\to \infty$. The main idea is to observe that if an $\omega$-geodesic ray starting from $o$ has two distinct accumulation points then they must have points $x_n,y_n\to \infty$ on them such that the word geodesic $[x_{n},y_{n}]$ passes through a bounded neighborhood of $o$. This would force $[x_n,y_n]_{\omega}$ to also almost surely pass through finite neighborhoods of $o$; which will lead to a contradiction. The first step of making this argument precise is the following proposition.


\begin{prop}\label{bt}
Let $ C \geq 0$ and $o \in \Gamma$ be given. Then  for a.e.\ $\om \in \omegap$, there exists $R_\om >0$ such that the following holds:\\
For every sequence $x_n, y_n \to \infty$ such that $[x_{n},y_n]$ passes through $N_C(o)$ (i.e., some word geodesic between $x_n$ and $y_n$ passes through $N_C(o)$), the $\om$-geodesic $[x_n,y_n]_{\om}$ passes through the  $R_\om-$neighborhood of $o$.
\end{prop}

The proof of this proposition adapts the proof of \cite[Theorem 1.3]{benjamini-tessera}, except that we do not assume $x_{n},y_{n}$ to lie on a fixed Morse geodesic passing though $o$, thus requiring some additional work.

For $A \subset \Gamma$ and  $p \in \Gamma$, define a {\bf nearest-point projection} of $p$ onto $A$ as $\Pi (p) = a$ if $d(p,a)= d(p,A)$.
Nearest-point projections onto geodesics (or more generally quasiconvex sets) are coarsely well-defined in the following sense (see the proof of \cite[Theorem 1.13, p405]{bh-book} and \cite[Lemma 2.20]{mahan-split}):
\begin{lemma}\label{lem-npp-coarse}
	Let $x \in \Gamma$ and $u, v \in \oxi$ be such that $d(x,\oxi)= d(x,u) = d(x,v)$.
	Then $ d(u,v) \leq 2 \delta$. More generally, given $\kappa>0$, there exists $C_0$ such that if $A \subset \Gamma$ is  $\kappa-$quasiconvex,
	then for any $x \in \Gamma$ and $u, v \in A$ with $d(x,A)= d(x,u) = d(x,v)$, we have $ d(u,v) \leq C_0$.
\end{lemma}

 We shall need the following geometric lemma, similar in content (and proof) to Lemma 2.2 and Lemma 3.1 of \cite{benjamini-tessera}.

\begin{lemma}\label{bt-lemma}
Let $\Ga$ be as above. Then there exist $A, C', R_0 > 0$ such that 
for all $R \geq R_0$ the following holds. Let $[x,y] \subset \Gamma$ be a geodesic and $o \in [x,y]$. Let $\sigma$ be a path joining $x, y$ such that $\sigma \cap N_{100R} (o) = \emptyset$. Then there exist $u, v \in \sigma$ such that $d(u,[x,y]) = R = d(v,[x,y])$, $[u,v] \cap N_{A} (o) \neq \emptyset$.
Further, 
$u, v$ satisfy the following properties in addition. Let $\sigma_{uv}$ be the subpath of $\sigma$ between $u, v$ and let $\Pi$ denote nearest-point projection onto $[x,y]$. Then $\Pi(u), \Pi(v) \in [x,y]$ lie on opposite sides
of $o$ with $d(\Pi(u),o) \geq 99R$,  $d(\Pi(v),o) \geq 99R$. Also,
\begin{enumerate}
\item \large{$\frac{\ell (\sigma_{uv})}{2R+d(\Pi(u), \Pi(v))} \to \infty$} as $R \to \infty$.
\item $d(o, \sigma_{uv}) \leq C' \ell (\sigma_{uv})$.
\end{enumerate}
\end{lemma}

The proof uses basic hyperbolic geometry and is postponed to Appendix \ref{sec-app}.
We are now ready to prove Proposition \ref{bt}.

\begin{proof}[Proof of Proposition \ref{bt}]
For the purpose of this proof, let us define $\mu:=\int xd\rho(x)$. Let $\Omega_1=\Omega_1(T)$ denote the set of all configurations such that $$\sup_{\gamma} (\ell_{\om}(\gamma)-T\mu \ell(\gamma)) <\infty$$
where the supremum is taken over all paths $\gamma$ passing through $N_C(o)$. We claim that, if $T$ is sufficiently large, $\P(\Omega_1(T))=1$.  
Indeed, since $N_C(o)$ is finite, it suffices to consider separately the paths passing through each vertex in $N_C(o)$. Clearly, there are at most $n|S|^{n}$ many paths of length $n$ passing through a fixed $o_1\in N_C(o)$ and by choosing $T$ sufficiently large and using Theorem \ref{t:subexp} it follows that the probability of any such path $\gamma$ satisfying  $\ell_{\om}(\gamma)\ge T\mu \ell (\gamma)$ is at most $e^{-2|S|n}$. Taking a union bound over all paths of length $n$ passing though $o_1$, followed by a union bound over all choices of $o_1$ and an application of Borel-Cantelli lemma completes the proof of the claim. From now on, we shall fix a full measure subset $\Omega_1$ satisfying the above property. 

Next, let $\Omega_2=\Omega_2(\epsilon)$ be the set of all configurations such that $$\inf_{\gamma} (\ell_{\om}(\gamma)-\epsilon \ell(\gamma)) >-\infty$$ where the infimum is taken over all paths $\gamma$ such that $d(o,\gamma) \leq C' \ell(\gamma)$ (where $C'$ is as in Lemma \ref{bt-lemma}). By the same argument as in the proofs of Lemma \ref{l:geodlen} and Lemma \ref{l:positive} (cf.\ \cite[Lemma 2.5]{benjamini-tessera}) there exists $\epsilon>0$ sufficiently small such that $\P(\Omega_2)=1$. Let us fix a full measure subset $\Omega_2$ satisfying the above property. 

We now show that  the full measure subset $\Omega':=\Omega_1 \cap \Omega_2$ satisfies the hypothesis in the statement of the proposition. We argue by contradiction. Suppose for $\omega\in \Omega'$, there exist $R_{n}\uparrow \infty$ and two sequences $x_n, y_n \to \infty$ such that $[x_{n},y_n]$ passes through $N_C(o)$ but $[x_n,y_n]_{\omega}$ avoids $N_{(100R_n+C)}(o)$. 

Let $\Pi_n$ denote nearest point projection onto $[x_n, y_n]$ and $o_n \in N_C(o) \cap [x_n,y_n]$. Since $N_C(o)$ is finite, we can assume after passing to a subsequence if necessary that $o_n = o'$ for all $n$. Hence, 
$[x_n,y_n]_{\omega}$ avoids $N_{(100R_n)}(o')$ for all $n$.
By Lemma \ref{bt-lemma}, there exist $A, C' \geq 0$ (depending only on $\delta$)  and $u_{n}, v_{n}$ on $[x_n,y_n]_{\omega}$ such that the following hold.
\begin{enumerate}
\item $d(u_n,[x_n,y_n]) = R_n = d(v_n,[x_n,y_n])$, 
\item $\Pi_n(u_n), \Pi_n(v_n) \in [x_n,y_n]$ lie on opposite sides
of $o$ with $d(\Pi_n(u_n),o') \geq 99R_n$,  $d(\Pi_n(v_n),o') \geq 99R_n$,
	\item \large{$\frac{\ell ([u_n,v_n]_{\omega})}{2R_n+d(\Pi_n(u_n), \Pi_n(v_n))} \to \infty$} as $n \to \infty$. 
	
	\normalsize
	In fact,  for some $\alpha>0$ depending only on the hyperbolicity constant $\delta$,
	\large $\ell([u_n,v_n]_{\omega}) \geq d(\Pi_n(u_n), \Pi_n(v_n))e^{\alpha R_n}$ 
	\item $d(o', [u_n,v_n]_{\omega}) \leq C' \ell ([u_n,v_n]_{\omega})$.
\end{enumerate}

Using the definition of $\Omega_2$ we can therefore write 
$$ \ell_\om([u_n,v_{n}]_{\omega}) \geq \epsilon \ell([u_n,v_{n}]_{\omega}) -r_1(\omega) \geq \epsilon B_{n} (2R_n+d(\Pi_n(u_n), \Pi_n(v_n)))-r_1(\omega)$$
where, $B_{n}\to \infty$ as $n\to \infty$.

The path between $u_n$ and $v_{n}$ obtained by concatenating $[u_n,\Pi_n(u_n)]$, $[\Pi_n(u_n),\Pi_n(v_n)]$ and $[\Pi_n(v_n),v_n]$  passes though $o'$ (by Property 2 above). Using the definition of $\Omega_1$ we also have 
$$ \ell_\om([u_n,v_{n}]_{\omega}) \leq  T\mu (2R_n+d(\Pi_n(u_n), \Pi_n(v_n))))+r_2(\omega).$$
Notice that $r_1$ and $r_2$ do not depend on $n$. Comparing these inequalities we get 
$(\epsilon B_{n}-T \mu) d(\Pi_n(u_n), \Pi_n(v_n)))\leq 2 T\mu R_{n}+r_1(\omega)+r_2(\omega)$, which is a contradiction for large enough $n$, since $d(\Pi_n(u_n), \Pi_n(v_n))\geq R_n$ and $B_{n}\to \infty$ as $n\to \infty$.
\end{proof}


We can now prove that almost surely all $\omega$-geodesic rays have direction. 

\begin{theorem}
\label{t:direxists}
For $o\in \Gamma$, for a.e. $\omega \in \omegap$, all $\om$-geodesic rays starting at $o$ {have} a direction $\xi\in \pG$.
\end{theorem}

\begin{proof}
	Since an $\om-$geodesic ray necessarily visits infinitely many $x \in \Gamma$, it has some accumulation point $\xi \in \pG$.
	
Let $C'\geq 0$ be fixed. We first show that almost surely there does not exist an $\om$-geodesic ray starting at $o$ such that it has two accumulation points $\xi,\eta \in \pG$ with $\langle\xi,\eta\rangle_{o} \leq C'$. Observe that there exists $C\geq 0$ (depending only on $C'$ and $\delta$) with the following property: if such an $\omega$-geodesic ray $\sigma$ existed then there would be points $x_{n},y_{n}\in \sigma$ with $x_n,y_n\to \infty$ such that (i) $y_{n}$ belongs to the restriction of $\sigma$ between $o$ and $x_n$ (ii) $[x_n,y_n]$ passes through $ N_C(o)$ (since we can choose $C = C' + 2\delta$, \cite[Chapter III.H.1]{bh-book}). Let $\Omega'=\Omega'(C)$ be as in Proposition \ref{bt}. By Proposition \ref{bt}, for all $\omega\in \Omega'$, $\sigma$ returns infinitely often to $N_C(o)$, implying that $\sigma$ is self-intersecting, a contradiction. Letting $C' \to \infty$ completes the proof of the theorem. 
\end{proof}

Next we show that for every $\xi\in \pG$, there is an $\omega$-geodesic ray with direction $\xi$.

\begin{theorem}\label{dirnexists}
	Fix $\xi \in \partial G$ and $x_{n}\in \Ga$ such that $x_{n}\to \xi$. For a.e.\ $\om \in \omegap$ the sequence of $\om-$geodesics $[o,x_n]_\om$ from $o$ to $x_n$ converges (up to subsequence) to an $\om-$geodesic ray $[o,\xi)_\om$ having direction $\xi$.
\end{theorem}

\begin{proof} The proof is similar to the previous one. Fixing $C'\geq 0$ we shall show that for a.e. $\omega$, any sub-sequential limit $[o,\xi)_{\om}$ of $[o,x_n]_\om$ (sub-sequential limits exist by a compactness argument) cannot accumulate at $\eta\in \pG$ with $\langle \xi, \eta \rangle _o \leq C'$. Observe again, that there exists $C\geq 0$ with the following property: if $[o,\xi)_{\om}$ had such an accumulation point, then (if necessary, passing to a subsequence), there will be points $y_n\to \infty$ such that $y_n\in [o,x_n]_{\om}$ and $[x_n,y_n]$ passes through $N_C(o)$ (choosing $C = C' + 2\delta$, as in the proof of Theorem \ref{t:direxists}). Note that the difference from the previous case is that $x_{n}$ does not necessarily lie on $[o,\xi)_{\om}$. Nevertheless, by considering $\omega\in \Omega'=\Omega'(C)$ given in Proposition \ref{bt} we see that $[y_n,x_n]_{\om}$ must pass through a point $z_n$ in a finite ($R_{\om}$) neighborhood of $o$. Since $y_{n}\in [o,x_n]_{\omega}$ it follows that $y_{n}\in [o,z_n]_{\omega}$ which implies that there exist infinitely many points on the finite union of $\om$-geodesics  $\cup_{o'\in N_{R_{\omega}}(o)}[o,o']_{\om}$, a contradiction. As before we finish the proof by taking $C'\to \infty$.
\end{proof}

\section{Coalescence}\label{sec-coalesce}
We shall prove in this section that for FPP on $\Ga$, semi-infinite geodesics in a fixed direction almost surely coalesce. This question is of  fundamental  importance in FPP on $\Z^d$, with progress being made under curvature assumptions \cite{New95, LN96} and with more recent progress using Busemann functions in \cite{H08, DH14, DH17, AH16}. Some of the finer questions have in recent years been addressed in the exactly solvable models of exponential directed last passage percolation on $\Z^2$ \cite{C11, FP05, Pim16, BSS17B, BHS18} but most of the fundamental questions remain open for FPP with general weights, even on $\Z^d$. For spaces with negative curvature, asymptotic coalescence is a folklore expectation due to the thin triangles condition of Gromov. For FPP on Cayley graphs of hyperbolic groups we shall establish this:  semi-infinite geodesics in the same boundary direction almost surely coalesce.

\subsection{Hyperplanes and their properties}\label{sec-hypgeoprel}
In this subsection we deduce some of the basic estimates from hyperbolic geometry that will be needed to prove coalescence.
Let $\xi \in \pG$ be a boundary point and let $[1,\xi)=\{1=x_0,x_1,\ldots,\}$
be a geodesic ray from $1$ to $\xi$.

\begin{defn}\label{def-hyperplane}
	Fix $\xi$, $[1,\xi)=\{1=x_0,x_1,\ldots,\}$ as above.
	For any $x_i \in [1,\xi)$ with $i>2 \delta$, define the {\bf elementary hyperplane} $H_i^e(=H_i^e(\xi))$ through $x_i$ as follows:
	$$H_i^e:= \{x \in \Ga \mid d(x,\oxi)= d(x,x_i) \}.$$
	Finally define the {\bf  hyperplane} $H_i(=H_i(\xi))$ through $x_i$ to be the 
	$2 \delta-$neighborhood of $H_i^e$: \\ $$H_i=N_{2\delta}(H_i^e).$$
	(See Figure \ref{hyperplane} below.)
\end{defn}

\begin{figure}[h]
	\centering

	\includegraphics[height=7cm]{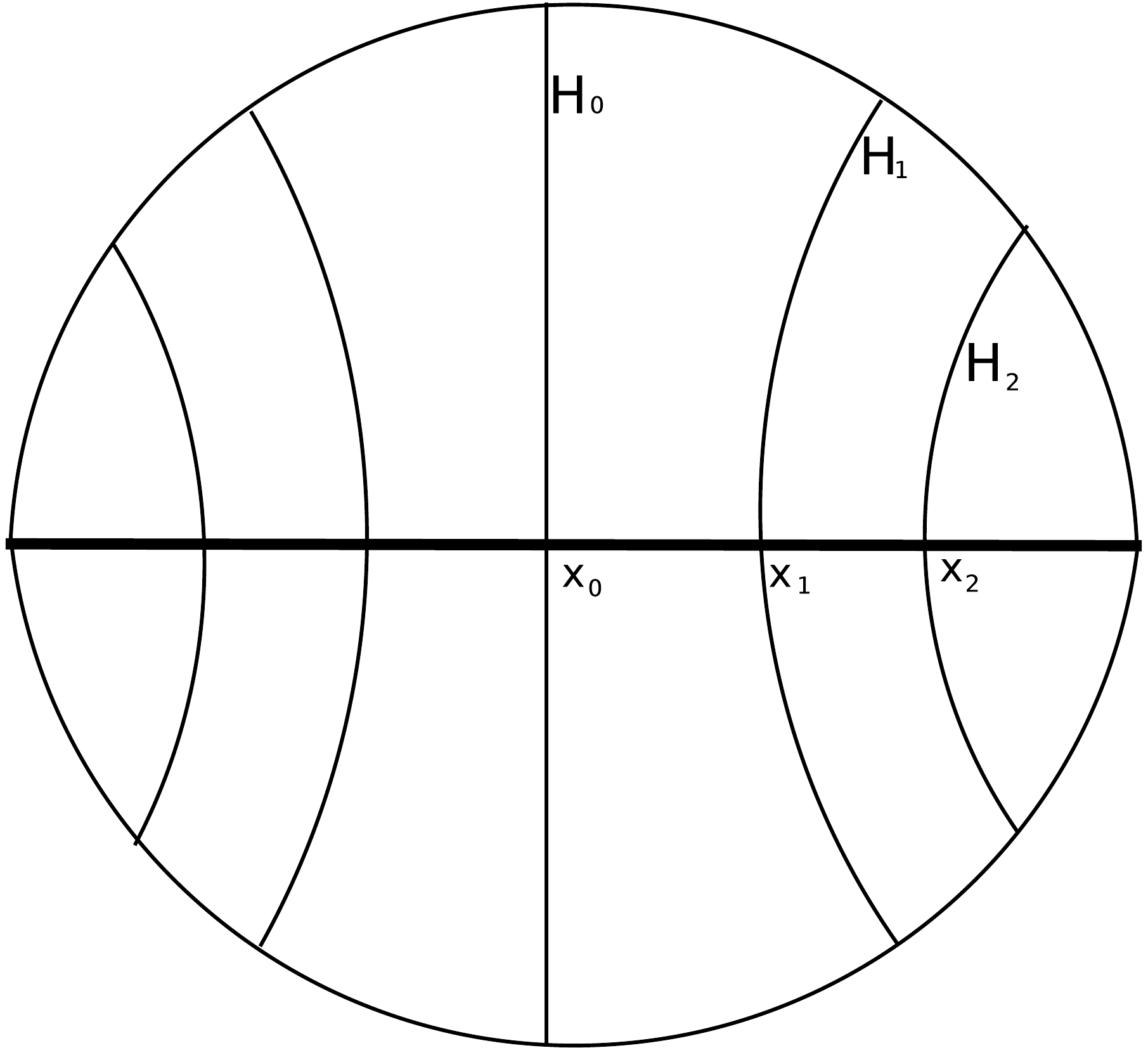}
	\bigskip
	
	\caption{Hyperplanes in the hyperbolic plane $\Hyp^2$ and their exponential divergence}
	\label{hyperplane}
\end{figure}

Equivalently, if $\Pi$ denotes nearest-point projection onto $\oxi$, $H_i^e = \Pi^{-1}(x_i)$ (see Lemma \ref{lem-npp-coarse}). 

\begin{rmk}{\rm
	The reason for thickening $H_i^e$ to $H_i$ in Definition \ref{def-hyperplane} above is to take care of the fact that nearest-point projections are only coarsely well-defined in the sense of Lemma \ref{lem-npp-coarse}.	}	
\end{rmk}

{	Define $H_i^{e+} = \cup_{j\geq i} H_j$, $H_i^{e-} = \Gamma \setminus H_i^{e+}$. Finally, define the {\bf half-space} $H_i^{+}$ (resp.\ $H_i^{-}$) to be  a
$4\delta-$neighborhood of $H_i^{e+}$ (resp.\ $H_i^{e-}$).} We record some of the properties of  hyperplanes. The following Lemma says that hyperplanes are quasiconvex and separate $\Ga$. Thus the resulting half-spaces may be regarded as nested.

\begin{lemma}\label{lem-hyperplane-props}
	For all $\xi$ and $i$, hyperplanes $H_i$ are $\delta-$quasiconvex.
	
	Hyperplanes $H_i$ separate $\Ga$, i.e.\ for half-spaces
	$H_i^+, H_i^-$ as above, 
	\begin{enumerate}
		\item[(a)] any two points of $H_i^+$ (resp.\ $H_i^-$) that can be connected by a path in $\Ga$ can be joined by a path lying in $H_i^+\cup H_i$ (resp.\ $H_i^-\cup H_i$),
		\item[(b)] any path from  $H_i^+$ to $H_i^-$ intersects $H_i$.
		\item[(c)] There exists $C_1$ depending on $\delta$ alone such that for all $k\geq C_1$,
		$H_{i+k} \subset H_i^+$, $H_{i+k}^+ \subset H_i^+$, $H_{i-k} \subset H_i^-$, $H_{i-k}^- \subset H_i^-$.
	\end{enumerate}
\end{lemma}

\begin{proof}
	Let $\Pi: \Gamma \to \oxi$ denote nearest-point projection.	Let $x \in H_i^e$.
	Then $\Pi(x)=x_i$.
	Every point  $y \in [x,x_i]$ satisfies $\Pi(y) = x_i$, and hence 
	$[x,x_i] \subset H_i^e$. For any $u, v \in H_i$, $[u,v] \subset N_\delta ([u,x_i]\cup [v,x_i]) \subset  N_\delta (H_i)$, i.e.\ 
	$H_i^e$ is $\delta-$quasiconvex. The first assertion follows (since a $2\delta-$neighborhood of a $\delta-$quasiconvex set is $\delta-$quasiconvex). 
	The same argument  establishes  quasiconvexity of 
$H_i^{e+}$ and $H_i^{e-}$.
	
	Now, let $\gamma$ be any path in $\Gamma$ joining $u, v \in H_i^+$. Let $\Pi'$ denote nearest-point projection onto $H_i$. By Lemma \ref{lem-npp-coarse},  nearest point projections onto hyperplanes are coarsely well defined. Projecting every maximal subpath
	of $\gamma$ contained in $H_i^-$ (if any) onto $H_i$ using  $\Pi'$ gives us a new path $\gamma'$ joining $u, v$ and contained in $H_i^+$, proving Item (a).
	
	Let $\sigma$ be a path in $\Gamma$ joining $u \in H_i^-$ to $v \in H_i^+$.
	Then $\Pi(u) = x_j$ for some $j<i$ and  $\Pi(v) = x_k$ for some $k>i$.
	Since nearest-point projections are coarsely Lipschitz \cite[Lemma 3.2]{mitra-trees}, it follows (essentially by continuity) that there exists
	$w \in \sigma$ such that  $\Pi(z) = x_i$, i.e.\ $w \in H_i$.  Item (b) follows.
	
	Item (c) now follows from the second assertion of Lemma \ref{lem-expdiv} below, which implies that $[x_i, x_{i+D}]$ is coarsely the shortest path between $H_i, H_{i+D}$ whenever $D$ is large enough.
\end{proof}

The following is a general fact 	\cite[Chapter 3.H.1]{bh-book}
(see also \cite[Lemma 3.1]{mitra-trees} and Figure \ref{hyperplane} above).

\begin{lemma}[Exponential divergence of hyperplanes]\label{lem-expdiv}
	For $D$ sufficiently large and any $i \in \N$, $H_i, H_{i+D}$ diverge exponentially, i.e.\ there exist $R_0 \geq 0$ and $C,\alpha > 0$ such that
	for $u \in H_i$ and $v \in H_{i+D}$ with $d(u, x_i) \geq R \geq R_0$ and 
	$d(v, x_{i+D}) \geq R \geq R_0$, any path joining $u, v$ lying outside the
	open $R-$neighborhood of $\oxi$ has length at least $CDe^{\alpha R}$.
	
	Further, $[u,x_i]\cup [x_i, x_{i+D}]\cup [x_{i+D},v]$ is a $(1,4\delta)-$quasigeodesic lying in a $4\delta-$neighborhood of $[u,v]$.
\end{lemma}

\begin{defn}\label{def-regionbt} For $C_1$ as in Lemma \ref{lem-hyperplane-props} and {a positive integer} $D\geq C_1$,  the {\bf region  $\B_{i,D}$ between $H_{i}$ and $H_{i+D}$} is defined to be 
\begin{equation}
\label{e:bidef}
\B_{i,D} = 
	H_i\cup	(H_i^+\cap H_{i+D}^-)\cup H_{i+D}.
\end{equation}	

\end{defn}

Let $a \in H_i$ and $b \in H_{i+D}$.
Let $\gamma$  be a path from $a$ to $b$ in the region $\B_{i,D}$ between hyperplanes. By Lemma \ref{lem-expdiv}, $[a,x_i]\cup[x_i, x_{i+D}]\cup [ x_{i+D},b]$ is a $(1,4\delta)-$quasigeodesic lying in a $4\delta-$neighborhood of the geodesic
$[a,b]$.  Then there exist $x_i', x_{i+D}'$ on $[a,b]$ such that $d(x_i, x_i') \leq 4 \delta$ and $d(x_{i+D}, x_{i+D}') \leq 4 \delta$.
Thus nearest point projections (of $x \in \Ga$) onto $[a,b]$ and $[a,x_i]\cup[x_i, x_{i+D}]\cup [ x_{i+D},b]$ lie in a uniformly bounded neighborhood of each other.

For each $z\in \gamma$, let $\Pi(z)$ denote a  nearest point projection onto  $[a,x_i]\cup[x_i, x_{i+D}]\cup [ x_{i+D},b]$. We assume below that $\Pi$ is  surjective (instead of being only coarsely so) to avoid cluttering the discussion.
Let $u$ be the last point on $\gamma$  projecting to $x_i$ and
let $v$ be the  first point projecting to $x_{i+D}$. {Let $R$ be a positive integer.}
Suppose that $\gamma$ lies outside the $R-$neighborhood of $[x_i, x_{i+D}]$. Equivalently, we might as well assume that $\gamma \cap N_R (\oxi) =\emptyset$.
(Strictly speaking, this is only coarsely true, but since $R$ will be assumed to be much larger than $\delta$, this will not affect our estimates below.)  
 Also, without loss of generality, assume that $\gamma \setminus N_R(H_i \cup H_{i+D})$ is connected (else, choose a connected component of
 $\gamma \setminus N_R(H_i \cup H_{i+D})$ joining $ N_R(H_i)$ to  $N_R( H_{i+D})$).
Let $u_1\in \gamma$  be the last 
point on $\gamma$ before $u$ such that $d(u_1,[a,b])=R$. Let 
$a_1 = \Pi(u_1)$. Then $a_1 \in [a,x_i]$. Similarly, let  $v_1\in \gamma$  be the first
point on $\gamma$ after $v$ such that $d(v_1,[a,b])=R$.
Let 
$b_1 = \Pi(v_1)$. Then $b_1 \in [x_{i+D},b]$. See Figure \ref{expdiv} below.

\begin{figure}[h]
	\centering

	\includegraphics[height=7cm]{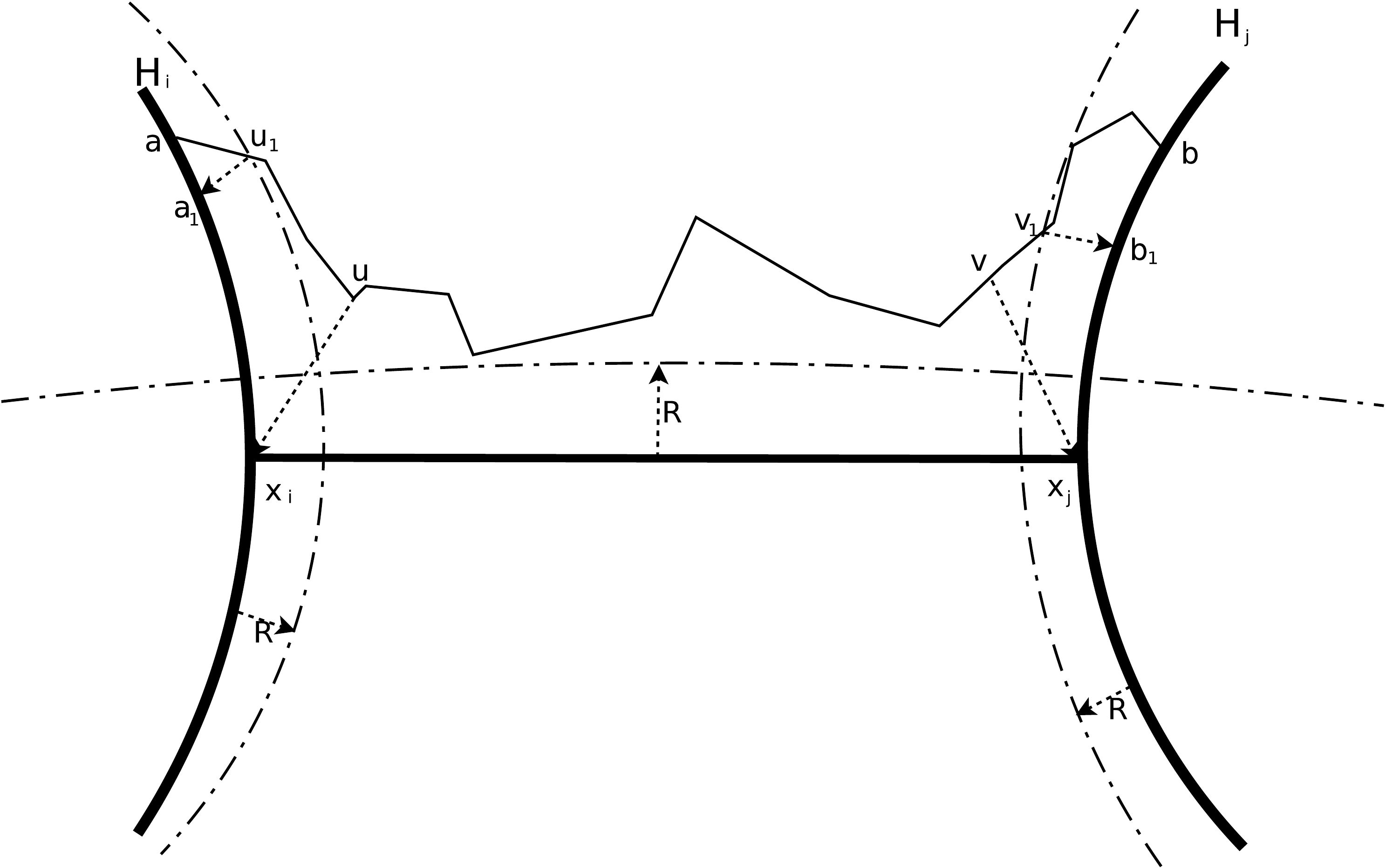}
	\bigskip
	
	\caption{Exponential divergence of paths away from geodesics.}
	\label{expdiv}
\end{figure}

Let $\gamma'$ be the subpath of $\gamma$ between $u_1$ and $v_1$.
By exponential divergence of geodesics \cite[p. 412-413]{bh-book} we have the following:

\begin{lemma}\label{lem-expdivlength}
	There exist $R_0 \geq 0$ and $\alpha > 0$ depending only on $\delta$ such that for $R \geq R_0$, $D \geq R_0$, $\gamma \cap N_R (\oxi) =\emptyset$
	and $\gamma', u_1, v_1, a_1, b_1$ as above,
	$$l(\gamma') \geq \bigg(d(a_1,x_i)+d(x_i, x_{i+D}) + d(x_{i+D},b_1)\bigg)e^{\alpha R}.$$
\end{lemma}

\subsection{Coalescence of $\om-$geodesic rays}\label{sec-coalesce-omega}
We now make the necessary definitions for $\om-$geodesics rays.

Recall the following from Definitions \ref{def-fppray} and \ref{def-dir}.
	For $\xi \in \partial G$ and $\omega\in \Omega$, a semi-infinite path $\gamma=\{x_{i}\}_{i\in \N}$ is called a semi-infinite $\omega$-geodesic (or an $\omega$-geodesic ray) in direction $\xi$ if every finite segment of $\gamma$ is an $\omega$-geodesic between the respective endpoints,  $x_{n}\to \xi$, and further for  any sequence of points $y_n$ on $\gamma$ with $d(y_n,1) \to \infty$, we have $y_n \to \xi$. We restate Theorem \ref{dirnexists} as follows for ready reference within this section:

\begin{lemma}
	\label{l:siexist}
	For each $o\in \Gamma$ and $\xi\in \partial G$,  for a.e. $\omega\in \Omega$ there exists a $\omega$-geodesic ray started at $o$ in direction $\xi$.
\end{lemma}

The next lemma asserts that for each fixed direction and starting point, the geodesic ray is almost surely unique. We shall prove it  and Theorem \ref{coalesce}) below together.

\begin{lemma}
	\label{l:siunique}
	For each $\xi\in \partial G$ and each $o\in \Ga$, for a.e. $\omega\in \Omega$, there exists a unique $\omega$-geodesic ray (denoted $[0,\xi)_{\omega}$) starting from $o$ in direction $\xi$.
\end{lemma}

We now state the main theorem of this section. \footnote{{Due to an oversight, the statement of this theorem was formulated imprecisely in the previous versions of this article including the published version. In particular, it was not explicitly made clear in the statement that, as mentioned in the introduction (and as will be clear from the proof), coalescence will hold for each fixed $\xi\in \pG$ and the almost sure event will depend on $\xi$. In fact, the result cannot be true in general for all $\xi$ simultaneously, as shown recently in \cite{BM24}.}}

\begin{theorem}\label{coalesce} For any fixed direction $\xi\in \pG$, for $\P$-a.e.\ $\om\in \Omega$, $o_1, o_2 \in G$, $\om-$geodesic rays $[o_1,\xi)_\om$ and $[o_2,\xi)_\om$ coalesce.
\end{theorem}

As mentioned before Lemma \ref{l:siunique} and Theorem \ref{coalesce} will be proved together. First we record the following simple geometric fact which will be useful throughout this section. Let $[1,\xi)=\{1=x_0,x_1,\ldots,\}$ denote a geodesic ray that corresponds to the direction $\xi$. For $i\geq 0$ and $D\in \N$, let $H_{iD}$ denote the hyperplane perpendicular to $[1,\xi)$ at $x_{iD}$. Then for $D$ sufficiently large, any $o\in \Ga$, and for a.e. $\omega \in \omegap$, every $\omega$-geodesic ray from $o$ to $\xi$ must cross $H_{iD}$ for all sufficiently large $i$
(see Lemma \ref{lem-hyperplane-props}). 



Before starting with the formal arguments let us present briefly the idea of the proof. Fix a direction $\xi\in \partial G$. Let us define the (almost surely well-defined) random object $\Upsilon_1$ (resp.\ $\Upsilon_2$) by setting $\Upsilon_1(\omega)=[o_1,\xi)_{\omega}$ (resp.\ $\Upsilon_2(\omega)=[o_2,\xi)_{\omega}$).  Without loss of generality we shall assume that $o_1=1$. Let us fix, as above, a geodesic ray $[1,\xi)=\{1=x_0,x_1,\ldots,\}$. For $D\in \N$, recall that $H_{iD}$ denotes the hyperplane perpendicular to $[1,\xi)$ at $x_{iD}$. Without loss of generality we shall also assume that $o_2$ is such that the projection of $o_2$ to $[1,\xi)$ is $1$; for if the projection to $\oxi$ is given by $\Pi(o_2) = x_i$, then we simply translate $x_i$ to $1$ and proceed with the argument below. It will be clear that the same proof can  be easily modified to include the more general case. Consider now the FPP geodesics $\Upsilon_1$ and $\Upsilon_2$. Notice that by Lemma \ref{lem-hyperplane-props}, if $D$ is sufficiently large, then for each $i\in \Z_{\geq 0}$, and for $j=1,2$, $\Upsilon_{j}$ must have a last exit from $H_{iD}$. We shall denote it as $U_j=U_j(i,D)$, where the dependence on $i$ and $D$ will be suppressed when clear from context. Further, there is a first entry point into $H_{(i+1)D}$ following $U_{j}$ which we shall denote as $V_{j}$. Let $\Upsilon_{1,i}$ (resp.\ $\Upsilon_{2,i}$) denote the restriction of $\Upsilon_1$ (resp.\ $\Upsilon_2$) between $U_1$ and $V_1$ (resp.\ $U_2$ and $V_2$). Notice that $\Upsilon_{1,i}$ (resp.\ $\Upsilon_{2,i}$) is the FPP geodesic between $U_1$ and $V_1$ (resp.\ $U_2$ and $V_2$) that is also further restricted to lie in the region $\B_{iD,D}$ ({as defined in \eqref{e:bidef}}). 

The basic idea behind the proof of Theorem \ref{coalesce} is to show that $\Upsilon_{1,i}$ and $\Upsilon_{2,i}$ intersect with probability bounded below independently of $i$. Theorem \ref{coalesce} (and also Lemma \ref{l:siunique}) will follow from the next proposition.



\begin{prop}
	\label{p:positive}
	For $D=D(\rho,\Ga)$ sufficiently large, there exists $\beta=\beta(D,\Ga, \rho)>0$ independent of $i$ such that for each $i\in \N$ there exists an event $B_{i}$ depending only on the configuration restricted to $N_{D/10}(\B_{iD,D})$ with the following properties:
	\begin{enumerate}
	\item[(i)] $\P(B_{i}) \geq \beta.$
	\item[(ii)] For all sufficiently large $i$, on $B_{i}$, every pair of $\omega$-geodesics rays started from either $1$ or $o_2$ in the direction $\xi$ intersect on $[x_{iD},x_{(i+1)D}]$.
	\end{enumerate}
\end{prop}

Postponing the proof of Proposition \ref{p:positive}, let us first quickly show how this lemma implies Lemma \ref{l:siunique} and Theorem \ref{coalesce}.

\begin{proof}[Proof of Lemma \ref{l:siunique}]
Observe that if two distinct $\omega$-geodesic rays intersect infinitely often, it implies that there exist multiple $\om$-geodesics between some pairs of vertices. By continuity of the passage time distribution $\rho$, $[u,v]_{\omega}$ is unique for each $u,v\in \Ga$ and for a.e $\omega$. Now, on this probability one event, if there exist distinct $\omega$-geodesic rays started from $o$ in direction $\xi$ then they cannot intersect infinitely often. Observe on the other hand the events $\{B_{2i}:i\geq 1\}$ as in Proposition \ref{p:positive} depend on disjoint sets of edges, hence are independent, which in turn using the Borel-Cantelli Lemma implies $\limsup_{i\to \infty} B_{i}:={\cap_{n=1}^{\infty} \cup_{i\ge n} B_{i}}$ is an almost sure event. On this almost sure event, distinct $\om$-geodesic rays from $o$ in direction $\xi$, if exist, must intersect infinitely often, leading to a contradiction. This completes the proof of the lemma.
\end{proof}

\begin{proof}[Proof of Theorem \ref{coalesce}]
As in the above proof, notice that $\limsup_{i\to \infty} B_{i}$ is an almost sure event. This in particular, implies, recalling the notation from the discussion preceding Proposition \ref{p:positive}, that  almost surely, for infinitely many $i$, $\Upsilon_{1,i}\cap\Upsilon_{2,i}\neq \emptyset$ and in particular, $\Upsilon_1$ and $\Upsilon_2$ almost surely intersect. By Lemma \ref{l:siunique} (actually this is merely a consequence of continuity of passage time distribution) it follows that if $\Upsilon_1$ and $\Upsilon_2$ intersect, they must coalesce (almost surely), completing the proof of the theorem.
\end{proof}


It remains to prove Proposition \ref{p:positive}. The proof is divided into three cases depending on the support of the passage time distribution $\rho$:
\begin{enumerate}
	\item[Case  (i):] $\rho$ has support $\subseteq [a,b]$ with $0<a<b<\infty$, 
	\item[Case\, (ii):] the passage time can take arbitrarily small values, but not arbitrarily large values,
	\item[Case\, (iii):] the passage time can also take arbitrarily large values.
\end{enumerate}

\noindent {\bf Case (i): Support of $\rho$ is bounded away from $0$ and $\infty$:}
This is the easiest case. Assume without loss of generality that the support is $[a,b]$ with $0<a<b<\infty$. By continuity of the passage time distribution this implies some mass around $a$ as well as $b$ (this is what we really need). Now, by the Morse Lemma \ref{lem-morse}, there exists $R$, depending on $a$, $b$  such that  $\Upsilon_1$ and $\Upsilon_2$ are contained in $N_{R}([1,\xi))$ and $N_{R}([o_2,\xi))$ respectively. Further, since $d(o_2,1)=d(o_2, \oxi)$, it follows that $[o_2,1]\cup \oxi$ is a quasigeodesic with constant depending only on $\delta$ (see \cite[Lemma 3.1]{mitra-trees} for instance) and hence we can assume, by increasing $R$ if necessary, that $\Upsilon_2$ is contained in  $N_{R}([o_2,1]\cup \oxi)$.
Hence there exists $c$ depending only on $a, {b}, \delta$ such that for each $i > cd(o_2,1) + c$, and any $\om$-geodesic ray from either $1$ or $o_2$ 
in the direction $\xi$ will have a sub-segment contained in $N_{R}([x_{iD},x_{(i+1)D}])$ with starting and ending points contained in $N_{R}(\{x_{iD}\})$ and $N_{R}(\{x_{(i+1)D}\})$ respectively. Proposition \ref{p:positive} in this case will follow from Lemma \ref{l:cptsupport} below.


\begin{lemma}
	\label{l:cptsupport}
	Suppose $\rho$ has support $[a,b]$ bounded away from $0$ and $\infty$, and let $R$ be as above. Then there exists $D=D(R,\Ga)$ sufficiently large, $\beta=\beta(D,R,\rho, \Ga)>0$ and an event $B_i$ for each $i$ depending only on the edges in $N_{R}([x_{iD},x_{(i+1)D}])$ with $\P(B_{i})\geq \beta$, such that for all $\omega \in B_{i}$ we have the following: for each $u\in N_{R}(\{x_{iD}\})$ and $v\in N_{R}(\{x_{(i+1)D}\})$, and a path $\gamma$ between $u$ and $v$ contained in $N_{R}([x_{iD}, x_{(i+1)D}])$, there exists a path $\gamma'$ between $u$ and $v$ that intersects $[x_{iD},x_{(i+1)D}]$ in at least $\frac{3D}{5}$ edges and $\ell_{\omega}(\gamma')\leq \ell_{\omega}(\gamma)$.  
\end{lemma} 

\begin{proof}
	Let $\epsilon>0$ be such that $a+\epsilon< a+10\epsilon <b-\epsilon$ and $\rho([a,a+\epsilon]), \rho([b-\epsilon,b])>0$. The geodesic subsegment of $[1,\xi)$ between $x_{iD}$ and $x_{(i+1)D}$ will be denoted by $[x_{iD},x_{(i+1)D}]$ as usual.
	Let $S=S_{i}$ be the union of 
	\begin{enumerate}
	\item the set of all edges on  $[x_{iD},x_{(i+1)D}]$, 
	\item the set of all edges that lie on some geodesic joining $x_{iD}$ to some point $u\in N_{R}(\{x_{iD}\})$, 
	\item and the set of all edges that lie on some geodesic joining $x_{(i+1)D}$ to some point $v\in N_{R}(\{x_{(i+1)D}\})$.
	\end{enumerate}
	 Let $T:=N_{R}([x_{iD},x_{(i+1)D}])\setminus S$. Let $B_{i}$ denote the event
	$$B_{i}:=\{\omega\in \Omega: \omega(e)\leq a+\epsilon~\forall e\in S,~\omega(e)\geq b-\epsilon~\forall e\in T\}.$$
	Clearly, there exists $\beta=\beta(D,R,\rho, \Ga)>0$ such that
	$\P(B_{i})\geq \beta$. Fix $\omega\in B_{i}$.  Consider $u\in N_{R}(\{x_{iD}\})$ and $v\in N_{R}(\{x_{(i+1)D}\})$, and a path $\gamma$ between $u$ and $v$ contained in $N_{R}([x_{iD}, x_{(i+1)D}])$ that has an intersection with $[x_{iD},x_{(i+1)D}]$ of cardinality less than $\frac{3D}{5}$. For $D$ sufficiently large ($D\gg R$) it follows that  $\gamma$ must intersect $T$ in at least $\frac{3D}{10}$ many edges.  This is because the number of edges in $N_{R}(\{x_{iD}\})$ or $N_{R}(\{x_{(i+1)D}\})$ can be made much smaller than $D$ by choosing $D\gg R$ (e.g. $D\sim e^{aR}$ for $a\gg 1$). Again by choosing $D$ sufficiently large (depending on $R$ and some fixed $\epsilon/2>\epsilon'>0$) we can ensure that   $\ell(\gamma)\geq (1-\epsilon')D$ (since the number of edges of $\gamma$ in  $N_{R}(\{x_{iD}\}$ or $N_{R}(\{x_{(i+1)D}\})$ can be made much smaller than $D$ by choosing $D\gg R$ as above). By definition of $B_{i}$ it follows that, $\ell_{\omega}(\gamma)\geq \frac{3D}{10}(b-\epsilon)+a(\frac{7D}{10}-\epsilon'D)$. Consider now the path $\gamma'$ obtained by concatenating $[u,x_{iD}]$, $[x_{iD},x_{(i+1)D}]$ and $[x_{(i+1)D},v]$. Clearly by definition of $B_{i}$, $\ell_{\omega}(\gamma')\leq  (a+\epsilon)(D+2R)$. Again, by choosing $D$ sufficiently large depending on $R,\epsilon$ and $\epsilon'$ we can ensure $\ell_{\omega}(\gamma')\leq \ell_{\omega}(\gamma)$, as desired.
\end{proof}

We shall prove statements similar to Lemma \ref{l:cptsupport} for the cases (ii) and (iii) as well. However, in those cases we are not afforded the luxury of knowing that the paths $\Upsilon_1$ and $\Upsilon_2$ are contained in some $R$ neighborhood of a fixed word geodesic. Hence we shall consider all FPP geodesics from some point in $H_{iD}$ to some point in $H_{(i+1)D}$ and show that on a positive probability event all those geodesics must spend sufficient time on the word geodesic segment $[x_{iD},x_{(i+1)D}]$. Although the statement of this lemma is identical for both of the cases, the proofs are slightly different and hence we shall state the results separately. \\

\noindent {\bf Case (ii): Support of $\rho$ is bounded away from  $\infty$ but not from $0$:}
We start with the case where $0$ is contained in the support of $\rho$. We have the following analogue of Lemma \ref{l:cptsupport} in this case. 

\begin{lemma}
	\label{l:zerosupport}
	Suppose that the support of $\rho$ contains 0 and is bounded away from  $\infty$.  Then there exists $D=D(\Ga)$ sufficiently large, $\beta=\beta(D,\rho, \Ga)>0$ and an event $B_i$ for each $i$ depending only on the edges in $N_{D/10}(\B_{iD,D})$ with $\P(B_{i})\geq \beta$, such that for all $\omega \in B_{i}$ we have the following: for each $u\in H_{iD}$ and $v\in H_{(i+1)D}$, and a path $\gamma$ from $u$ to $v$ contained in the region $\B_{iD,D}$, there exists a path $\gamma'$ between $u$ and $v$ that intersects $[x_{iD},x_{(i+1)D}]$ in at least $\frac{3D}{5}$ edges and $\ell_{\omega}(\gamma')\leq \ell_{\omega}(\gamma)$.  
\end{lemma} 

\begin{proof}
	Let us set $j=(i+0.01)D, j'=(i+0.99)D, k=(i+0.02)D, k'=(i+0.98)D$ and assume without loss of generality that these are all integers (taking $D$ a large enough multiple of 100). Let $S=S_{i,d}$ denote the set of all edges contained in $N_{d}([x_{iD},x_j]) \cup N_{d}([x_{j'},x_{(i+1)D}]) \cup [x_{iD},x_{(i+1)D}]$. 
	We call the edges in $\B_{k, 0.96D}$ (i.e., the edges in the region between $H_{k}$ and $H_{k'}$) except those on $[1,\xi)$ \textbf{unspecified edges}. Let us now define the following events. Let $\cC=\cC(\epsilon):=\{X(e)\leq \epsilon \, \, \forall e\in S\}$. Let $\cD=\cD(d)$ (resp.\ $\cD'$) denote the event that all $\om-$geodesics between pairs of points, one each in $H_{iD}$ and $H_{j}$ (resp.\ one each in $H_{j'}$ and $H_{(i+1)D}$) that are also contained in $\B_{iD,D}$ enter $N_{d}([x_{iD},x_{j}])$ (resp.\ $N_{d}([x_{j'},x_{(i+1)D}]$). Finally, let $\EE=\EE(\epsilon)$ denote the event that the total weight of unspecified edges along any path from a point in $N_{d}([x_{iD},x_{j}])$ to a point in $N_{d}([x_{j'},x_{(i+1)D}])$ that intersects $[x_{iD},x_{(i+1)D}]$ in less than $\frac{3D}{5}$ edges is at least $2 \epsilon D$. 
	
	We claim that, for appropriate choices of $d,\epsilon$ and for $D$ sufficiently large, $B_{i}:=\cC\cap \cD \cap \cD' \cap \EE$ satisfies the condition in the statement of the the lemma. Clearly, for $D$ sufficiently large compared to $d$, $B_{i}$ depends only on the edges in $N_{D/10}(\B_{iD,D})$, as required. First let us fix $\omega \in B_{i}$. Consider a path $\gamma$ as in the statement of the lemma. Clearly, without loss of generality we can replace, in $\gamma$, the last crossing of $H_{iD}$ to $H_{j}$ and the last crossing of $H_{j'}$ to $H_{(i+1)D}$ by the $\omega$-geodesics between the respective endpoints. With a minor abuse of notation we shall also call this path $\gamma$. If $\gamma$ intersects $[x_{iD},x_{(i+1)D}]$ in at least $\frac{3D}{5}$ edges then there is nothing to prove, so let us suppose the contrary. By definition of $\cD$ and $\cD'$ there exist points $w_1\in \gamma \cap N_{d}([x_{iD},x_{j}])$  and $w_2\in \gamma \cap N_{d}([x_{j'},x_{(i+1)D}]$. Let $w'_1$ and $w'_2$ denote the projections of $w_1$ and $w_2$ onto $[1,\xi)$ respectively. Let $\gamma'$ be the path obtained from $\gamma$ by replacing the segment between $w_1$ and $w_2$ (call it $\gamma_{*}$) by the path obtained by concatenating $[w_1,w'_1]$, $[w'_1,w'_2]$ (i.e., the restriction of $[1,\xi)$ between $w'_1$ and $w'_2$ ) and $[w'_2,w_2]$. Clearly it suffices to show that $\ell_{\omega}([w_1,w'_1])+\ell_{\omega}([w'_1,w'_2])+\ell_{\omega}([w'_2,w_2]) \leq \ell_{\omega}(\gamma_{*})$. By the definition of $\cC$ the left hand side above is at most $2\epsilon D$ (for $D$ sufficiently large compared to $d$) whereas by definition of $\EE$, the right hand side is bounded below by $2\epsilon D$. Hence the event $B_i$ defined above satisfies the conclusion of the lemma. 
	
	It remains to show a lower bound for $\P(B_{i})$. Observe first that for $D$ sufficiently large the events $\cC\cap \cD\cap \cD'$ and $\EE$ depend on disjoint sets of edges (Lemma \ref{lem-hyperplane-props}) and hence $\P(B_i)= \P(\EE)\P(\cC\cap \cD \cap \cD')$. Observe also that both $\cC$ and $\cD \cap \cD'$ are decreasing events in the weight configuration on $S$ and hence by the FKG inequality (Theorem \ref{t:FKG}), we have $\P(\cC\cap \cD \cap \cD')\geq \P(\cC)\P(\cD\cap \cD')$. By our assumption on $\rho$, for each $\epsilon>0$, there exists $\beta_1= \beta_1(D,d,\epsilon, \rho)>0$ such that  $\P(\cC)\geq \beta_1$. Hence to complete the proof, it suffices to show that $\P(\EE)$ and $\P(\cD\cap \cD')$ are also both bounded away from $0$. This is done below in Lemma \ref{l:middle1} and Lemma \ref{l:sidetoside}, invoking which the proof is completed.
\end{proof}

\begin{lemma}
	\label{l:middle1}
	In the set-up above, there exists $\epsilon>0$ such that $\P(\EE)\geq \frac{1}{2}$ for all $D$ sufficiently large.
\end{lemma}

\begin{proof}
	The proof will use essentially the same argument as in the proof of Lemma \ref{l:geodlen} and Lemma \ref{l:positive}. For any path $\gamma'$ satisfying  the condition in the definition of $\EE$, let $\gamma'_*$ denote a maximal subpath contained in the region $\B_{k,0.96D}$. Clearly, $\ell(\gamma'_{*})\geq 0.95D$ and the number of unspecified edges in $\gamma'_{*}\geq \ell(\gamma'_{*})-0.7 D$ (since the fraction of edges of $\gamma'$ in $N_{d}([x_{iD},x_j]) \cup N_{d}([x_{j'},x_{(i+1)D}])$ can be made arbitrarily small by choosing $D\gg d$). For all $h\geq 0.95D$, and a $\gamma'_*$ as above with $\ell_{\gamma'_{*}}=h$; it follows by choosing $\epsilon$ sufficiently small and using Theorem \ref{t:chernoff} as in the proof of Lemma \ref{l:positive}, that the probability that the total weight of unspecified edges in $\gamma'_{*}$ is $\leq 2\epsilon D$ is  bounded above by $e^{-c(\epsilon)h}$ where $c$ can be made arbitrarily small by making $\epsilon$ sufficiently small. Taking a union bound over all $\gamma'_*$ with $\ell(\gamma'_{*})=h$ (at most $|S|^{h}$ in number where $S$ is the generating set for $\Gamma$) and then taking a union bound over all $h\geq 0.95D$ the result follows.
\end{proof}

\begin{lemma}
	\label{l:sidetoside}
	In the set-up of proof of Lemma \ref{l:zerosupport}, there exists $d_0$ (independent of $D$) such that for all $d\geq d_0$ and all $D$ sufficiently large we have $\P(\cD \cap \cD')\geq \frac{1}{2}$.
\end{lemma}

\begin{proof}
	Observe first that for $D$ sufficiently large, $\cD$ and $\cD'$ depend on disjoint sets of edges (by Lemma \ref{lem-expdiv}), and hence are independent. So it suffices to show that $\P(\cD), \P(\cD')\geq \frac{3}{4}$. We shall only show this lower bound for $\P(\cD)$, the corresponding argument for $\cD'$ is identical.  
		In fact we shall show that $\P(\cD^c)\leq \frac{1}{4}$. For this, notice that in $\cD^{c}$ there must exist a geodesic from $H_{iD}$ to $H_{j}$ that does not intersect $N_{d}([x_{iD},x_{j}])$ where $d$ is sufficiently large to be appropriately chosen later. 
	
	We adapt the notation from Lemma \ref{lem-expdivlength} as follows. We replace $R$ in Lemma \ref{lem-expdivlength}  by $d$ here (as it will be determined later). For any points $x_p, x_l$ on the word geodesic $[1,\xi)$, the half-planes through them will be denoted by $H_p, H_l$ as usual. For a path $\gamma$  from $H_p$ to $H_l$,   $a(\gamma)$ (resp.\ $b(\gamma)$) will denote the initial (resp.\ terminal) point on
	$H_p$ (resp.\ $H_l$). Further,  suppose $\gamma$ lies outside the $d-$neighborhood of $[x_p, x_l]$, i.e.\ $\gamma \cap N_d (\oxi) =\emptyset$. We shall use nearest-point projection ({in the word metric}) $\Pi$ onto $[a(\gamma),x_p]\cup [x_p, x_l] \cup [x_l,b(\gamma)]$. Since $[a(\gamma),x_p]\cup [x_p, x_l] \cup [x_l,b(\gamma)]$ is  $4\delta-$close to the  geodesic $[a(\gamma),b(\gamma)]$ (see the  discussion before Lemma \ref{lem-expdivlength}), we shall for the purposes of this proof, not distinguish between $[a(\gamma),x_p]\cup [x_p, x_l] \cup [x_l,b(\gamma)]$ and $[a(\gamma),b(\gamma)]$.
	Let $u=u(\gamma)$ be the last point on $\gamma$  projecting to $x_p$ and 
	let $v=v(\gamma)$ be the  first point projecting to $x_l$. Let $u_1=u_1(\gamma)\in \gamma$  be the last 
	point on $\gamma$ before $u$ such that {$d(u_1,[a(\gamma),b(\gamma)])= d$. (Here, we are assuming that
		$d$ is an integer, and hence such a point $u_1$ must exist.)}
	 Let 
	$a_1(\gamma) = \Pi(u_1)$. Thus $a_1 (\gamma)\in [a(\gamma),x_p]$. Similarly, let  $v_1=v_1(\gamma)\in \gamma$  be the first
	point on $\gamma$ after $v$ such that $d(v_1,[a(\gamma),b(\gamma)])= d$. 
	Let 
	$b_1(\gamma) = \Pi(v_1)$; thus $b_1 \in [x_l,b(\gamma)]$.
	
	We shall specialize now to the case {$p=iD$ and $l=j=(i+0.01)D$ (where $p, l$ are defined in the beginning of the previous paragraph, and $j$ is as in the proof of Lemma \ref{l:zerosupport})}.
	For $j_1,j_2\geq 0$, let $\cF_{j_1,j_2}$ denote the event that  a $\omega$-geodesic $\gamma$ as above exists such that $d(a_1(\gamma),x_{iD})\in [j_1d, (j_1+1)d]$ and {$d(b_1(\gamma),x_{j})\in [j_2d, (j_2+1)d]$}. Clearly,
	$$\P(\cD^c)\leq \sum_{j_1,j_2\geq 0} \P(\cF_{j_1,j_2}).$$
	
	We want to approximate the event $\cF_{j_1,j_2}$ by taking a union bound over all possible locations of $u_1=u_1(\gamma)$ and $v_1=v_1(\gamma)$ as described above. Observe that, for a path $\gamma$ satisfying the conditions in the definition of $\cF_{j_1,j_2}$, by definition $a_1(\gamma)\in \big(N_{(j_1+1)d}(x_{iD})\setminus N_{j_1d-1}(x_{iD})\big)$ and $b_1(\gamma)\in \big( N_{(j_2+1)d}(x_{j})\setminus  N_{j_2d-1}(x_{j})\big)$. Therefore, $u_1(\gamma)\in \big(N_{(j_1+2)d}(x_{iD})\setminus N_{(j_1-1)d-1}(x_{iD})\big)\cap N_{d}(H_{iD})$ and $v_1(\gamma)\in \big(N_{(j_2+2)d}(x_{j})\setminus N_{(j_2-1)d-1}(x_{j})\big)\cap N_{d}(H_{j})$ where the set $N_{(j_k-1)d-1}(\cdot)$ is interpreted as the empty set for $j_{k}=0,1$. Now, let
 $\cg_{j_1,j_2}$ denote the event that there exist $u_1\in \big(N_{(j_1+2)d}(x_{iD})\setminus N_{(j_1-1)d-1}(x_{iD})\big)\cap N_{d}(H_{iD})$  and $v_1\in \big( N_{(j_2+2)d}(x_{j})\setminus  N_{(j_2-1)d-1}(x_{j})\big)\cap N_{d}(H_{j})$, we have $\ell(\Upsilon(u_1,v_1))\geq \frac{1}{2}e^{\alpha d} d(u_1,v_1)$ where $\alpha$ is as in Lemma \ref{lem-expdivlength}. Now for $d$ sufficiently large and $D$ sufficiently large depending on $d$, it follows that for all $u_1,v_1$ as above we have $d(u_1,v_1)\geq \frac{1}{2}(j_1+j_2+4)d$
	(by Lemma \ref{lem-expdiv}) and also  $\ell([a_1,x_{iD}]\cup [x_{iD},x_{j}]\cup [x_{j},b_1]) \geq \frac{1}{2}d(u_1,v_1)$ (again by Lemma \ref{lem-expdiv}). Using Lemma \ref{lem-expdivlength}, it follows that for such choices of $d$ and $D$ we have $\cF_{j_1,j_2}\subseteq \cg_{j_1,j_2}$ and hence $\P(\cF_{j_1,j_2})\leq \P(\cg_{j_1,j_2})$. Observe now that the total number of pairs of $u_1,v_1$ as above is bounded above  by $|S|^{(j_1+j_2+4)d}$ and hence using Lemma \ref{l:geodlen} we get $\P(\cg_{j_1,j_2}) \leq  |S|^{(j_1+j_2+4)d} e^{-c(j_1+j_2+4)d}$
 and by choosing $d$ (and hence $e^{\alpha d}$) sufficiently large we can make the $c$ above arbitrarily large uniformly in $j_1,j_2\ge 0$. It follows that for suitable choices of $d$ and $D$ we get $\P(\cg_{j_1,j_2})\leq e^{-c(j_1+j_2+4)d}$ for some $c>0$. Summing over all $j_1,j_2\geq 0$ we get the desired upper bound on $\P(\cD^c)$. This completes the proof of the lemma. 
\end{proof}

\noindent {\bf Case (iii): Support of $\rho$ is noncompact:}
Let us now turn to case (iii), where $X(e)$ can take arbitrarily large values. The corresponding lemma in this case is the following. 

\begin{lemma}
	\label{l:infinitysupport}
	Suppose the support of $\rho$ is non-compact. Then there exist $D=D(\Ga,\rho)$ sufficiently large, $\beta=\beta(D,R,\rho, \Ga)>0$ and an event $B_i$ for each $i$ depending only on the edges in $N_{D/10}(\B_{iD,D})$with $\P(B_{i})\geq \beta$ such that for all $\omega \in B_{i}$ we have the following: for each $u\in H_{iD}$ and $v\in H_{(i+1)D}$, and a path $\gamma$ from $u$ to $v$ contained in the region $\B_{iD,D}$, there exists a path $\gamma'$ between $u$ and $v$ that intersects $[x_{iD},x_{(i+1)D}]$ in at least $\frac{3D}{5}$ edges and $\ell_{\omega}(\gamma')\leq \ell_{\omega}(\gamma)$. 
\end{lemma} 

\begin{proof}
	We shall use the same notations as in the proof of Lemma \ref{l:zerosupport}. Let $d$ be such that the conclusion of Lemma \ref{l:sidetoside} holds (notice that Lemma \ref{l:sidetoside} did not use the condition that support of $\rho$ contained $0$), and let $\cD, \cD'$ be as before. Let $\EE'=\EE'(M)$ denote the event that for all edges in the region between $H_{k}$ and $H_{k'}$ (except those on $[1,\xi)$) that are within a distance $100d$ of $[1,\xi)$, we have $X(e)\geq M$. Let $\cC'$ denote the set of all $\omega\in \Omega$ such that for all $w_1\in N_{d}([x_{iD},x_{j}])$ and for all $w_2\in N_{d}([x_{j'},x_{(i+1)D}])$ we have $\ell_{\omega}([w_1,w'_1])+\ell_{\omega}([w'_1,w'_2])+\ell_{\omega}([w'_2,w_2]) \leq 2D \E X(e)$ (recall that $w'_1$ {and} $w'_2$ are projections of $w_1$ and $w_2$ respectively onto $[x_{iD},x_{(i+1)D}]$). Finally, let $\EE''$ denote the event that the total weight of unspecified edges (as defined in the proof of Lemma \ref{l:zerosupport}, although they are not all unspecified in the changed context) along any path from a point in $N_{d}([x_{iD},x_{j}])$ to a point in $N_{d}([x_{j'},x_{(i+1)D}]$ that intersects $[x_{iD},x_{(i+1)}D]$ in less than $\frac{3D}{5}$ edges is at least $2D \E X(e)$.
	
	Our claim now is that $B_{i}:=\cC'\cap \cD \cap \cD' \cap \EE'\cap \EE''$ does our job. We first show that the appropriate conditions are satisfied for this event. Fix a path $\gamma$ as in the statement of the lemma. As in the proof of Lemma \ref{l:zerosupport}, we shall assume without loss of generality that in $\gamma$, the last crossing of $H_{iD}$ to $H_{j}$ and the last crossing of $H_{j'}$ to $H_{(i+1)D}$ are the $\omega$-geodesics between the respective endpoints. As before, let $w_1\in \gamma \cap N_{d}([x_{iD},x_{j}])$  and $w_2\in \gamma \cap N_{d}([x_{j'},x_{(i+1)D}])$ and let $w'_1$ and $w'_2$ be their respective projections. By the definition of $\cC'$, it suffices to show that the $\omega$-length of the segment of $\gamma$ between $w_1$ and $w_2$ is at least $2\E X(e)D$, which is guaranteed by the definition of $\EE''$. 
	
	Now to show the lower bound for $\P(B_i)$, observe first that $\cC'\cap \cD\cap \cD'$ and $\EE' \cap \EE''$ depend on disjoint sets of edges for $D$ sufficiently large (Lemma \ref{lem-hyperplane-props}) and hence they are independent.  So it suffices to  bound from below the respective probabilities separately. For $\cC'$ observe that if $D$ is sufficiently large, we have  $\ell([w_1,w'_1])+\ell([w'_1,w'_2])+\ell([w'_2,w_2]) \leq 1.5D$
hence using Theorem \ref{t:subexp} we get that for each fixed $w_1$ and $w_2$ we have $\ell_{\omega}([w_1,w'_1])+\ell_{\omega}([w'_1,w'_2])+\ell_{\omega}([w'_2,w_2]) \leq 2D \E X(e)$ with failure probability at most $e^{-cD}$. Taking a union bound over all pairs of $(w_1,w_2)$ (these are polynomially many in $D$  for a fixed $d$) we get $\P(\cC')\geq \frac{3}{4}$. As mentioned above,  the proof of Lemma \ref{l:sidetoside} remains valid in this set-up also and hence we get $\P(\cC'\cap\cD\cap \cD')\geq \frac{1}{4}$.

As for $\EE '\cap \EE''$, observe that $\P(\EE')\geq \beta_2(M,d,D,\rho)>0$ for each $M$ by our hypothesis on $\rho$. Observe now that any path satisfying the condition in the definition of $\EE''$ must either have $50 d$ many edges that are set to weight at least $ M$ by $\EE'$ or have at least $\frac{1}{1000}e^{\alpha d}D$ many unspecified edges (provided $d$ and $D$ are sufficiently large) by Lemma \ref{lem-expdivlength}. By choosing $M \geq \frac{2 D \E X(e)}{50 d}$, we can ensure that the condition in the definition of $\EE''$ is satisfied in the former case. In the latter case, we repeat the argument in the proof of Lemma \ref{l:middle1}. Observe that if $d$ is chosen sufficiently large, then the proof of Lemma \ref{l:middle1} implies that with probability at least $\frac{1}{2}$, every path $\gamma$ satisfying the conditions in the definition of $\EE''$ with at least $\frac{1}{1000}e^{\alpha d}D$ many unspecified edges will have $\ell_{\omega}(\gamma) \geq 2D \E X(e)$. Finally noticing that $\EE'$ and $\EE''$ are both increasing events in the weight configurations of the unspecified edges, we invoke the FKG inequality (Theorem \ref{t:FKG}) to conclude that $\P(\EE' \cap \EE'')\geq \beta_3>0$. This completes the proof of the lemma. 
\end{proof}

We are now in a position to complete the proof of Proposition \ref{p:positive}.

\begin{proof}[Proof of Proposition \ref{p:positive}]
Depending on the support of $\rho$ consider the event $B_{i}$ defined in Lemma \ref{l:cptsupport}, Lemma \ref{l:zerosupport} or Lemma \ref{l:infinitysupport}, and let $D$ be sufficiently large so that the conclusion of those lemmas hold. The event $B_i$ in each case depends on the configuration restricted to the edges in $N_{D/10}(\B_{iD,D})$ and satisfies the required probability lower bound. Observe that, for all sufficiently large $i$ and on $B_{i}$, every $\omega$-geodesic ray from from either $1$ or $o_2$ (that must cross from {$H_{iD}$} to $H_{(i+1)D}$) must intersect $[x_{iD},x_{(i+1)D}]$ in at least $\frac{3D}{5}$ many edges. Clearly this implies that any two such geodesic ray must intersect on $[x_{iD},x_{(i+1)D}]$ completing the proof of the proposition.
\end{proof}

\section{Linear growth of variance along word geodesics}\label{sec-linvar}
The aim of this section is to use the technology in the previous section to prove that under suitable conditions on the distribution $\rho$ on edges, the variance of the length of the FPP geodesic between points on a semi-infinite  geodesic ray
$[1,\xi)$ in a fixed direction $\xi$ grows linearly with the word distance on $\Gamma$. Fix $\xi\in \partial G$  and a geodesic ray $[1,\xi)=\{x_0=1,x_1,x_2,\ldots\}$  in $\Ga$. The following is the main result
of this section.

\begin{theorem}
	\label{t:linvar}
	In the above set-up, there exist $0<C_1<C_2<\infty$ such that 
	$$ C_1 n\leq \mathrm{Var} (T(1,x_n))\leq C_2 n.$$
\end{theorem}

Because the FPP geodesic $\Upsilon(1,x_{n})$ is expected to remain close to the word geodesic (see Section \ref{sec-wqg}), the linear growth of variance was  conjectured by Benjamini, Tessera and Zeitouni \cite[Question 5]{bz-tightness}, \cite[Section 4]{benjamini-tessera}. Theorem \ref{t:linvar} above proves their conjecture. As already pointed out in the introduction, in contrast, in the Euclidean setting, the variance is expected to grow sub-linearly in all dimensions with exponent strictly less than $1$, but the best known result so far for a general FPP on $\Z^d$ is an upper bound of $O(\frac{n}{\log n})$ of the variance \cite{bks-sublinear} (see also \cite{BR08,DHS14}) . 

Linear growth of variance (or the related behavior of diffusive fluctuations) has often been observed in constrained models of first passage percolation and its variants. Gaussian fluctuations or linear variance have been shown for first passage percolation across thin cylinders, and in certain one-dimensional graphs \cite{CD13,Ahl15} (see also \cite{DPM17,BB17} for similar results and \cite{BG18} for variance bounds for passage times across on-scale cylinders). Unlike these results, for FPP on hyperbolic groups, the FPP-geodesic is not restricted to a uniformly bounded neighborhood of the word geodesic (except for the special case when the support of $\rho$ is bounded away from $0$ and $\infty$). So we cannot consider FPP geodesics restricted to lie in a thin cylinder. The upper bound for the variance in Theorem \ref{t:linvar} will follow from a standard Poinc\'are inequality. For the lower bound, \cite{benjamini-tessera} already speculated that results like
Proposition \ref{bt} could be useful for showing that the variance grows linearly. We shall use for effective and quantitative version of such results, obtained in Section \ref{sec-coalesce}, and show that a linear number of vertices contributed a uniformly positive amount to the variance, thereby getting a lower bound of the matching order. This is philosophically similar to the proof of \cite[Proposition 7.2]{BSS17A}.\\ 

\noindent {\bf Upper bound:}
As  mentioned above, the upper bound in Theorem \ref{t:linvar} will use a standard argument due to Kesten \cite{Kes93},  who used it to obtain a linear upper bound on variance in Euclidean FPP. Kesten's argument is rather general and the proof of (1.13) in \cite{Kes93} shows the following in our set-up. 

\begin{prop}
	\label{p:ub}
	There exists $C>0$ such that $\mathrm{Var} (T(1,x_n))\leq C\E [\ell(\Upsilon(1,x_n))]$.
\end{prop}

The upper bound in Theorem \ref{t:linvar}  follows immediately from Proposition \ref{p:ub} together with Corollary \ref{l:expgeodlen}. \\

\noindent {\bf Lower Bound:}
Fix $n\in \N$. Consider the passage time $T_{n}:=T(1,x_n)$. Let $D=D(\Ga,\rho)$ be such that the conclusion of Lemma \ref{l:cptsupport}, Lemma \ref{l:zerosupport} or Lemma \ref{l:infinitysupport} holds depending on the support of $\rho$. Recall that $\B_{i,k}$ denotes the region between the hyperplanes $H_{i}$ and $H_{i+k}$. For $i\geq 1$, let $\cg_{i}$ denote the $\sigma$-algebra generated by the random variables $\{X(e):e\in \B_{(2i-1)D,D}\}$. Note that when $D$ is sufficiently large, these sets are mutually disjoint (as $i$ varies). Let $\cg_{*}$ denote the $\sigma$-algebra generated by the remaining edge weights, i.e, 
$$\left\{X(e):e\in \Ga \setminus (\cup_{i=1}^{\infty} \B_{(2i-1)D,D})\right\}.$$
We define a filtration $\{\cf_{i}\}_{i\geq 0}$ by setting $\cf_0=\emptyset$, $\cf_1=\cg_{*}$, and for $i\geq 2$, set $\cf_{i}$ to be the $\sigma$-algebra generated by $\cup_{j=2}^{i}\cg_{i-1}\cup \cg_{*}$. 
Consider the Doob Martingale of $T_{n}$ with respect to this filtration $\{\cf_{i}\}_{i\geq 0}$ given by 
$$M_{i}=\E[T_{n}\mid \cf_{i}].$$ It is a standard fact (see, e.g.\  \cite[Section 3.3]{fppsurvey}) using the orthogonality of martingale difference sequences that
\begin{equation}
\label{e:vardecomp}
\mbox{Var} (T_{n})=\E \left[ \sum_{i=1}^{\infty} \mbox{Var}(M_{i}\mid \cf_{i-1})\right].
\end{equation} 
This type of decomposition of variance is a standard method of proving upper and lower bounds of the variance in models of first and last passage percolation. The lower bound in Theorem \ref{t:linvar}  is an immediate consequence of \eqref{e:vardecomp} together with the following lemma. 

\begin{lemma}
	\label{l:positivevar}
	There exists $c=c(\rho,\Ga,D)>0$ such that for each $2\leq i\leq \frac{n}{4D}$ we have 
	$\mathrm{Var}(M_{i}\mid \cf_{i-1})\geq c$.
\end{lemma}

\begin{proof}
	Let $2\leq i\leq \frac{n}{4D}$ be fixed. We shall divide the set of all edges in $\Ga$ into three parts. Let $S=S_{i}$ denote the set of all edges whose weights generate the $\sigma$-algebra $\cf_{i-1}$. Let $V=V_{i}$ denote the set of all edges in $\cup_{j\geq i} \B_{(2j-1)D,D}$. Finally let $U=U_{i}$ denote the set of all edges in $\Ga \setminus (S_{i}\cup V_{i})$, i.e., the set of all edges in $\B_{(2i-3)D,D}$. We shall write any $\omega\in \Omega$ as $\omega=(\omega_{S}, \omega_{V}, \omega_{U})$ where $\omega_{S}, \omega_{V}$, and $\omega_{U}$ are the restrictions of $\omega$ to $S_{i}, V_{i}$ and $U_{i}$ respectively. Condition on $\cf_{i-1}$ (i.e., fix $\omega_{S}$ for the rest of this proof). Using the standard fact that $\mbox{Var}{X}\geq \frac{1}{4}\E[X-X']^2$ when $X$ and $X'$ have the same distribution and are defined on the same probability space we have 	
$$\mbox{Var}(M_{i}\mid \cf_{i-1}) \geq \frac{1}{4} \int \left( \int (T_n(\omega_{S}, \omega_{V}, \omega_{U})-T_{n}(\omega_{S},\omega_{V},\omega'_{U})~d\omega_{V}\right)^2~d\omega_{U}~d\omega'_{U},$$ 
where $\omega'_{U}$ is an copy of $\omega_{U}$ (i.e., an assignment of $\rho$-distributed i.i.d.\ random weights on the edges $U_{i}$) which need not be independent of $\omega_{U}$. In fact we shall consider an appropriate coupling of $(\omega_{U},\omega'_{U})$ below. To see why the above equation is true, note that for a fixed $\omega_{S}$, $\int (T_n(\omega_{S}, \omega_{V}, \omega_{U})~d\omega_{V}$ and $\int T_{n}(\omega_{S},\omega_{V},\omega'_{U})~d\omega_{V}$ are two coupled copies of $\E[T_{n}\mid \cf_{i-1}]$.

To prove the lemma, it now suffices to show that there exists a coupling of $(\omega_{U},\omega'_{U})$ such that under this coupling there is a subset $\cB$ of configurations $(\omega_{U},\omega'_{U})$ with $\P(\cB)\geq c_1>0$, and on $\cB$
\begin{equation}
\label{e:condpositive2}
\int (T_n(\omega_{S}, \omega_{V}, \omega_{U})-T_{n}(\omega_{S},\omega_{V},\omega'_{U})~d\omega_{V} \geq c_2>0.
\end{equation}
Indeed, \eqref{e:condpositive2} will imply $\mbox{Var}(M_{i}\mid \cf_{i-1}) \geq 
\frac{c_1c_2^2}{4}$, and establish the desired result. We establish \eqref{e:condpositive2} in Lemma \ref{l:condpositive} below, invoking which finishes the proof of Lemma \ref{l:positivevar} and hence that of Theorem \ref{t:linvar}.
\end{proof}

Before establishing \eqref{e:condpositive2}, let us explain the basic idea. Roughly we want to show that once we have fixed $\omega_{S}$ and $\omega_{V}$ there is still enough randomness in the configuration $\omega_{U}$ such that we can ensure that for two coupled copies $(\omega_{U},\omega'_{U})$ with the same marginal we can ensure with uniformly positive probability that $(T_n(\omega_{S}, \omega_{V}, \omega_{U})-T_{n}(\omega_{S},\omega_{V},\omega'_{U})$ is uniformly bounded away from $0$. To show this we use the technology from Sec. \ref{sec-coalesce} to ensure that with a positive probability $\omega_{U}$ can be chosen in such a way that the geodesic $\Upsilon(1,x_{n})$ intersects the segment $[x_{(2i-3)D},x_{(2i-2)D}]$. We then obtain $\omega'_{U}$ from $\omega_{U}$ by resampling the weights on $[x_{(2i-3)D},x_{(2i-2)D}]$ while keeping the other weights fixed. On the positive probability event that each changed co-ordinate in $\omega'_{U}$ is strictly smaller than the corresponding co-ordinate in $\omega$ we get the desired decrease in $T_{n}(\omega_{S},\omega_{V},\omega'_{U})$ compared to $T_n(\omega_{S}, \omega_{V}, \omega_{U})$.  We now make this argument rigorous in the following lemma. 

%

\begin{lemma}
\label{l:condpositive}
In the set-up of proof of Lemma \ref{l:positivevar} we have the following: there exists a coupling $(\omega_{U},\omega'_{U})$ with both marginals $\rho^{\otimes U}$ such that under this coupling there is a subset $\cB$ of configurations $(\omega_{U},\omega'_{U})$ with $\P(\cB)\geq c_1>0$, and on $\cB$
$$\int (T_n(\omega_{S}, \omega_{V}, \omega_{U})-T_{n}(\omega_{S},\omega_{V},\omega'_{U})~d\omega_{V} \geq c_2>0.$$
\end{lemma}
	
\begin{proof}

Let $I$ denote the set of edges on $[x_{(2i-3)D},x_{(2i-2)D}]$, and let $U^*=U_{i}\setminus I$. Let $\omega_{U^*}$ (resp.\ $\omega_{I}$) denote the restriction of $\omega_{U}$ to $U^*$ (resp.\ $I$). Let $a\geq 0$ denote the essential infimum of $\rho$, i.e., the smallest point in the support of $\rho$. For $\epsilon>0$, let $\cA(\epsilon)$ denote the set of all possible weight configurations $\omega_{U^*}$ such that in the environment $\omega=(\omega_{S},\omega_{V},\omega_{U^*},\omega_{I})$ the geodesic $\Upsilon(1,x_{n})$ passes through at least one edge in $I$ for all choices of $\omega_{V}$ provided $\omega_{e}\leq a+\epsilon$ for all $e\in I$. If $D$ is sufficiently large and $\beta$ is as in Lemma \ref{l:cptsupport}, Lemma \ref{l:zerosupport} or Lemma \ref{l:infinitysupport} (depending on the support of $\rho$) we have 
	$$ \P(\cA(\epsilon))+D\rho([a,a+\epsilon]) \geq \beta.$$
Indeed, notice that by Lemma \ref{l:cptsupport}, Lemma \ref{l:zerosupport} or Lemma \ref{l:infinitysupport} we have with probability at least $\beta$,
$\omega_{U}$ is such that every geodesic from $H_{(2i-3)D}$ to $H_{(2i-2)D}$ passes through at least one edge in $I$. Further, for each $\omega_{U}$ for which the above happens we either have $\omega_{e}<a+\epsilon$ for some $e\in I$ or we have that the projection $\omega_{U^*}\in \cA(\epsilon)$. Clearly, since $\rho$ does not have an atom at $a$, it follows that by choosing $\epsilon$ sufficiently small we get $\P(\cA(\epsilon))\geq \frac{\beta}{2}$. Since $\rho$ is assumed to be continuous it follows that for each $\epsilon>0$ sufficiently small there exists $0<\epsilon'< \epsilon'' <\epsilon$ such that $h(\epsilon,\epsilon',\epsilon'')=\min\{\rho([a,a+\epsilon']), \rho([a+\epsilon'',a+\epsilon])\}>0$. 

Let us now choose the coupling $(\omega_{U},\omega'_{U})$ as follows: Let $\omega_{U^*}=\omega'_{U^*}$ and let $\omega'_{I}$ denote an independent copy of $\omega_{I}$. Let $\cB=\cB(\epsilon,\epsilon',\epsilon'')$ denote the set of all configurations $(\omega_{U},\omega'_{U})$ such that $\omega_{U^*}=\omega'_{U^*}\in \cA(\epsilon)$, $\omega_{e}\in [a+\epsilon'',a+\epsilon]$ for all $e\in I$ and $\omega_{e}'\in [a,a+\epsilon']$ for all $e\in I$. 

Clearly, for appropriate choices of the parameters we have $\P(\cB(\epsilon,\epsilon',\epsilon''))\geq \frac{\beta h(\epsilon,\epsilon',\epsilon'')^2}{2}=c_1>0$. For $(\omega_{U},\omega'_{U})\in \cB$,  a.e., $\omega_{V}$ (and the fixed choice of $\omega_{S}$) we finally need
to establish:
\begin{equation*}
(T_n(\omega_{S}, \omega_{V}, \omega_{U})-T_{n}(\omega_{S},\omega_{V},\omega'_{U})\geq \epsilon''-\epsilon'=c_2>0.
\end{equation*} 
Indeed, observe that, changing $\omega_{U}$ to $\omega'_{U}$ decreases the length of any path by at least $\ell c_2$ where $\ell$ is the number of edges of $I$ used by the path. By definition of $\cB$, the geodesic must pass through at least one edge of $I$ and hence we get the above equation, completing the proof of the lemma.
\end{proof}

\section{Discussion and Future Directions}
\label{s:future}
We have, in this paper, investigated some of the fundamental questions for first passage percolation on the Cayley graph of a hyperbolic group. In contrast to $\Z^d$, some of the results, e.g., the existence of velocity along almost every direction required much more work owing to the more complicated geometry of the underlying graph, whereas the hyperbolic geometry helped us resolve some other problems that are well-known to be difficult in Euclidean FPP (e.g.\ coalescence of geodesics). To maintain transparency of exposition, we have often worked with sub-optimal arguments and many of our results can possibly be strengthened. We finish with a discussion of some of these and a few of the many remaining open questions. 

First of all, we have worked with a rather strong assumption on the passage time distribution $\rho$, which one should be able to relax to a large extent. It would be interesting to know the optimal conditions on $\rho$ under which the results like Theorem \ref{thm-vexists} or Theorem \ref{t:linvar} hold. In particular, a proof of the existence of velocity (Theorem \ref{thm-vexists}) that directly appeals to the subadditive ergodic theorem (see Appendix \ref{s:set}) is worth investigating as it would likely provide weaker moment conditions needed for the existence of velocity. However, as we need the exponential tail on the passage time distributions for the other results in this paper, we did not pursue this direction here.

Observe also that  Theorem \ref{thm-vexists} only gives a convergence in mean for first passage times along word geodesics to almost every boundary direction. Using Theorem \ref{t:linvar} one can immediately upgrade this to an in probability convergence. Although we did not pursue this direction in this paper, one can fairly easily upgrade this to an almost sure convergence, by either developing stronger concentration inequalities as in \cite{Kes93}, or appealing directly to Kingman's theorem as described in Appendix \ref{s:set} below. 

After establishing the law of large number and order of fluctuations, the next natural question is to ask for the scaling limit for $T(1,x_{n})$ for $x_{n}$ along a word geodesic ray. Comparing with results from \cite{CD13, Ahl15}, one would expect a positive answer to the following question which was conjectured also in \cite{bz-tightness}.

\begin{qn}[Central Limit Theorem]\label{qn-clt} {\rm Does $\frac{T(1,x_{n})-\E T(1,x_{n})}{\sqrt{{\rm Var}(T(1,x_{n}))}}$ converge weakly to a standard Gaussian variable?}
\end{qn}

More on the geometric side of things, by Theorem \ref{thm-vexists}, 
existence and continuity of velocity $v(\xi)$ on $\pG$ implies that $v(\xi)$ is constant for all $\xi \in \pG$, whereas we have also given examples to show that this need not be true in general. It is clear that if $G$ is free of rank $m$ and the generating sets are of the form $\{a_i^{\pm j} \vert i= 1\cdots m; \, j = 1, \cdots k\}$ for some $k \geq 1$, then $v(\xi)$ exists and is  constant
for all $\xi \in \pG$. Are there other examples? Towards this, we propose:
\begin{qn}\label{rigidity} {\rm Find conditions on a hyperbolic group $G$ such that $v(\xi)$ exists and is continuous (and is hence constant) on $\pG$?}
\end{qn}

What we have in mind here is the following: if $B-$neighborhoods of  $[1,\xi)$, $\xi \in \pG$, are all nearly isometric (in a suitable sense), then $v(\xi)$ is constant. Can one deduce restrictions on the geometry of $\Ga$ from 
continuity of velocity $v(\xi)$?

\begin{rmk}[Cayley graphs versus arbitrary hyperbolic graphs]\label{rmk-ref}
In Section \ref{sec-free3}, we have shown that Theorem \ref{thm-vexists} fails dramatically when we replace the Cayley graph of a hyperbolic group by a graph quasi-isometric to it. However, the arguments in Sections 
\ref{sec-dir}, \ref{sec-coalesce}, and \ref{sec-linvar} are purely geometric in nature and do not use group-invariance at any point. In particular, they hold true for any bounded degree hyperbolic graph.  As pointed out at the end of the Introduction, the only difference lies in setting up the statements.  For a Cayley graph, the identity
element is chosen as a preferred base-point, whereas, for an arbitrary
Gromov-hyperbolic graph with uniformly bounded degree, any point can be chosen as a base-point. As the boundary of any hyperbolic graph is independent of the base-point,  the group-independent arguments of Sections \ref{sec-dir}, \ref{sec-coalesce} and \ref{sec-linvar} go through.
\end{rmk}

\appendix
\section{}\label{sec-app}
We provide in this section the postponed proofs of Theorem \ref{t:iidtail}, Proposition \ref{p:frequency} and Lemma \ref{bt-lemma}.

\begin{proof}[Proof of Theorem \ref{t:iidtail}]
	For $j\geq 0$, let us denote $Y_{i,j}=(X_{i}-M)1_{\{{X_{i}\in [2^jM, 2^{j+1}M]}\}}$. Let $\epsilon>0$ be fixed and let us choose $L>0$ sufficiently large. We shall show that by choosing $M=M(\epsilon, L)$ sufficiently large we can ensure 
	\begin{equation}
	\label{e:smallj}
	\P\left(\sum_{i=1}^n Y_{i,j} \geq \epsilon 2^{-(j+1)}n \right) \leq e^{-Ln}
	\end{equation}
	for $j\leq \log n$. Observe that by the sub-Gaussian tails of $X_i$ we have 
	\begin{equation}
	\label{e:bigj}
	\P\left(\sum_{j\geq \log n}\sum_{i=1}^n Y_{i,j} \neq 0 \right) \leq ne^{-cn^2M^2}, 
	\end{equation}
	and hence together with \eqref{e:bigj}, \eqref{e:smallj} completes the proof. 
	Notice now that for a fixed $j\leq \log n$, 
	$$\P\left(\sum_{i=1}^n Y_{i,j} \geq \epsilon 2^{-(j+1)}n \right) \leq \P\left(\sum_{i=1}^n 1(Y_{i,j}>0) \geq \epsilon M^{-1} 4^{-(j+1)}n \right)$$
	Observe that $\sum_{i=1}^n 1(Y_{i,j}>0)$ is a $\mbox{Bin}(n,p_{j,M})$ variable where $p_{j,M}=\P(X_i\in [2^jM,2^{j+1}M]) \leq C_1e^{-c_14^{j}M^2}$. It therefore follows by a Chernoff inequality that for $M$ sufficiently large the above probability is bounded above  by 
	$$\exp\left (-\epsilon M^{-1} 4^{-(j+1)}n (c_14^j M^2-\log C_1+ \log (\epsilon M^{-1} 4^{-(j+1)})) \right).$$ 
	Denoting $q_j:=\epsilon M^{-1} 4^{-(j+1)}$ it follows that the above probability is bounded above  by 
	$$\exp \left(-n \biggl(\frac{cM\epsilon}{8}-\log C_1+ q_j\log q_j\biggr)\right)$$
	and the proof is completed by noting that $x\log x$ is bounded away from $-\infty$ for $x\in [0,1]$ and choosing $M$ sufficiently large.
\end{proof}

\noindent
\textbf{Proof of Proposition \ref{p:frequency}:}
The proof  of Proposition \ref{p:frequency} below is fairly standard  in the theory of finite Markov chains, but we provide it for completeness. For this we need to consider the vector valued Markov chain $\{\mathbf{X}_n\}_{n\geq 1}$ obtained from the Markov chain $\{X_{i}\}_{i\geq 1}$ where $\mathbf{X}_i=(X_{(i-1)k+1},\ldots , X_{ik})$ for $i\geq 1$. Let $P^{(k)}$ denote the transition matrix of its chain. Also let us say an element $\mathbf{x}=(x_1,x_2,\ldots, x_{k})\in \Sigma^{k}$ is admissible if $P(x_i,x_{i+1})>0$ for all $i=1,2,\ldots, k-1$. Notice that if a state $\mathbf{x}$ is not admissible then almost surely the word $\mathbf{x}$ never occurs in the trajectory of the chain $\{X_{i}\}_{i\geq 1}$. We first need the following lemma. 

\begin{lemma}
\label{l:aperiodic}
For $k$ a multiple of $d$, each recurrent components of $P^{(k)}$ is aperiodic. Further, every admissible state in $\Sigma^{k}$ belongs to a recurrent component of $P^{(k)}$.  
\end{lemma}

\begin{proof} 
Let $\mathbf{x}\in \Sigma^{k}$ denote an admissible state starting with some $x_{1}\in \Sigma^{(k)}$. As $k$ is a multiple of $d$  and chain $P$ is irreducible with period $d$, starting from $\mathbf{x}$ any admissible state in $\Sigma^k$ starting with $x_{i}$ can be reached in $P^{(k)}$ for all $x_{i}$ in the $d$-periodic orbit of $x_1$ in $P$ (i.e., all $x_{i}$ that can be reached from $x_1$ in multiples of $d$ steps). Clearly, these states form a recurrent component, and hence all admissible states belong to one such recurrent component. It now remains to show that the recurrent components are all aperiodic. For this, notice that since the chain $P$ has period $d$, and $k$ is a multiple of $d$, it follows that for all sufficiently large $i$,
$$\P(X_{ik+1}=x_1\mid (X_1,X_2,\ldots, X_{k})=\mathbf{x})>0.$$
This immediately shows that $\P(\mathbf{X}_{i}=\mathbf{x}\mid \mathbf{X}_1=\mathbf{x})>0$ for all $i$ sufficiently large, thus showing that the component of $\mathbf{x}$ is aperiodic, as desired.
\end{proof}

We can now complete the proof of Proposition \ref{p:frequency}.

\begin{proof}[Proof of Proposition \ref{p:frequency}]
As $k$ is a multiple of $d$, clearly if $\mathbf{x}$ is not admissible or does not start with a state in the $d$-periodic orbit of $x$ (in $P$), $N_{n}(\mathbf{x},x)=0$ and there is nothing to prove. So let us assume the contrary. Let $\underline{\mathbf{\pi}}$ denote the unique invariant measure of the irreducible aperiodic recurrent component of $P^{(k)}$ containing $\mathbf{x}$. We shall show that the conclusion of the proposition holds with $(\mathbf{x},x)=\underline{\mathbf{\pi}}(\mathbf{x})$ (notice that the reason to retain the dependence on $x$ is that the starting state $x$ uniquely determines which recurrent component a particular realization of the vector valued chain $\{\mathbf{X}_n\}_{n\geq 1}$ belongs to). Let $x$ and $\mathbf{x}$ be as above and let $\epsilon>0$ be fixed. We shall show that 
	\begin{equation}
	\label{e:toshow}
	\limsup_{n\to \infty} |\frac{N_n(x_1,x)}{n}-\underline{\mathbf{\pi}}(\mathbf{x})|\leq \epsilon
	\end{equation}
	almost surely, which gives the desired result. 
	
	It is a standard result from the theory of (not necessarily reversible) finite Markov chains  that for every irreducible aperiodic finite Markov chain with transition matrix $Q$ and stationary distribution $\pi$ there exists $\alpha=\alpha(Q)<1$ such that for all $x$ in the state space we have
	\begin{equation}
	\label{e:TV}
	||Q^t(x,\cdot)-\pi(\cdot)||_{TV} \leq \alpha^{t},
	\end{equation} 
	where $||\, . \,||_{TV}$ denotes the total variation norm. Let us denote by $Q$ the restriction of $P^{(k)}$ to the recurrent component of $\mathbf{x}$ and choose $L$ sufficiently large so that $L\alpha^{L} \leq \frac{\epsilon}{2}$. 
	Observe that 
	$$N_{n}(\mathbf{x},x)=\sum_{i=1}^{L} N_{n,i}(\mathbf{x},x)$$ where 
	$$N_{n,i}(\mathbf{x},x):=\sum_{j=0}^{\frac{n}{L}-1} 1(\mathbf{X}_{jL+i}=\mathbf{x}).$$
	Next, by \eqref{e:TV}, for each fixed $i$,  and for $j\geq 1$, the indicator functions $1(\mathbf{X}_{jL+i}=\mathbf{x})$ are stochastically dominated above and below by a family of i.i.d.\ Bernoulli random variables with parameters $\underline{\mathbf{\pi}}(\mathbf{x})+\frac{\epsilon}{2L}$ and $\underline{\mathbf{\pi}}(\mathbf{x})-\frac{\epsilon}{2L}$ respectively. A Chernoff inequality (Theorem \ref{t:chernoff}) now yields that 
	$$\P\left( \bigg\vert \frac{1}{n}N_{n,i}(\mathbf{x},x)-\frac{\underline{\mathbf{\pi}}(\mathbf{x})}{L}\bigg\vert \geq \frac{\epsilon}{L}\right)\leq e^{-c(\epsilon)n/L}.$$
	Summing over all $i$ and  applying the Borel-Cantelli Lemma now yields \eqref{e:toshow} which completes the proof of the proposition.
%
\end{proof}

%
%


\begin{proof}[Proof of Lemma \ref{bt-lemma}]
	Though $\Pi$ is only coarsely Lipschitz and coarsely surjective with  uniformly bounded constants depending
	only on $\delta$, we shall assume below that $\Pi$ is, in fact, surjective and continuous for ease of exposition.
	Let $w \in \sigma$ be such that $\Pi(w) = o$. Consider the $R-$neighborhood $N_R([x,y])$ of $[x,y]$. Let $u$ denote the last point on $\sigma$ {\it before} $w$ where $\sigma$ exits $N_R([x,y])$ so that $d(u, \Pi(u)) = R$. 
	Similarly, let $v$ be the  first point {\it after} $w$ where  
	$\sigma$ re-enters $N_R([x,y])$ so that $d(v, \Pi(v)) = R$. Since 
	$\sigma \cap N_{100R} (o) = \emptyset$, it follows from the triangle inequality that $\Pi(u), \Pi(v) \in [x,y]$ lie on opposite sides
	of $o$ with $d(\Pi(u),o) \geq 99R$,  $d(\Pi(v),o) \geq 99R$.

	We shall show that
	$u,v$ satisfy the properties required by the Lemma.
	For $R$ large enough $[u,\Pi(u)]\cup [\Pi(u),\Pi(v)] \cup [\Pi(v) \cup v]$ (concatenated in this order) forms a $(1,4\delta)-$quasigeodesic (see \cite[Chapter III.H.1]{bh-book} or \cite[Lemma 3.1]{mitra-trees} or Lemma \ref{lem-expdiv} below). Hence there exists $A>0$ (depending only on $\delta$)
	such that $[u,v] \cap N_{A} (o) \neq \emptyset$.

	Using the standard fact that paths leaving large neighborhoods of a geodesic are exponentially inefficient in hyperbolic space (\cite[Chapter III.H.1]{bh-book} or Lemma \ref{lem-expdiv}) we have that there exist $R_0>0, \alpha>0$ depending only on the hyperbolicity constant $\delta$ of $\Gamma$,
	such that 
	$\ell(\sigma_{uv}) \geq d(\Pi(u), \Pi(v))e^{\alpha R}$ for $R \geq R_0$. Item (1) of the Lemma follows immediately. 
	
	Next, $d(o, \sigma_{uv}) \leq  R + d(\Pi(u),\Pi(v))$. Choosing $R_0$ large enough, it follows from the previous paragraph that there exists $C'$ depending only on $\delta$ such that for $R \geq R_0$,
	$d(o, \sigma_{uv}) \leq C' \ell(\sigma_{uv})$, proving Item (2) of the Lemma.
\end{proof}

\section{}
\label{s:set}
In this appendix we discuss an alternative approach to proving Theorem \ref{thm-vexists} using Kingman's subadditive ergodic theorem: for an ergodic transformation $\mathcal{T}$ on a probability space and for any sequence $g_{n}$ of $L^1$ functions satisfying $g_{m+n}(\omega)\leq g_{m}(\omega)+g_{n}(\mathcal{T}\omega)$, $n^{-1}g_{n}$ converges a.e. to a constant. First passage times are naturally subadditive, i.e, $T(u,v)\leq T(u,w)+T(w,v)$ for all $u,v,w$, and this is the standard machinery used to show the existence of velocity for FPP on the Euclidean lattice $\Z^d$ using a simple argument analogous to the one described in Lemma \ref{velatpoles}. For general directions, however, it is not clear if there exists a translation (an action by a group element) along any given geodesic ray which will be translation invariant with respect to the underlying product measure, let alone be ergodic. 

The Calegari-Fujiwara machinery \cite{calegarifujiwara} generating the Patterson-Sullivan measure on the boundary via a Markov chain gives a way to approach Theorem \ref{thm-vexists} via Kingman's theorem which we briefly describe below. The reader would recognize this approach to be a more compact and sophisticated version of the argument presented in Section \ref{sec-vel}. We, however, believe that the more bare hands approach taken by us is more geometrically intuitive and would be more appealing to our intended audience simultaneously consisting of probabilists and geometers. Hence we only present a brief sketch of this argument for the interested reader. 
	
For the simplicity of exposition, let us restrict ourselves to the case where the Markov chain generating the Patterson-Sullivan measure (the Markov chain $N$ from Definition \ref{def-nmu}) is irreducible and aperiodic, i.e., it consists of a single recurrent class. The case where it has multiple components can be dealt with by decomposing it into aperiodic maximal components similar to the argument presented in Section \ref{sec-vel} and we shall not discuss this issue further. 

Recall that the Markov chain trajectories are geodesic rays started at 1 and each Markov chain step corresponds to taking one step along an edge on the geodesic which has an associated passage time. Hence one can consider the joint probability space determined by the Markov chain together with these i.i.d.\ weights. Although the different weights are i.i.d.\ across edges, observe that two different (finite) Markov chain trajectories can lead to the same vertex and hence the weight associated to the next step of the chain will be the same. Therefore the joint chain is not Markov; however it still exhibits a decay of correlation. Consider the transformation $\mathcal{T}$ which denotes a shift of length one along the Markov chain trajectory. If one considers the stationary version of the chain (since we have assumed the chain $N$ is irreducible and aperiodic it has a unique stationary distribution) this transformation is measure preserving. Once one checks the ergodicity of this transformation, one can apply the subadditive Ergodic theorem 
on the sequence of functions $g_{n}(\omega)=T(u_0,u_n)$ where $u_{i}$ denotes position of the chain on the $i$-th step. Notice that in the above argument we have started with a stationary chain whereas in our original description the chain was not necessarily stationary. However, since we have assumed the chain to be ergodic, it converges to the unique stationary distribution and standard techniques may be used to conclude almost sure convergence along the Markov chain trajectories started at 1 as well. An application of Fubini's theorem then would imply that there is an almost sure set of the Markov chain trajectories (i.e., a full measure subset of $\partial G$) such that almost sure convergence of the average passage time happens along these directions. 

We have not attempted to write down the details of this approach here as we believe that our proof from first principles provide a better geometric description of the set of directions along which the convergence holds.

\section*{Acknowledgments}
We thank Danny Calegari for useful email correspondence, particularly for explaining to us one of  the key points of \cite{calegarifujiwara}.
We thank Itai Benjamini for several useful comments on an earlier draft, including directions for future research. Comments of several others have helped improve the quality of the paper. The contents of Section \ref{sec-free3}, Remark \ref{rmk-ref} and Appendix \ref{s:set}, are due, 
at least in part, to comments on an earlier draft.  We thank the referee for a careful reading of the manuscript and for suggesting several corrections and improvements.
RB was partially supported by Ramanujan Fellowship (SB/S2/RJN-097/2017) and a MATRICS grant (MTR/2021/000093) from SERB, Govt.~of India, an ICTS--Simons Junior Faculty Fellowship, and by DAE, Govt.~of India, under project no.~project no. RTI4001 to ICTS. MM was partly supported by a DST JC Bose Fellowship, Matrics research project grant  MTR/2017/000005, CEFIPRA  project No. 5801-1
and by  the Department of Atomic Energy, Government of India, under project no.12-R\&D-TFR-5.01-0500. MM also acknowledges  partial support from the grant 346300 for IMPAN from the Simons Foundation and the matching 2015-2019 Polish MNiSW fund. Both authors were supported in part by an endowment of the Infosys Foundation.
This paper began at the International Centre for Theoretical Sciences (ICTS) during  the program - Probabilistic Methods in Negative Curvature (Code: ICTS/PMNC2019/03) and was completed during a visit to Indian Institute of Science, Bengaluru. We  gratefully acknowledge the hospitality of these institutions.

\bibliography{hypperc}
\bibliographystyle{alpha}

\end{document}